\documentclass[preprint,11pt,authoryear]{elsarticle}

\usepackage[authoryear]{natbib}
  
\usepackage[export]{adjustbox}
\usepackage{float}
\usepackage{amsmath}
\usepackage{amssymb}
\usepackage{geometry}
\usepackage{enumitem}
\usepackage[utf8]{inputenc}
\usepackage[english]{babel}
\usepackage[dvipsnames]{xcolor}
\usepackage{booktabs}
\usepackage[ruled,linesnumbered, noend]{algorithm2e}
\usepackage{bm}
\usepackage{mathtools}
\usepackage{graphicx}
\usepackage{longtable}
\usepackage{subcaption}

\usepackage{pgfplots}
\usepackage[title]{appendix}
\usepackage{multirow}
\usepackage{scalerel}
\usepackage{comment}
\usepackage{multirow}
\usepackage{scalerel}
\usepackage{comment}
\usepackage{booktabs}
\usepackage{svg}
\usepackage{colortbl}
\usepackage{diagbox}

\usepackage{tikz}
\usetikzlibrary{arrows,automata}
\usetikzlibrary{plotmarks}
\usetikzlibrary{shadings,shadows,calc,arrows}
\usetikzlibrary{decorations.pathmorphing,patterns}
\usetikzlibrary{decorations.pathreplacing}
\usetikzlibrary{decorations.markings}
\usetikzlibrary{calc} 
\usetikzlibrary{shapes,snakes}
\usetikzlibrary{positioning}
\usetikzlibrary{shapes.geometric}

\newcommand{\R}{\mathbb{R}}
\newcommand{\Z}{\mathbb{Z}}

\newcommand{\bsubeq}{\begin{subequations}}
	\newcommand{\esubeq}{\end{subequations}}
\newcommand{\BI}{\begin{itemize}}
	\newcommand{\EI}{\end{itemize}}
\newcommand{\I}{\item}
\newcommand{\BE}{\begin{enumerate}}
	\newcommand{\EE}{\end{enumerate}}

\newcommand{\cred}{\color{black}}
\newcommand{\credrev}{\color{black}}

\newcommand{\cblue}{\color{blue}}

\newcommand{\xvar}{x}
\newcommand{\yvar}{y}

\newcommand{\cond}{L_i^\omega(\xvar)}
\newcommand{\indvar}{\mu^\omega}

\newcommand{\remsol}{\bar{Y}^\omega}
\newcommand{\node}{u}
\newcommand{\dualvar}{\delta}
\newcommand{\mpsol}{\hat{\mu}}

\newcommand{\Primal}[1]{\texttt{P#1}(\mpsol^\omega, \omega)}
\newcommand{\AltPrimal}[1]{\overline{\texttt{P#1}}(\mpsol^\omega, \omega)}
\newcommand{\Dual}[1]{\texttt{D#1}(\mpsol^\omega, \omega)}
\newcommand{\AltDual}[1]{\overline{\texttt{D#1}}(\mpsol^\omega, \omega)}

\newcommand{\bdd}[1]{\texttt{BDD#1}}
\newcommand{\vertexset}{\mathcal{V}}
\newcommand{\edgeset}{\mathcal{E}}
\newcommand{\scenset}{\Omega}
\newcommand{\stateset}{\mathcal{S}^\omega}
\newcommand{\termstateset}{\mathcal{T}^\omega}
\newcommand{\state}[1]{s^\omega_{#1}}

\newcommand{\transitionfunc}{\rho}
\newcommand{\neighb}{\mathcal{N}}
\newcommand{\nodes}{\mathcal{N}}
\newcommand{\arcs}{\mathcal{A}}
\newcommand{\zeroarcs}{\mathcal{A}^{\texttt{zero},\omega}}
\newcommand{\onearcs}{\mathcal{A}^{\texttt{one},\omega}}
\newcommand{\nodelayer}[1]{\mathcal{N}^\omega_#1}
\newcommand{\arclayer}[1]{\mathcal{A}^\omega_#1}

\newcommand{\BDD}{\mathcal{B}}
\newcommand{\viol}{\mathcal{I}^\omega}
\newcommand{\MP}{\texttt{MP}}
\newcommand{\SP}{\texttt{SP}}
\newcommand{\LB}{LB}
\newcommand{\UB}{UB}
\newcommand{\cutset}{\mathcal{C}}

\newcommand{\spobj}{\tau}
\newcommand{\map}{\sigma}
\newcommand{\prob}{p}
\newcommand{\Prob}{\mathbb{P}}
\newcommand{\xweights}{c}
\newcommand{\yweights}{d}
\newcommand{\graph}{G}
\newcommand{\rootnode}{\texttt{r}}
\newcommand{\term}{\texttt{t}}

\newcommand{\cvar}{CVaR_\alpha}
\newcommand{\var}{VaR_\alpha}
\newcommand{\expect}{\mathbb{E}}
\newcommand{\indicator}{\mathbb{I}}

\newcommand{\binxvar}{x^\texttt{B}}
\newcommand{\binindex}{n_\binxvar}
\newcommand{\intxvar}{x^\texttt{I}}
\newcommand{\intindex}{n_\intxvar}
\newcommand{\contxvar}{x^\texttt{C}}
\newcommand{\contindex}{n_\contxvar}

\newcommand{\yvarindex}{n_\texttt{\yvar}}

\newcommand{\costindex}{q(i)}
\newcommand{\costone}[1]{d^1_{#1}(\omega)}
\newcommand{\costtwo}[1]{d^2_{#1}(\omega)}
\newcommand{\secondstagecost}[1]{d_{#1}(\omega)}
\newcommand{\betamax}{\hat{\beta}^{\max}}
\newcommand{\zmax}{\hat{z}^{\max}}

\graphicspath{ {images/} }

\usepackage[colorlinks=true, linkcolor=blue, citecolor=blue]{hyperref} 

\newproof{pf}{Proof}
\newproof{proof}{Proof}
\newtheorem{example}{Example}
\newtheorem{proposition}{Proposition}

\begin{document}
	
\begin{frontmatter}

	\title{Leveraging Decision Diagrams to Solve Two-stage Stochastic Programs with Binary Recourse and Logical Linking Constraints}

	\author[a]{Moira MacNeil%
		\corref{cor1}} 
	\ead{m.macneil@mail.utoronto.ca}

\author[b]{Merve Bodur}
\ead{merve.bodur@ed.ac.uk}

	\cortext[cor1]{Corresponding author}

	\affiliation[a]{organization={Department of  Mechanical and Industrial Engineering, University of Toronto},
		city={Toronto},
		state={Ontario},
		country={Canada}}
		
			\affiliation[b]{organization={School of Mathematics, The University of Edinburgh},
			city={Edinburgh},
			country={United Kingdom}}

	\begin{abstract}
		We generalize an existing binary decision diagram-based (BDD-based) approach of Lozano and Smith (MP, 2022) to solve a special class of two-stage stochastic programs (2SPs) with binary recourse, where the first-stage decisions impact the second-stage constraints. {\cred First,} we extend the second-stage {\cred problem} to a more general setting where logical expressions of the first-stage solutions enforce constraints in the second stage. {\cred Then, as} our primary contribution, {\cred we introduce} a complementary problem, {\cred that appears more naturally for many of the same applications of the former approach}, and {a distinct BDD-based} solution method, {\cred that is more efficient than the existing BDD-based approach on commonly applicable problem classes}. In the novel problem, second-stage costs, {\cred rather than constraints,} are impacted by expressions of the first-stage decisions. 
		In both settings, we convexify the second-stage problems using BDDs and parametrize either the BDD arc costs or capacities with first-stage solutions. 
		We extend this work by incorporating conditional value-at-risk and propose the first decomposition method for 2SP with binary 
		recourse and a risk measure. We apply these methods to a novel stochastic problem, {\cred namely stochastic minimum dominating set problem,} and 
		present numerical results to support their effectiveness.
	\end{abstract}

%

\end{frontmatter}

 \section{Introduction}

Decision-making problems under uncertainty are often challenging to model and solve, especially when the decisions must be modelled as binary ``yes-or-no'' choices such as in the knapsack or facility location problems. One approach to address these problems is two-stage stochastic programming.  This modelling framework divides decision making into two stages: first we make immediate ``here-and-now'' decisions without the knowledge of the value of uncertain parameters, and second, after observing the outcome of the random parameters we make the remaining recourse decisions. Typically, the objective of a two-stage stochastic program (2SP) is to minimize the total cost given as the first-stage cost and the expected second-stage cost. We will study a class of 2SPs with binary recourse and specific structure that links first-stage and second-stage decisions and their costs.  \cite{lozano2018bdd} first introduced binary decision diagram-based (BDD-based) methods to solve such a problem class. {\cred A decision diagram is a directed acyclic graph which encodes feasible solutions of an optimization problem (traditionally for deterministic problems) as paths from its root node to the terminal node with path length corresponding to the objective value of the associated feasible solution, as such reduces a minimization problem into a shortest path problem.}

In this paper we generalize the family of problems studied by Lozano and Smith and extend their BDD-based methodology allowing it to be employed in a variety of application types. Moreover, as our primary contribution, we propose an alternative problem class that appears {\cred (more)} naturally in many of the same applications. We develop a distinct BDD-based algorithm to solve the novel problem class and show via computation on an application problem that this algorithm is more efficient than that of Lozano and Smith.

In the remainder of this section we will introduce the original problem of Lozano and Smith, followed by our generalization. We will then propose our novel alternative to that problem and give an array of applications for these problem classes.
We will then summarize the contributions of this work and give the paper outline.

\subsection{The Problem of  \cite{lozano2018bdd}} \label{subsec:ls_form}

Given $n = \binindex + \intindex + \contindex$ first-stage decision variables $\xvar$, and  real-valued objective and constraint coefficients, the two-stage stochastic programming problem considered by \cite{lozano2018bdd} is:
\bsubeq \label{prob:ls} 
\begin{alignat}{3}
{\textsc{Problem 1}}:	\min \ & {\xweights^\top \xvar} + \expect[\mathcal{Q}({\xvar}, \omega)]\\
	\text{s.t.} \ &  {\xvar}= (\binxvar, \intxvar, \contxvar) \in \{ x:	A\xvar \geq b, \xvar  \in\{0,1\}^{\binindex}  \times \Z^{\intindex} \times \R^{\contindex}  \}\\*[0.15cm]
%
&\hspace*{-4cm} \text{where the value (recourse) function is} \nonumber \\*[0.15cm]
	\mathcal{Q}({\xvar}, \omega) = 	\min \ & \secondstagecost{}^{\top} {\yvar} \\
\text{s.t.} \ &  \binxvar_i=0  \implies {\yvar} \in \mathcal{W}_i(\omega) & \hspace*{-3.5cm} \forall i =1,\hdots, \binindex \label{eq:ls_logical2} \\
&{\yvar} \in \mathcal{Y}(\omega) \subseteq  \{0,1\}^{\yvarindex} \label{eq:sls_yconstr}
\end{alignat}%
\esubeq 
and $\omega$ is a random vector defined on the probability space $(\Phi, \mathcal{F}, \mathcal{P})$.
The logical constraints \eqref{eq:ls_logical2} of the second-stage problem 
enforce constraints $\mathcal{W}_i(\omega)$ on the recourse variables $\yvar$  only if $\binxvar_i =0$. This amounts to the selection of a \textit{single} first-stage variable $\binxvar_i$ making the constraints in $\mathcal{W}_i(\omega)$ redundant. 

The key insight of Lozano and Smith is that this second-stage problem can be reformulated using  BDDs where the arc capacities are parameterized by the first-stage binary variable solutions, which in turn makes the problem amenable to be solved via the Benders decomposition algorithm. They use this so-called BDD-based reformulation to create a single BDD equivalent to each second-stage scenario problem, which can then be passed a  first-stage candidate solution and solved efficiently via a shortest path algorithm. The shortest-path solution is then used to generate Benders cuts. 
Broadly, the algorithms for solving 2SP outlined in this paper use the same mechanisms: {\cred reformulate the second-stage model as a BDD that incorporates some parameterization of the first-stage solutions, then iteratively solve the first-stage problem for a candidate solution, utilize it to find the shortest path of the BDD and derive a cut that will be used to generate an improved candidate solution in the first stage.} However, the details of the reformulations proposed in this work are quite distinct and can lead to very different computational results.

\subsection{A Generalization of Problem \ref{prob:ls}}
We will consider a generalization of Problem \ref{prob:ls} where the left-hand-side of the implication in the constraints \eqref{eq:ls_logical2} is now a \emph{logical expression} of the first-stage binary variables $\binxvar$. We also present the first-stage problem in a more general setting where the first-stage decisions are constrained to the set  $\mathcal{X}$ which does not have any restrictions.
Most notably, the first-stage problem does not need to be linear, which broadens the possibilities for applications. We will use a similar procedure to that in \citep{lozano2018bdd} to reformulate \eqref{prob:constrs} via a BDD that has  \emph{arc capacities} parameterized by logical functions of first-stage solutions.
\bsubeq\label{prob:constrs}
\begin{alignat}{2}
	{\textsc{Problem 2}}:	\min  \ & {\xweights^\top \xvar} + \expect[\mathcal{Q}({\xvar}, \omega)]\\
	\text{s.t.} \ & {\xvar} = (\binxvar, \intxvar, \contxvar) \in \mathcal{X} \subseteq  \{0,1\}^{\binindex} \times \Z^{\intindex} \times \R^{\contindex}  \\*[0.15cm]
&\hspace*{-3.4cm} \text{where} \nonumber \\*[0.15cm]
	\mathcal{Q}({\xvar}, \omega) = 	\min \ & \secondstagecost{}^{\top} {\yvar} \label{eq:recourse2} \\
	\text{s.t.} \ &  \indicator(\cond) = 1 \implies {\yvar} \in \mathcal{W}_i(\omega) & \quad \forall i =1,\hdots,m_1   \label{eq:sp2_logical2} \\
	&{\yvar} \in \mathcal{Y}(\omega) \subseteq  \{0,1\}^{\yvarindex}. \label{eq:sp2_yconstr} 
\end{alignat}%
\esubeq 
In the new logical constraints \eqref{eq:sp2_logical2}, the function $\cond$ is a logical expression on the first-stage variables, and $\indicator(\cdot)$ is an indicator function such that $\indicator(\cond)=1$ if the expression $\cond$ evaluates to true. These constraints require that for a fixed $i$, when $\cond$ is true, the constraint set in $\mathcal{W}_i(\omega)$ is enforced for ${\yvar}$. 
This logical expression is scenario-dependent as there are cases where $\cond$ may be true in one scenario but not another.  

{\cred We note that in the BDD-based algorithm, explained in Section \ref{sec:benders_alg}, an indicator variable for the logical expression will be used to capacitate some of the BDD arcs, limiting the feasible second-stage solutions set under certain first-stage solutions. However, in many of the applications we consider there is no need for a separate indicator expression as we can directly use an expression of the first-stage variables. This is the strength of the generalized problem, in theory we can use any logical expression of the first-stage binary variables, including those consisting of multiple variables as well as nonlinear ones, but in practice many combinatorial applications do not require the use of an extra indicator variable to model such expressions.}
We {\cred also} remark that whereas in Problem \ref{prob:ls} the constraint set is enforced when the first-stage binary variable $\binxvar_i =0$, in our generalization we enforce constraints when the logical expression is true, that is, the indicator returns value $1$. This change was to remain consistent with Problem \ref{prob:costs} in Section \ref{subsec:alternative}. 


\subsection{An Alternative to Problem  \ref{prob:constrs}} \label{subsec:alternative}

One of the limitations of the BDD-based reformulation for the second-stage problem of  \eqref{prob:constrs} is that the BDDs are less likely to contain isomorphic subgraphs. If the arc capacities of an otherwise isomorphic subgraph are not the same, that subgraph cannot be used to reduce the size of the BDD. Interested readers are referred to \citep{bergman2016decision, wegener2000} for more information on reducing BDDs. 
Since the reformulation of  \eqref{prob:constrs}  relies on parameterizing the arc capacities using logical functions of first-stage solutions, the resulting (reduced) BDDs are larger than if the arc capacities were all unary, as such finding the shortest path can be computationally expensive. Motivated by the search for a more computationally efficient solution approach, we present an alternative to  Problem \ref{prob:constrs} that we will also reformulate using BDDs  but whose reformulation 
instead manipulates the \emph{arc costs}. {\cred In this formulation all the arcs of the BDD have the same capacity, which will result in smaller reduced BDDs.}
\bsubeq \label{prob:costs}
\begin{alignat}{2}
{\textsc{Problem 3}}:	\min \ & {\xweights^\top \xvar} + \expect[\mathcal{Q}({\xvar}, \omega)]\\
		\text{s.t.} \ & {\xvar} = (\binxvar, \intxvar, \contxvar) \in \mathcal{X} \subseteq  \{0,1\}^{\binindex} \times \Z^{\intindex} \times \R^{\contindex}  \\*[0.15cm]
&\hspace*{-3.34cm} \text{where} \nonumber \\*[0.15cm]
	\mathcal{Q}({\xvar}, \omega) = 	\min \ & (\costone{} + \costtwo{})^\top {\yvar} \label{eq:altprob_obj} \\
	\text{s.t.} \ 	&  \indicator\left(\cond\right) = 1 \implies \costone{\costindex} =   0 & \quad \forall i =1,\hdots,m_1 \label{eq:sp1_logical}\\
	&  {\yvar} \in \mathcal{Y}(\omega) \subseteq   \{0,1\}^{\yvarindex}. \label{eq:sp1_yconstr}
\end{alignat}%
\esubeq %
The first-stage problem is exactly that of Problem \ref{prob:constrs}. 
The second-stage problem minimizes the cost \eqref{eq:altprob_obj} which is comprised of two parts, where we require {\cred only} $\costone{} \geq 0$. 
The logical constraints \eqref{eq:sp1_logical} enforce that if $\cond$ is true, the cost of decision $\yvar_{\costindex}$ is reduced to  $\costtwo{\costindex}$ which can be $0$. {\cred  Here we use $\costindex$ as the index of the second-stage cost since the $i^{\text{th}}$ logical expression may not affect  the $i^{\text{th}}$ cost coefficient.}
The BDD-based reformulation for Problem \ref{prob:costs} is not straightforward as it requires some manipulation to be able to derive Benders cuts. This problem structure is unconventional when compared to the literature but lends itself well to many applications.

{\cred We remark that Problems \ref{prob:constrs} and \ref{prob:costs} are not true alternatives to each other as they do not have the same feasible region. What we mean is that given a fixed application we can model it as  Problem \ref{prob:constrs} or as Problem \ref{prob:costs}.  We will next detail some applications which can be modelled and solved using either method. We stop short of a formal proof of this equivalence but we have not yet found an application that can be modelled as one of these problems but not the other. 
	}

\subsection{Applications of Problems  \ref{prob:constrs} and \ref{prob:costs}} \label{subsec:applications}

Generally, in applications for problems with the structure of Problem \ref{prob:constrs} a set of constraints can either be satisfied in the first stage or in the second stage. 
Whereas, applications for problems with the structure of Problem \ref{prob:costs} appear when the first-stage decisions allow for investment in an item and in the second stage the item can be used at a reduced cost, typically for free. 
All the Problem \ref{prob:costs} instances that we have encountered can also be formulated as a Problem  \ref{prob:constrs} instance and vice versa but  the original problem structure may naturally lend itself to one over the other. 
As we will see in Section \ref{sec:recourse_reform}, the BDD-based representations may also favour one problem over the other. In this section we will overview a wide variety of applications where models of the form of Problems \ref{prob:constrs} and \ref{prob:costs} arise, a summary can be found in Table \ref{table:applications}.
Note that although we do not explicitly state it for each application, the constraint set in $\mathcal{Y}(\omega)$ typically differs between the two recourse problems.

\subsubsection{Stochastic shortest path}

In the stochastic shortest path (SSP) problem of \citep{ravi2006uncert}, the source is known but the destination vertex is uncertain. In the first stage, a subset of edges is selected, then the destination is observed, and in the second stage edges are selected to complete the path from source to destination.  \cite{ravi2006uncert} introduce a variation on the problem where the first-stage solution must form a tree and the edge costs are also stochastic while still forming a metric, which they call the tree-stochastic metric shortest paths (TSMSP) problem. 

Consider a formulation of the TSMSP problem where the variable $x_e =1$ corresponds to selecting edge $e$ in the first-stage tree. Then the edge costs and destination are observed and the remaining edges are selected where the variable $y_e= 1$ corresponds to selecting an edge in the second stage. 
Given a vertex set $\vertexset$ and a subset $U \subset \vertexset$, let $\delta(U)$ be the set of edges with exactly one endpoint in $U$. Denoting the source node $s$ and the destination node in a given scenario $t$, the logical constraints in Problem 2 enforce 
for each $U  \subset \vertexset$ with $t \in U$ and $s \notin U$ 
$\sum_{e \in \delta(U)} x_{e} = 0 \implies \sum_{e \in \delta(U)} y_{e} \geq 1$. 
Denoting the second-stage cost of an edge as $d_e$, the logical constraints in the alternative Problem 3 for each edge $e$ are
$x_e = 1 \implies d_e = 0$. 

\subsubsection{Stochastic facility location and assignment}

Consider the uncapacitated facility location and assignment problem with a stochastic client set and stochastic facility opening costs of \citep{ravi2006uncert}, where in the first stage we can select some facilities in the facility set $F$ to be opened, then observe the client set $D$ and the costs $c$, then assign clients to facilities and take the recourse action of opening more facilities if required to serve all the clients.

Let variable $x_i =1$ if facility $i \in F$ is opened in the first stage and $y_i =1$ if it is opened in the second stage, and let $z_{ij} =1 $ if client $j \in D$ is assigned to facility $i \in F$. Then, the logical constraints in Problem 2 are, for each $i \in F$ and $j \in D$,
$x_i = 0 \implies z_{ij} \leq y_i$.
These constraints enforce if a facility is not opened in the first stage we must enforce that no client can be assigned to it unless it is opened in the second stage.
In Problem 3,  the logical constraints for each facility $i \in F$ are
$ x_i = 1 \implies c_i = 0$, that is, if a facility is opened in the first stage, it can be used in the second stage at no extra cost.

\subsubsection{Stochastic knapsack}
We can also apply our problem setting to stochastic knapsack problems with uncertain profit, similar to the robust version formulated by  \cite{arslan2021decomposition}. In this problem, we must first commit to procuring (producing or outsourcing), a subset of a set of items, $I$, with stochastic costs. These items have expected profit $p^0_i$, weight $c^0_i$, and are subject to a first-stage knapsack constraint with capacity $C^0$.
We observe a profit degradation $d_i$, and then have three recourse options for each item.
Firstly, we can choose to produce the item at the reduced profit $p_i = p^0_i - d_i$ and consume second-stage knapsack capacity $c_i$. 
Secondly, we can repair the item consuming some extra second-stage knapsack capacity, $c_i + t_i$, but recovering the original expected profit of $p_i^0$, or finally, we can outsource the item so that the actual profit of the outsourced item is $p^0_i - f_i$. The items produced or repaired in the second stage are subject to a knapsack constraint with capacity $C$.

Let $x_i = 1$ if we select item $i$ to procure in the first stage, and in the second stage let $y_i = 1$ if we select item $i$ to produce ourselves and let $r_i = 1$ if we choose to repair item $i$. 

%
In Problem 2 the logical constraints are 
$ x_i = 0 \implies y_i \leq 0$ for each item $i$, since if item $i$ is not selected for production in the first stage it can not be produced in the second stage.
In Problem 3,  the logical constraints for each item $i$ are
$ x_i = 0 \implies f_i =0 ${\credrev, which yields the objective function of $- d_i y_i + d_i r_i $ in the second-stage problem. Along with the constraints $r_i \leq y_i$ and the fact that both $y_i$ and $r_i$ consume second-stage knapsack capacity, this results in an optimal second-stage solution with $y_i = r_i = 0$. Accordingly, it ensures that}
not selecting item $i$ for procurement in the first stage means it cannot be produced nor outsourced in the second stage and will not result in any profit. The full stochastic knapsack problem model can be found in \ref{app:stochknap}.

\subsubsection{Capital budgeting under uncertainty}

We consider a stochastic investment planning problem similar to the robust optimization version of \citep{arslan2021decomposition}. A company has allocated a budget to invest in projects with uncertain profit now or to wait and invest in them later. The early investments are incentivized by higher profits than if the company waits, where the reduced second-stage profit is denoted $p_i$ for project $i$. 
Let $x_i = 1 $ if a project $i$ is selected in the first stage and let $y_i = 1$ if it is selected in the second stage.
In Problem 2 the logical constraints  for each project $i$ are 
$x_i = 1 \implies y_i \leq 0$ that is, if a project is selected in the first stage it cannot be selected in the second stage.
In Problem 3,  the logical constraints project $i$ are
$x_i = 1 \implies p_i =0 $ as a project selected in the first stage will not gain profits in the second stage.

\subsubsection{Stochastic vertex cover}

Consider a stochastic vertex cover problem, such as in \cite{ravi2006uncert}, on a graph $\graph = (\vertexset, \edgeset_0)$. In the first stage, a set of vertices is selected, then the scenario set of edges $\edgeset$ and the second-stage vertex costs $c_v$ are observed, notice that  $\edgeset$ may or may not be a subset of  $\edgeset_0$.   Finally, in the second stage, the vertex set is extended to a vertex cover of the edges in $\edgeset$.
Let $x_v = 1$ if vertex $v \in \vertexset$ is selected in the first stage, and let $y_v = 1$ if vertex $v \in \vertexset$ is selected in the second stage. In Problem 2 the logical constraints are for each $\{u,v\}  \in \edgeset\cap \edgeset_0$ {\cred in the form of} 
$ x_v + x_u = 0 \implies y_v+y_u \geq 1$,
that is if an edge that appears in both stages is not covered in the first stage, it must be covered in the second stage.
In Problem 3,  the logical constraints for each $v \in \vertexset$ are
$ x_v = 1 \implies c_v = 0$, that is, if a vertex $v$ is selected in the first stage, there is no additional cost to cover the edges incident to that vertex in the second stage.

\subsubsection{Submodular optimization} \label{subsubsec:submod}
In their recent work \cite{Coniglio2020submodular} study problems which maximize a concave, strictly increasing, and differentiable utility function over a discrete set of scenarios. The decision maker must choose from two sets of items, the first is a set of meta items $\hat{N}$ that is linked to the set of items $N$ via a covering relationship, i.e., an item $j \in N$ can only be selected if a meta item $\ell \in \hat{N}$ covering it is also selected.
The objective function is a linear combination of submodular set functions and thus is submodular itself. These problems arise in social network applications, marketing problems and stochastic facility location settings \citep{Coniglio2020submodular, kempe2003max}.


In our setting, we have first-stage decisions $x_\ell = 1 $ if meta item $\ell \in \hat{N}$ is selected and the second-stage decisions $y_j=1$ if item $j \in N$ is selected. Given a stochastic non-negative item rewards $a_j$ and stochastic  non-negative  constant $d$, the second-stage objective function for a fixed scenario is 
\begin{equation} \label{eq:submod}
	f\left( \sum_{j \in N} a_jy_j + d\right),
\end{equation}
where we remark that this function is nonlinear. 
Here we have a pure binary optimization problem, with a nonlinear objective and linear constraints. As we are in the stochastic setting, we will take the expectation of \eqref{eq:submod} over a set of scenarios for $a_j$ and $d$, each of which occur with a probability $p_i$. Therefore we are able to decompose the second-stage objective function into a sum of nonlinear functions, one function \eqref{eq:submod} for each scenario. 
This is precisely the setting of \cite{bergman2018nonlin}, who solve a binary optimization problem with linear constraints and a nonlinear objective function that can be written as the sum of auxiliary nonlinear functions. Their work models each component of the objective as a set function and then uses dynamic programming  to create a decision diagram for each component.  In our setting, we can model each scenario in the second stage as a decision diagram.
In our setting the logical constraints in Problem 2 are, for each item $j$,
$ \sum_{\ell \in \hat{N}(j)} x_\ell = 0  \implies y_j = 0,$
where we represent the set of meta items that cover item $j$ by $\hat{N}(j)$. Thus if no meta item covering item $j$ is selected in the first stage we cannot select item $j$ in the second stage.
In Problem 3,  the logical constraints for each item $j$ are
$\sum_{\ell \in \hat{N}(j)} x_\ell = 0  \implies a_{j} = 0$.
That is, if no meta item covering item $j$ is selected in the first stage, no profit can be gained from $j$ in the second stage.

\subsection{Stochastic minimum weight dominating set} \label{subsubsec:smwds}
As a final example of the studied structures, we propose a novel stochastic dominating set problem that we call the stochastic minimum weight dominating set (SMWDS). We present further details of the problem and numerical results in Section \ref{sec:smwds} and will use it as a running example throughout this paper.
 
In the design of wireless sensor networks it is often useful to partition the network so that a subset of sensors are more powerful than the others and are able coordinate the rest of the network \citep{santos2009}. Ideally, these powerful sensors will be linked to every other sensor in the network which amounts to finding a dominating set in the wireless sensor network \citep{boria2018}. 
We often wish to identify the Minimum Weight Dominating Set (MWDS), where node weights represent mobility characteristics or characteristics such as the residual energy or the available bandwidth \citep{bouamama2016hybrid}. Other applications of MWDS include disease suppression and control \citep{PhysRevE.90.012807}, the study of public goods in networks \citep{bramoulle2007public}, social network analysis \citep{lin2018}, and routing in wireless networks \citep{Li2001routing}.

In many applications, the network structure is uncertain or time dependent, thus a deterministic setting is not a realistic approximation of the problem, however work in the stochastic setting has been limited. 
In \citep{boria2018}, a graph with an existing MWDS is given as input and the nodes fail with some given probability, the goal is then to find a new dominating set which includes any functional sensors from the original solution. 
The authors propose solution algorithms for cases where the problem is polynomial and polynomial approximations for the general case, and give an IP formulation for the problem but do not implement it. In \citep{torkestani2012}, the graph is fixed but the node weights are stochastic with an unknown distribution, and the objective is to find a connected dominating set. 
The authors present a series of learning automata-based algorithms and show they reduce the number of samples needed to construct a solution. \cite{He2011genetic} present a variant where the edges can transmit signals with a given probability and the objective is to find a connected dominating set with a certain reliability. 

In the SMWDS problem we assume the existence of a graph where the vertex weights are known and no failure has been observed yet, 
we then observe changes in the vertex weights and  vertex failures. 
To the best of our knowledge, we present the first study of the probabilistic MWDS problem on stochastic graphs. 
Given a graph $\graph = (\vertexset, \edgeset)$, for each $v \in \vertexset$, let $\xweights_v > 0$ be the first-stage vertex weights and let $\secondstagecost{v} \geq 0$ be the second-stage vertex-weight random variables. Let binary decision variables $\xvar_v = 1 $ if the vertex $v \in \vertexset$ is assigned to the dominating set in the first stage, and $\yvar_v = 1 $ if the vertex $v$ is assigned to the dominating set in the second stage under equally probable scenario $\omega \in \scenset$. We also define the neighbourhood of a vertex in the first stage $\neighb(v) = \{u \in \vertexset : \{u,v\} \in \edgeset\}$,  and in the second stage $\neighb_\omega'(v) = \{u \in \vertexset : \{u,v\}\in \edgeset, \secondstagecost{u} \neq 0\}$ where we only consider in the neighbourhood vertices that have not failed, similarly for the second-stage vertex set $\vertexset_\omega' = \{v \in \vertexset : \secondstagecost{v} \neq 0\}$. 
Given this formulation, the logical constraints in Problem 2 for each $v \in \vertexset'$ are
$
\xvar_v +   \sum_{u \in \neighb'(v)}  \xvar_u  =0 \implies \yvar_v + \sum_{u \in \neighb'(v)} \yvar_u  \geq 1.
$
They enforce if neither $v$ nor one of its second-stage neighbours is selected in the first stage, one of those vertices must be selected in the second stage. That is, if $v$ is  not dominated in the first stage, it must be dominated in the second stage.
In Problem \ref{prob:costs},  the logical constraints for each $v \in \vertexset'$ are
$
\xvar_v = 1 \implies  \yweights_v = 0, 
$
i.e., if vertex $v$ is selected in the first stage, it can be used to dominate second-stage vertices at no additional cost.

%

%
%
%

\begin{table}[h] \caption{\label{table:applications} Applications of Problems  \ref{prob:constrs} and \ref{prob:costs}.}
\small
	\centering
	\scalebox{0.965}{
	\begin{tabular}{l l l r}
		\toprule
		Stochastic   &Problem \ref{prob:constrs} &  Problem  \ref{prob:costs} & References\\
				 Application  &Logical Constr.&  Logical Constr. & \\
		\midrule
		 Shortest path &$\sum_{e \in \delta(U)} x_{e} = 0 $&$x_e = 1 \Rightarrow  d_e = 0$& \citep{ravi2006uncert}\\
		 & \quad$\Rightarrow \sum_{e \in \delta(U)} y_{e} \geq 1$&& \\
		 		\midrule
		 Facility location  &$x_i = 0 \Rightarrow z_{ij} \leq y_i$&$ x_i = 1 \Rightarrow c_i = 0$& \citep{ravi2006uncert}\\
		  \quad \& assignment &&& \\
		  		\midrule
		  		Vertex cover &$x_v + x_u = 0 $&$ x_v = 1 \Rightarrow c_v = 0$& \citep{ravi2006uncert}\\
		  		&\quad$\Rightarrow y_v+y_u \geq 1$&&\\
		  				\midrule
		 Knapsack &$x_i = 0 \Rightarrow y_i \leq 0 $&$ x_i = 0 \Rightarrow f_i =0$& \citep{arslan2021decomposition}\\
		 		\midrule
		Capital budgeting &$x_i = 1 \Rightarrow y_i \leq 0$&$x_i = 1 \Rightarrow p_i =0 $& \citep{arslan2021decomposition}\\
		\midrule
		Submodular  & $ \sum_{\ell \in \hat{N}(j)} x_\ell = 0$&$\sum_{\ell \in \hat{N}(j)} x_\ell = 0 $& \citep{Coniglio2020submodular},\\
			\quad optimization&\quad$  \Rightarrow y_j = 0$& \quad $ \Rightarrow a_{j} = 0$& \citep{bergman2018nonlin}\\
					\midrule
			Minimum weight  &$\xvar_v +   \sum_{u \in \neighb'(v)}  \xvar_u  =0 $&$\xvar_v = 1 \Rightarrow  \yweights_v = 0$& Section \ref{sec:smwds}\\
		dominating set &\quad $\Rightarrow \yvar_v + \sum_{u \in \neighb'(v)} \yvar_u  \geq 1$\quad&& \\
		\bottomrule
	\end{tabular}
	}
\end{table}

\subsection{Contributions and Outline}

In the remainder of this paper we will present the BDD-based reformulations for the recourse structures in  Problems  \ref{prob:constrs} and \ref{prob:costs} and incorporate them into a Benders decomposition algorithm. We remark that the reformulation is non-trivial for Problem  \ref{prob:costs}. To demonstrate the strength of both formulations we will analyze the strength of the Benders cuts derived from the BDDs as compared to each other and to the integer L-shaped cuts. We will also extend these problems to the risk averse setting, by showing how to incorporate conditional value-at-risk (CVaR) into the models. Finally, we will examine a novel application, the SMWDS problem, and present computational results.

The major contributions of this study 
are:
\BI     \setlength{\itemsep}{0.25pt}
\I Proposing a novel BDD-based method to solve 2SP with binary recourse, this approach  parametrizes BDD solutions using arc costs. We remark that these ideas have potential to be extended to other problems where a convex representation of second-stage problems is difficult to achieve, for example, 2SP with integer recourse or non-linear second-stage problems.  
\I {\cred Extending} the related work of \cite{lozano2018bdd} to solve similar problems, which parametrizes BDD solutions using arc capacities. The generalization allows the work of these authors to be applied to  problems where the second-stage constraints are not explicitly linked to a single first-stage variable. 
\I Extending our findings to a risk-averse setting, which, to our knowledge, is the first decomposition method for 2SP with binary recourse and using CVaR.
\I Proposing the novel SMWDS problem and applying our methodologies to it.
\EI 

The rest of the paper is organized as follows. In Section \ref{sec:litreview}, we review the related BDD and stochastic programming literature. In Section \ref{sec:benders_alg}, we give an overview of the Benders decomposition framework. In Section \ref{sec:recourse_reform}, we reformulate the recourse problems of \eqref{prob:constrs} and  \eqref{prob:costs} using BDDs. Next, in Section \ref{sec:benders}, we describe how to derive cuts from the BDDs that will be used in the Benders decomposition algorithms and then compare their strength.
In Section \ref{sec:cvar} we extend the results to a risk-averse setting. In Section \ref{sec:smwds} we present computational results on the novel SMWDS problem. Finally, in Section \ref{sec:concl} we conclude the paper.
We note that for concision all the proofs are provided in \ref{app:proofs}.

\section{Literature Review} \label{sec:litreview}

\textit{Decomposition algorithms} are a cornerstone of the methods for solving 2SP.  
 \cite{LAPORTE1993} proposed the first extension of traditional Benders decomposition (i.e., the L-shaped method) to the integer setting, which they called the integer L-shaped method. They consider binary first-stage variables and complete recourse, and follow a standard branch-and-cut procedure to iteratively outer approximate second-stage value function. The integer L-shaped method is a special case of the logic-based Benders decomposition first introduced by  \cite{HookerLBBD}.
We will compare our BDD-based methods to the integer L-shaped method for 2SP with binary recourse to derive logic-based Benders cuts.
\cite{caroe1998} generalize the integer L-shaped method by proposing a Benders decomposition algorithm for 2SP with continuous first-stage variables and integer recourse. They derive feasibility and optimality cuts via general duality theory which leads to a nonlinear master problem. 
 \cite{Sherali2002} use a reformulation-linearization technique or lift-and-project cutting planes to generate a partial description of the convex hull of the feasible recourse solutions for problems where the variables of both stages are binary.  \cite{Sherali2006} also use a reformulation-linearization technique or lift-and-project cutting planes in a similar way
but for problems with mixed-integer variables in the first and second stage.
 \cite{ahmed2004finite} propose a finitely terminating algorithm for 2SP with  mixed-integer first stage and pure-integer second-stage variables. Their algorithm exploits structural properties of the value function to avoid explicit enumeration of the search space. 
 \cite{Zhang2014} consider 2SP with pure-integer variables in both stages, they propose a Benders decomposition algorithm that utilizes parametric Gomory cuts to iteratively approximate the second-stage problems. 
 \cite{Qi2017} propose an ancestral Benders cutting plane algorithm for 2SP with mixed-integer variables in both stages and where the integer variables are bounded. 
 \cite{li2019nlp} present a Benders decomposition-based algorithm for nonlinear 2SP with mixed-binary variables in both stages. 
Most recently, \cite{van2020converging} consider 2SP with mixed-integer recourse, they propose a recursive scheme to update the outer approximation of the recourse problem and to derive a new family of optimality cuts called scaled cuts.

The literature also contains many decomposition algorithms developed for the \textit{special case} of binary first-stage variables and integer second-stage variables 
\citep{angulo2016, Gade2014,Ntaimo2010disjunctive, Ntaimo2013Fenchel,  Ntaimo2007, Ntaimo2008, Sen2005, Sen2006}.
A wide array of other approaches to solve stochastic integer programs (IPs) exactly have been proposed including Lagranian duality-based methods \citep{caroe1999dual}, enumeration algorithms \citep{schultz1998solving}, branch-and-price schemes \citep{lulli2004branch}. Most recently neural networks and supervised learning have been used to approximate the second-stage value function \citep{dumochelle2022neur2sp}. 

A natural extension of stochastic programming models is to incorporate \textit{risk measures}, especially in fields such as finance and natural disaster planning. However, the literature contains few works which incorporate risk measures into 2SP with integer second-stage variables.
 \cite{schultz2006} consider 2SP with CVaR and integer variables in the recourse problem. Their solution method uses Lagrangian relaxation of nonanticipativity to convexify the second-stage problems. 
 \cite{Noyan2012} proposes the first Benders decomposition-based algorithms for 2SP with continuous second-stage problems and the CVaR risk measure, and apply these algorithms to disaster management.
In their work, \cite{vanbeesten2020} consider mixed-integer recourse problems with CVaR, where they derive convex approximation
models and corresponding error bounds to approximate the second-stage problem. 
  \cite{arslan2021decomposition} study robust optimization problems where the second-stage decisions are mixed-binary. They propose a relaxation that can be solved via a branch-and-price algorithm, and give conditions under which this relaxation is exact. We will show that CVaR can be incorporated into problems with binary recourse and solved via decomposition using either BDDs or integer L-shaped cuts to convexify the second-stage problems.

We refer readers interested in background on the use of decision diagrams (DDs) in optimization to \citep{bergman2016decision} and \citep{Castro2022}. Here we will focus on the literature at the \textit{intersection of DDs and stochastic optimization}, which has been limited to date.
 \cite{haus2017scenario} consider endogenous uncertainty where decisions can impact scenario probabilities. They characterize these probabilities via a set of BDDs to obtain a mixed-integer programming (MIP) reformulation. 
As discussed in Section \ref{subsec:ls_form}, \cite{lozano2018bdd} consider 2SP with pure-binary second-stage variables, they consider linear first-stage constraints and generic second-stage constraints. In the class of problems they study, selection of a first-stage binary variable makes a set of second-stage constraints redundant. They model the second-stage problems via BDDs to derive Benders cuts and demonstrate their techniques on stochastic traveling salesman problems.
 \cite{serra2019lastmile} propose two-stage stochastic programming for a scheduling problem, where they model both stages via multivalued DDs and link the two stages through assignment constraints to obtain an IP.  \cite{castro2021combinatorial} propose a BDD-based combinatorial cut-and-lift procedure which they apply to a class of pure-binary chance-constrained problems. 
 Most recently,  \cite{hooker2022} presents stochastic decision diagrams for discrete stochastic dynamic programming models by incorporating probabilities into the arcs of the DDs. These diagrams can then be used to derive optimal policies for the problem. 
 Similar to \cite{lozano2018bdd}, we will consider pure-binary second-stage variables and model these problems via BDDs. We remark that these second-stage problems need not be linear and the structure of the BDDs in our novel problem is distinct from those of the generalization of  Lozano and Smith.

%

%
%

\section{Benders Decomposition Framework} \label{sec:benders_alg}

We solve Problems \ref{prob:constrs} and \ref{prob:costs} using a Benders decomposition framework and reformulate the recourse problems using BDDs. 
Given a finite set of scenarios $\Omega$, and a {\credrev set of cuts indexed by $\cutset = \{1, 2, \hdots\}$} we formulate the Benders master problem as:
	\bsubeq\label{form:rmp}
\begin{alignat}{2}
	\MP(\cutset): \min_{\substack{
				{\xvar} \in \mathcal{X} \\ \indvar \in \{0,1\}^{m_1}}} \ & {\xweights^\top \xvar} + \sum_{\omega \in \Omega} p_\omega \eta_\omega\\
	\text{s.t.} \	& \indvar_i =	\indicator(\cond)  \qquad&&  \forall i =1,\hdots,m_1, \ \forall \omega \in \Omega \label{eq:mp_logical}\\
	& {\credrev \eta_\omega \geq C^k_{\mpsol^\omega}(\indvar) }&& {\credrev \forall k \in \cutset} \label{eq:benderscuts} 
\end{alignat}%
\esubeq%
The variables $\eta_\omega$ represent the second-stage objective function estimates, {\credrev whose values are determined via the cuts in \eqref{eq:benderscuts}. The function $C^k_{\mpsol^\omega}(\indvar)$ comes from the subproblem solved at iteration $k$ with the input vector $\mpsol^\omega$ and returns an affine expression of $\indvar$ (detailed in Section \ref{sec:benders})}.  In the master problem, the constraints \eqref{eq:mp_logical} have been modified with the use of an indicator variable $\indvar_i$. These variables are required to formulate the problem generically, as the logical expressions modify either the costs or capacities of BDD arcs, although typically, these variables and the associated constraints \eqref{eq:mp_logical}
are not necessary. As can be seen in Table \ref{table:applications}, the logical expressions of the first-stage variables are often expressions that must be equal to a binary value. In such cases we do not need the indicator variable or constraints and can directly pass these expressions to the BDDs as arc capacities or costs. 
 {\cred For example, if we are modelling vertex cover as Problem 
 \ref{prob:costs}, the logical expression is true if and only if $x_v =1$ for vertex $v$, we can therefore replace the indicator variable with $x_v$, thus work with arc costs of $1-x_v$ as we will show later. On the other hand, for Problem \ref{prob:constrs}, the logical expression is true if and only if $x_u+x_v =0$ for edge $(u,v)$, we can therefore replace the indicator variable with $1 - x_u - x_v$, thus work with arc capacities of $x_u + x_v$ as the complement of the indicator variable as we will show later.}
In the remainder of this work we will use the indicator variables $\indvar$ to show the results both for ease of presentation and to keep the problem formulations as general as possible. 



Using the indicator variables $\indvar_i \in \{0,1\}$ for all $i = 1, \hdots, m_1$, the constraints \eqref{eq:sp1_logical} for scenario $\omega$ become
\bsubeq
\begin{alignat}{2}
&{ \cred \indvar_i =	\indicator(\cond)} & \quad \forall i =1,\hdots,m_1 \label{eq:logical1_mp} \\
&	\indvar_i =1  \implies\costone{\costindex} =   0 & \quad\forall i =1,\hdots,m_1 \label{eq:logical1_sp}
\end{alignat}%
\esubeq %
and the constraints \eqref{eq:sp2_logical2} for scenario $\omega$ become 
	\bsubeq
\begin{alignat}{2}
&{ \cred \indvar_i =	\indicator(\cond)}&& \quad \forall i =1,\hdots,m_1\label{eq:logical2_mp} \\
&	\indvar_i =1  \implies  {\yvar} \in \mathcal{W}_i(\omega) && \quad\forall i =1,\hdots,m_1. \label{eq:logical2_sp}
\end{alignat}%
\esubeq %

We now have that constraints \eqref{eq:logical1_mp} and \eqref{eq:logical2_mp}  will now appear in the master problem as \eqref{eq:mp_logical}, and constraints  \eqref{eq:logical1_sp} and \eqref{eq:logical2_sp} remain in the respective subproblems. In the worst case, using these indicator variable will add $|\Omega|$-many variables for each of the $m_1$ logical constraints, however, as previously mentioned, the use of indicator variables is typically not necessary since the indicators are often simple expressions on the first-stage variables. We remark that if the technology matrix in the second stage is fixed, i.e., is not scenario dependent, we would not require the scenario index for the indicator variables.


For both of the BDD reformulation methods, the cutting plane algorithm is the same, the only difference will be in building the BDDs in the initialization step and the subproblems and therefore the cuts' structure. 
The details of the BDD construction are outlined in Section \ref{sec:recourse_reform} and an overview of the Benders algorithm can be found in \ref{app:bendersoverview}.

\section{Recourse Problem Reformulation via BDD} \label{sec:recourse_reform}

\subsection{Deterministic Equivalent Problem}
As is well established in the literature as good modelling practice, in both problem types, we assume relatively complete recourse, that is, for any ${x} \in \mathcal{X}$ and $\omega \in \Phi$ there exists a feasible solution to the second-stage problem.
We remark that in the case relatively complete recourse does not exist, feasibility cuts can easily be incorporated into the methods we propose.
Following the traditional approach in two-stage stochastic programming we focus on solving the sample average approximation (SAA) problem. We assume a finite number of realizations for the vector of random variables, each given by a scenario $\omega \in \Omega$, and probability $\prob_\omega$. Under such conditions, introducing scenario copies of the recourse decisions, $\yvar^\omega$, we are able to rewrite the two-stage problem as the monolithic deterministic equivalent problem:
\bsubeq\label{form:deteq}
\begin{alignat}{2}
	\min \ & {\xweights^\top \xvar} + \sum_{\omega \in \Omega} \prob_\omega  (\costone{} + \costtwo{})^\top  {\yvar}^\omega \\
	\text{s.t.} \ & {\xvar} = (\binxvar, \intxvar, \contxvar) \in \mathcal{X} \subseteq  \{0,1\}^{\binindex} \times \Z^{\intindex} \times \R^{\contindex}  	\\
	&  \indicator\left(\cond\right) = 1 \implies \costone{\costindex}=   0 && \quad \forall i =1,\hdots,m_1, \ \forall \omega \in \Omega \label{eq:logical1}\\
	&  {\yvar}^\omega \in \mathcal{Y}(\omega) && \quad \forall \omega \in \Omega \\
	&  {\yvar}^\omega \in  \{0,1\}^{\yvarindex} && \quad \forall \omega \in \Omega
\end{alignat}%
\esubeq%
Here we use the subproblem of Problem 3  to formulate the deterministic equivalent but the constraints of Problem 2 could be substituted similarly.
Typically, since the second-stage variables are binary, the problem \eqref{form:deteq} is solved as a monolithic  model, or in specific cases it is solved via the integer L-shaped decomposition method. However, solving the deterministic equivalent problem does not scale well when we have a large number of scenarios and the integer L-shaped method does not have strong cuts and may be slow to converge. This motivates the search for better methods to solve stochastic programs with binary recourse. We propose decomposition methods that use BDDs to convexify the second-stage problems and derive Benders cuts that are stronger than integer L-shaped cuts. 

\subsection{Recourse Problem Reformulation}

For both Problems \ref{prob:constrs} and \ref{prob:costs} a single BDD per scenario $\omega$ is generated in the initialization step of the algorithm.
{\credrev To construct the BDDs we start by assuming for  Problem  \ref{prob:constrs} that $\indicator(\cond) = 1$ for all $i =1,\hdots,m_1$ , that is, the most constrained problem for each scenario and for Problem  \ref{prob:costs}  we assume $\indicator(\cond) =0$ for all $i =1,\hdots,m_1$, that is, there is no discount for any of the variables. 
	For Problem \ref{prob:constrs}, as each node is added we check if the state is infeasible. If this infeasibility is due to violation of one or more of the constraints $\mathcal{W}_i(\omega)$,  for each violated constraint, we add a capacity to the arc that transitions to this state ensuring it can only be reached if that infeasibility is eliminated by a first-stage solution. 
	For Problem \ref{prob:costs}, when we create an arc corresponding to assigning a second-stage variable value $1$, we parametrize its cost by the first stage solution, so that if the logical expression is true, the cost is now $0$.}
The structure of these BDDs will remain the same during the Benders algorithm, only their arc costs or capacities will be updated via the master problem solution $\mpsol^\omega$  from one iteration to the next. 

 We denote a BDD for scenario $\omega$ as $\BDD^\omega = (\nodes^\omega, \arcs^\omega)$, and partition the arcs $\arcs^\omega$ into two sets: $\onearcs$ the arcs that assign a value of $1$ to a second-stage variable and $\zeroarcs$ the arcs that assign a value of $0$ to a second-stage variable. $\BDD^\omega$ consists of {\cred $\yvarindex$} arc layers 
$\arclayer{1}, \arclayer{2}, \hdots, \arclayer{{\yvarindex}}$ and of $\yvarindex+1$ 
node layers  $\nodelayer{1}, \nodelayer{2}, \hdots, \nodelayer{{\yvarindex}}, \nodelayer{{\yvarindex+1}}$, where layer $\nodelayer{1}$ contains only the root node $\rootnode$ and layer $\nodelayer{{\yvarindex+1}}$ contains only the terminal node $\term$.
We map from an arc $a$ to a second-stage variable index $j$ via the function $\map^\omega(a) = j$, for ease of presentation we assume the same variable ordering for each scenario and write the function as $\map(a)$. 
Each $\rootnode$-$\term$ path in a BDD represents a feasible solution to the associated recourse problem by assigning values  to $\yvar$, and an optimal recourse solution is one corresponding to a shortest path. {\cred Any node in the BDD associated with a state that cannot be completed to a feasible solution is called an infeasible state and can be pruned from the BDD.}

Without loss of generality, we assume the variable ordering of the second-stage variables $\yvar$ is $1,2,\hdots,\yvarindex$. We denote the set of all states by $\stateset$ and the set of terminal states by $\termstateset$. For a fixed scenario $\omega$, we denote the state where the variables  $1, 2, \hdots,k-1$,  but not variable $k$, have been assigned values by $\state{k}$. We define the state transition function $\transitionfunc_k: \stateset \setminus \termstateset \times \{0,1\} \mapsto \stateset$, where $\state{k+1} = \transitionfunc_k(\state{k}, \hat{y}_k)$, where $\hat{y}_k$ is the value assigned to variable $y_k$. We denote the set of remaining second-stage solutions that can be explored given a state $ \state{k}$ and  assignment $\hat{y}_k$, $\remsol(\state{k}, \hat{y}_k)$. 
We denote the BDDs resulting from Problems \ref{prob:constrs} and \ref{prob:costs} as \bdd{2} and \bdd{3}, respectively.

{\credrev We reiterate at this time that while we present Problems \ref{prob:constrs} and \ref{prob:costs} as alternatives we are abusing notation slightly. They are not truly equivalent, since they do not necessarily have the same feasible region of the second-stage problem given a first-stage decision. However as seen in Section \ref{subsec:applications}, for a fixed application we can model and solve it as an instance of either Problem \ref{prob:constrs} or Problem \ref{prob:costs}. In the following sections we will detail the BDD-based reformulations for both problems. These BDDs will typically result in very different recourse solution vectors, but will nonetheless solve the same problem with the same optimal objective value and first-stage optimal solution.
}

\subsection{\bdd{2} Reformulation}

We restate the method of  \cite{lozano2018bdd}, which allows us to reformulate the recourse problem \eqref{eq:recourse2}-\eqref{eq:sp2_yconstr}
 by building a BDD where some arcs have capacities parameterized by the master problem solution. We remark that when the variables (or expressions of variables) of the first-stage problem can be used directly as the indicators these variables (or expressions) will appear as the arc capacities. 

\subsubsection{\bdd{2} Representation}

The procedure to build \bdd{2}  is given in Algorithm \ref{alg:bdd1}. It begins by building the BDD assuming all second-stage constraints are enforced, using the transition function to move to a new state for each value choice for a variable $\yvar_j$. 
When a state, $s^{j+1}$, becomes infeasible after assigning variable $y_{j}$ value $\hat{y}_{j}$, we first check if the cause of the infeasibility is due to the constraint set in $\mathcal{Y}(\omega)$.
That is if $\remsol(s^j, \hat{y}_j ) \cap \mathcal{Y}(\omega) = \emptyset$, then no matter the value of the first-stage solution, every solution in $\remsol(s^j, \hat{y}_j )$ is infeasible. 
However, if the state is infeasible but $\remsol(s^j, \hat{y}_j ) \cap \mathcal{Y}(\omega) \neq \emptyset$, then there exists some set of violated constraints with indices in $\viol = \{ i : \remsol(s^j, \hat{y}_j ) \cap \mathcal{W}_i(\omega) = \emptyset\}$ corresponding to indicator variables $\indvar_i$. 
{\cred We remark that in Algorithm \ref{alg:bdd1} we abuse notation slightly by using $\state{\node}$ to denote the state at node $u$ within the expression $\remsol(\cdot)$. This is because at a node $u$ in node layer $\nodelayer{k}$ of the BDD we have assigned values to the first $k-1$ many  $\yvar$ variables via the arcs so in some sense it is the same as $ \state{k}$.}
The new state is defined by the transition function and by setting $\indvar_i =0$ for all $i \in \viol$, that is, we update the state defined by the transition function to reflect the second-stage constraints ${\yvar} \in \mathcal{W}_i(\omega)$ being made redundant. We then add capacities $\indvar_i$ for all $i \in \viol$ to the arc corresponding to the value $\hat{y}_j$. We remark that a single arc may have multiple capacities, one for each $i$ in $\viol$, all of which must equal one in order to utilize this arc. This ensures that the new state is only achievable when the violated constraints have been made redundant by the first-stage solutions. After the BDD is created, it is reduced so that no pair of nodes in a layer are equivalent, see \citep{bergman2016decision,wegener2000} for more information about BDD reductions.
	\begin{algorithm}[htbp]
	\KwIn{Transition functions $\transitionfunc_j$, scenario $\omega$}
	\KwOut{A reduced, exact BDD for \eqref{eq:recourse2}-\eqref{eq:sp2_yconstr}}
	Create root node $\rootnode$ with initial state $\state{\rootnode}$\;
	\ForAll{$j \in \{1,2,\hdots, \yvarindex\}$, $\node \in \nodelayer{j}$}{
		\For{$b \in \{0,1\}$}{
			Update $\remsol(\state{\node},b)$\;
			\eIf{$\transitionfunc_j(\state{\node}, b)$ is feasible}{
				Create node $v \in \nodelayer{{j+1}}$ with state $\transitionfunc_j(\state{\node}, b)$\;
				\eIf{$b = 1$}{
					Add uncapacitated arc $a = (u,v) \in \onearcs$ with cost $\yweights_{\map(a)}(\omega)$ to $\arclayer{j}$\;
				}{
					Add uncapacitated arc $(u,v) \in \zeroarcs$ with cost $0$ to $\arclayer{j}$\;
				}
			}{
				\If{$\remsol(\state{\node},b) \cap \mathcal{Y}(\omega) \neq \emptyset$}{
					Find set of violated constraint indices $\viol = \{ i : \remsol(\state{\node},b) \cap \mathcal{W}_i(\omega) = \emptyset\}$\;
					Create node $v \in \nodelayer{{j+1}}$ with state {\cred $\bar{\state{v}}$, which assigns value $b$ to $\hat{\yvar}_j$ and sets $\indvar_i = 0 \ \forall i \in \viol$ }\;
					\eIf{$b = 1$}{
						Add arc $a = (u,v) \in \onearcs$ with cost $\yweights_{\map(a)}(\omega)$ and capacities $(1-\indvar_i )\ \forall i \in \viol$  to $\arclayer{j}$\;
					}{
						Add arc $(u,v) \in \zeroarcs$ with cost $0$ and capacities $(1-\indvar_i )\ \forall i \in \viol$ to $\arclayer{j}$\;
					}
				}
			}
		}
	}
	Reduce the resulting BDD 
	\caption{\bdd{2} construction} \label{alg:bdd1}
\end{algorithm}

\begin{example} \label{ex:bdd1}
	We will use the SMWDS as a running example throughout this paper. An IP formulation for the recourse problem of SMWDS is:
\bsubeq
	\begin{alignat}{2} \label{form:smwds_sp1}
		\mathcal{Q}({\xvar}, \omega) = \min_{\yvar \in  \{0,1\} ^{|\vertexset' |} } &\sum_{v \in \vertexset'} d_v(\omega) \yvar_v \\
		\text{s.t.} \ & \yvar_v + \sum_{u \in \neighb'(v)} \yvar_u  \geq 1 -\left( \hat{\xvar}_v +  \sum_{u \in \neighb'(v)}  \hat{\xvar}_u \right)  \quad && \forall \ v \in \vertexset' \label{eq:cover1}
	\end{alignat}%
	\esubeq %
where $\vertexset'$ is the second-stage set of vertices, and  $\neighb'(v)$ is the  neighbourhood of vertex $v$ in the second stage. For illustrative purposes, we consider an instance with a single scenario whose scenario graph is pictured in Figure \ref{fig:smwds_graph}. 
\begin{figure}[htbp]
	\centering
	\begin{tikzpicture}[main_node/.style={circle,fill=white!80,draw,inner sep=0pt, minimum size=16pt},
		line width=1.2pt]
		\node[] (w0) at (0, 2) {\cblue$1 + 0$};
		\node[main_node] (v0) at (0, 1.5) {$0$};
		\node[]( w1) at (2,2) {\cblue$1 + 0$};
		\node[main_node] (v1) at (2,1.5) {$1$};
		\node[] (w2) at (2,-.5) {\cblue$1 + 0$};
		\node[main_node] (v2) at (2,0) {$2$};
		\node[] (w3) at (0,-.5) {\cblue$1 + 0$};
		\node[main_node,] (v3) at (0,0) {$3$};
		\node[] (w4) at (-1,1.25) {\cblue$1 + 0$};
		\node[main_node] (v4) at (-1,0.75) {$4$};
		
		\draw[-] (v3) -- (v0)  -- (v1) -- (v2) -- (v3) -- (v4) -- (v2);	

	\end{tikzpicture}
	\caption{A single scenario for the SMWDS problem, where the indices are given inside the vertices, and the vertex weights ($\secondstagecost{} = \costone{}+ \costtwo{}$) are outside the vertices. \label{fig:smwds_graph}}
\end{figure}
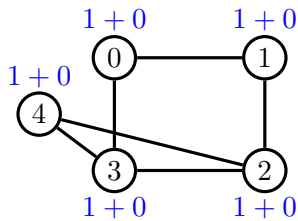

	 The indicator constraints for this problem are equivalent to the IP constraints \eqref{eq:cover1}, enforce that if  a vertex $v \in \vertexset'$ is not dominated in the first stage it must be dominated in the second stage.  We also remark that we do not need to introduce an indicator variable for this problem and instead use a summation of the appropriate first-stage variables. For a vertex $v$, if $ \xvar_v +   \sum_{u \in \neighb'(v)}  \xvar_u  =0$ then we must enforce covering constraints in the second stage. We use this expression as the negation of the indicator $(1-\indvar_v)$.
%
	   The \bdd{2} for this scenario is pictured in Figure \ref{fig:smwds_bdd1}. 
	    The state at each node represents the vertices that still need to be dominated, and each solid arc has a weight of one.
	    {\cred Recall that to build  \bdd{2} we assumed that all the second-stage constraints are being enforced, for this problem that means we have not dominated any vertices in the first stage, i.e., $\hat{\xvar}_v = 0$ for all $v \in \vertexset'$. Thus the state at the root node is all the vertices and if we choose $\hat{y}_0=1$ the state becomes $2,4$ as the vertices $0,1,3$ are dominated by vertex $0$.}
	   We can see that if we select $\yvar_{0} = \yvar_{1} = 0$, under the assumption that all second-stage constraints are enforced, if we also set $\yvar_{2} = 0$ we transition to an infeasible state. 
	 
	 This state is infeasible because we have $\hat{\xvar}_v =0$ for all $v \in \vertexset'$ while constructing the BDD,  and  $\yvar_{0} = \yvar_{1} = \yvar_{2} = 0$ and the domination constraint \eqref{eq:cover1} for vertex $1$ is violated since
	 $\yvar_{0} + \yvar_{1} + \yvar_{2} < 1 - \hat{\xvar}_0 -  \hat{\xvar}_1 -  \hat{\xvar}_2.$
{\cred 	 However, the violated constraint in this infeasible state is the logical constraint and not the other constraints on the second-stage variables, i.e., we have violated a constraint in the set $\mathcal{W}_1(\omega)$ not in the set $\mathcal{Y}(\omega)$ and $\viol = \{1\}$.
	 Thus, we could have selected one of $\xvar_0$, $\xvar_1$, or $\xvar_2$ to dominate vertex $1$ in the first stage and this state would no longer be infeasible. To model this in the BDD, we add a node in the next layer with state $0,2,3,4$, i.e., assuming vertex $1$ is dominated (not necessarily selected) in the first stage, and add an arc for the decision $\yvar_{2} = 0$ to this new node with capacity $\xvar_{0} + \xvar_{1} + \xvar_{2}$. The arc capacity now enforces that if $ \hat{\xvar}_0 = \hat{\xvar}_1 = \hat{\xvar}_2 = 0$, there can be no flow on the arc and the new state is unreachable.
}	 
	  In this example, we can also see that there may be multiple capacity constraints on a single arc, such as in the last arc layer of the BDD. Finally, we remark that since the arcs have different capacities, we are unable to further reduce the BDD.
	\end{example}
	
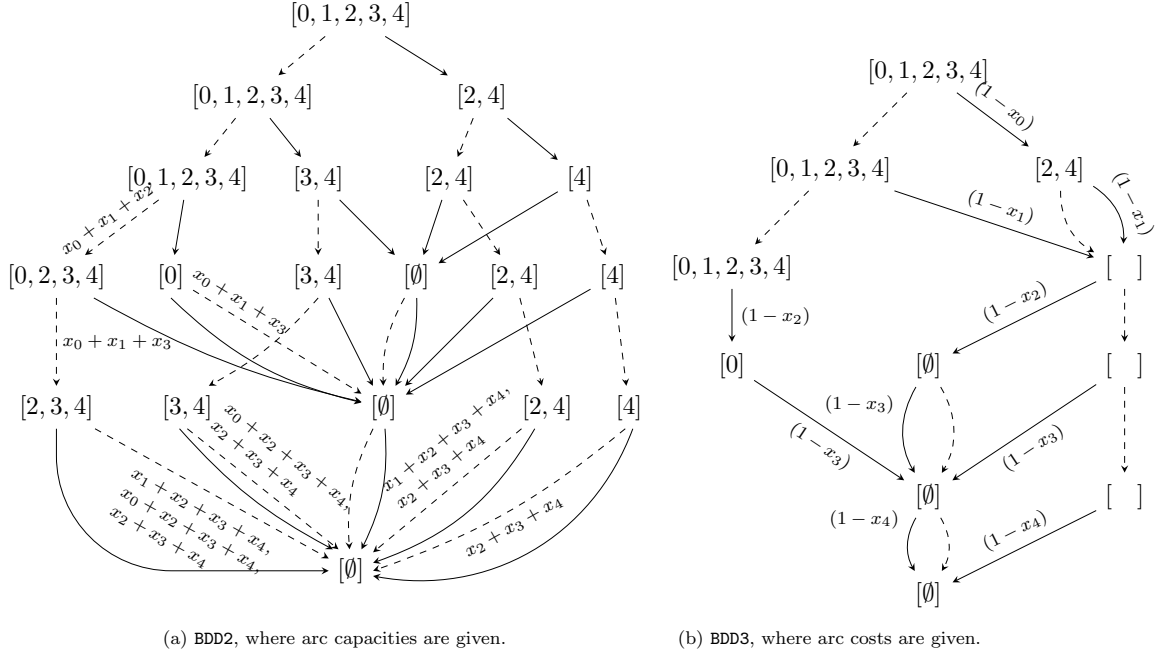
\begin{figure}[t]
	\centering
	\scalebox{0.87}{
	\begin{subfigure}[t]{0.5\textwidth}
\hspace*{-1.5cm}	
\begin{tikzpicture}[	main_node/.style={}, node distance = 1cm and 1cm,  > = stealth, 
		]

		\node[main_node] (s) at  (-0.5,0) {$[0,1,2,3,4]$};
		
		\node[main_node] (l1n1) at (-2,-1.25) {$[0,1,2,3,4]$};
		\node[main_node] (l1n2) at  (1.5,-1.25) {$[2,4]$};
		
		\node[main_node] (l2n1) at (-3,-2.5) {$[0,1,2,3,4]$};
		\node[main_node] (l2n2) at (-1,-2.5) {$[3,4]$};
		\node[main_node] (l2n3) at (1,-2.5) {$[2,4]$};
		\node[main_node] (l2n4) at (3,-2.5) {$[4]$};
		
		\node[main_node] (l3n0) at (-5,-4) {$[0,2,3,4]$};
		\node[main_node] (l3n1) at (-3.25,-4) {$[0]$};
		\node[main_node] (l3n2) at (-1,-4) {$[3,4]$};
		\node[main_node](l3n3) at (0.5,-4) {$[\emptyset]$};
		\node[main_node] (l3n4) at (2,-4) {$[2,4]$};
		\node[main_node] (l3n5) at (3.5,-4) {$[4]$};

		\node[main_node] (l4n0) at (-5,-6) {$[2,3,4]$};
		\node[main_node] (l4n1) at (-3,-6) {$[3,4]$};
		\node[main_node](l4n2) at (0,-6) {$[\emptyset]$};
		\node[main_node] (l4n3) at (2.5,-6) {$[2,4]$};
		\node[main_node] (l4n4) at (3.75,-6) {$[4]$};
		
		\node[main_node](t) at (-.5,-8.5) {$[\emptyset]$};

		\draw[dashed, ->] (s) to (l1n1);
		\draw[->] (s) to (l1n2);
		
		\draw[dashed, ->] (l1n1) to  (l2n1);
		\draw[->] (l1n1) to (l2n2);
		\draw[dashed, ->] (l1n2) to []  (l2n3);
		\draw[->] (l1n2) to []  (l2n4);

		\draw[dashed, ->] (l2n1) to node[sloped, above, xshift=-0.1cm]{\scriptsize $x_0+x_1+x_2$} (l3n0);
		\draw[->] (l2n1) to (l3n1);
		\draw[dashed, ->] (l2n2) to (l3n2);
		\draw[->] (l2n2) to (l3n3);
		\draw[dashed, ->] (l2n3) to (l3n4);
		\draw[->] (l2n3) to (l3n3);
		\draw[dashed, ->] (l2n4) to (l3n5);
		\draw[->] (l2n4) to (l3n3);

		\draw[dashed, ->] (l3n0) to [bend right=0] node[right, xshift =-0.25cm,rotate=0]{\scriptsize \begin{tabular}{l}
				$x_0+x_1+x_3$
			\end{tabular} }  (l4n0);
		\draw[->] (l3n0) to [bend right=8]  (l4n2);
		
		\draw[dashed, ->] (l3n1) to [bend right=0] node[sloped, above, xshift=-.8cm, rotate=0]{\scriptsize  $x_0+x_1+x_3$}  (l4n2);
		\draw[->] (l3n1.south) to [bend right= 15]  (l4n2);
		
		\draw[dashed, ->] (l3n2.south) to  [bend left=10]  (l4n1.45);
		\draw[->] (l3n2) to (l4n2);
		
		\draw[dashed, ->] (l3n3) to [bend right=15]  (l4n2);
		\draw[->] (l3n3) to [bend left=15]  (l4n2);
		
		\draw[dashed, ->] (l3n4) to (l4n3);
		\draw[->] (l3n4) to (l4n2);
		\draw[dashed, ->] (l3n5) to (l4n4);
		\draw[->] (l3n5) to (l4n2);

		\draw[dashed, ->] (l4n0) to [bend right=0] node[sloped,below,rotate= 0 ]{\scriptsize
			\begin{tabular}{l}
				$x_1+x_2+x_3+x_4$,\\$x_0+x_2+x_3+x_4$,\\ $x_2+x_3+x_4$
		\end{tabular}}   (t);
		\draw[rounded corners=50pt,->] (l4n0.south) |- (t.west);

		\draw[dashed, ->] (l4n1.south) to [bend left=0] node[sloped,above, rotate=0 ]{\scriptsize
			\begin{tabular}{l}
				$x_0+x_2+x_3+x_4$,\\ $x_2+x_3+x_4$
		\end{tabular}}   (t.120);
		\draw[->] (l4n1.250) to [bend right=10]  (t.130);
		
		\draw[dashed, ->] (l4n2) to [bend right=15]  (t);
		\draw[->] (l4n2) to [bend left=15]  (t);
		
		\draw[dashed, ->] (l4n3) to [bend right=0] node[sloped, above, rotate= 0, xshift =0.5cm]{\scriptsize \begin{tabular}{l}
				$x_1+x_2+x_3+x_4$,\\ $x_2+x_3+x_4$
		\end{tabular} }   (t);
		\draw[->] (l4n3) to [bend left=20]  (t);
		
		\draw[dashed, ->] (l4n4.260) to [bend left=10] node[sloped, below, rotate=0]{\scriptsize $x_2+x_3+x_4$}   (t.east);

		\draw[->] (l4n4.280) to [bend left=400]  (t);
		
	\end{tikzpicture}
	\caption{\bdd{2}, where arc capacities  are given. \label{fig:smwds_bdd1}}
	\end{subfigure}
~
	\begin{subfigure}[t]{.45\textwidth}
		\centering
	\hspace*{0.7cm}		\begin{tikzpicture}[main_node/.style={},
			node distance = 1cm and 1cm,  > = stealth, 
			shorten > = 1pt, 
			]
			\node[main_node] (s) at  (0,0) {$[0,1,2,3,4]$};
			
			\node[main_node] (l1n1) at (-1.5,-1.5) {$[0,1,2,3,4]$};
			\node[main_node] (l1n2) at  (2,-1.5) {$[2,4]$};
			
			\node[main_node] (l2n1) at (-3,-3) {$[0,1,2,3,4]$};
			\node[main_node] (l2n2) at (3,-3) {$[\quad]$};

			\node[main_node] (l3n1) at (-3,-4.5) {$[0]$};
			\node[main_node](l3n2) at (0,-4.5) {$[\emptyset]$};
			\node[main_node] (l3n3) at (3,-4.5) {$[\quad]$};

			\node[main_node](l4n1) at (0,-6.5) {$[\emptyset]$};
			\node[main_node] (l4n2) at (3,-6.5) {$[\quad]$};

			\node[main_node](t) at (0,-8) {$[\emptyset]$};

			\draw[dashed, ->] (s) to (l1n1);
			\draw[->] (s) to node[sloped, above, rotate=0]{ \scriptsize$(1-\xvar_{0})$}  (l1n2);
			
			\draw[->] (l1n1) to node[sloped, above, rotate= 0]{ \scriptsize$(1-\xvar_{1})$}  (l2n2);
			\draw[dashed, ->] (l1n1) to (l2n1);
			
			\draw[dashed, ->] (l1n2) to [bend right=30]  (l2n2);
			\draw[->] (l1n2) to [bend left=30]  node[sloped, above, rotate= 0]{ \scriptsize$(1-\xvar_{1})$} (l2n2);
			
			\draw[->] (l2n1) to node[right, rotate= 0]{ \scriptsize$(1-\xvar_{2})$}  (l3n1);
			\draw[->] (l2n2) to node[above, sloped,rotate= 0]{ \scriptsize$(1-\xvar_{2})$}  (l3n2);
			\draw[dashed, ->] (l2n2) to (l3n3);

			\draw[->] (l3n1) to node[sloped, below, rotate= 0]{ \scriptsize$(1-\xvar_{3})$} (l4n1);
			\draw[dashed, ->] (l3n2) to [bend left=30]  (l4n1);
			\draw[->] (l3n2) to [bend right=30]  node[left, rotate= 0, yshift =0.4cm]{ \scriptsize$(1-\xvar_{3})$}  (l4n1);
			\draw[dashed, ->] (l3n3) to  (l4n2);
			\draw[->] (l3n3) to node[sloped, below, rotate= 0]{ \scriptsize$(1-\xvar_{3})$}(l4n1);

			\draw[dashed, ->] (l4n1) to [bend left=30]  (t);
			\draw[->] (l4n1) to [bend right=30] node[left, rotate= 0, yshift =0.4cm]{ \scriptsize$(1-\xvar_{4})$}    (t);		
			\draw[->] (l4n2) to [] node[above, sloped, rotate= 0]{ \scriptsize$(1-\xvar_{4})$}   (t);

		\end{tikzpicture}
		\caption{\bdd{3}, where arc costs are given. \label{fig:smwds_bdd2}}
	\end{subfigure}
	}
	\caption{\bdd{2} and \bdd{3} {\credrev (reduced)} reformulations for the instance in Figure \ref{fig:smwds_graph}. {\credrev Node states indicate the vertices left to dominate in the problem. State $[\emptyset]$ indicates all vertices have been  dominated, whereas state $[\quad ]$ indicates that multiple nodes have been merged during the reduction of the BDD}. \label{fig:bdd_eg}}
\end{figure}%

\subsubsection{\bdd{2} Recourse Problem Reformulation}

Given a scenario $\omega$  and a first-stage solution $\mpsol^\omega$, the \bdd{2}  reformulation of \eqref{eq:recourse2}-\eqref{eq:sp2_yconstr} yields a minimum cost network flow problem $\Primal{2}$ with capacity constraints \eqref{eq:capconstr} on some arcs of the BDD. In the linear program (LP)  \eqref{form:BDD1_primal} we have continuous decision variables $f_a$ that give the flow over arc $a \in \arcs^\omega$ and we define the set of indices $\viol_a$ for an arc $a \in \arcs^\omega$ such that there can be no flow on $a$ unless $\mpsol^\omega_i = 0$ for all $ i \in \viol_a$.
\bsubeq\label{form:BDD1_primal}
\begin{alignat}{2}
	\Primal{2}: \min_{f \in \R_+^{|\arcs^\omega |}} \ & \sum_{a  \in \onearcs} f_a\secondstagecost{\map(a)} \\
		\text{s.t.} \ & \text{valid $\rootnode$-$\term$ flow}\\
	& f_a \leq 1 - \mpsol^\omega_i && \forall \ a \in \arcs^\omega , i \in \viol_a \label{eq:capconstr}
\end{alignat}
\esubeq
To derive traditional Benders cuts from this problem we take the dual of the modified minimum cost network flow problem, introducing variables $\pi$ and $\beta$ for the flow conservation and capacity constraints respectively. The complete formulations can be found in \ref{app:bddforms}.

Without loss of generality, we assume $\pi_\term =0$ as the constraint of the primal associated with this variable is linearly dependant on the other constraints.
Given a first-stage solution $\mpsol^\omega$ and a scenario $\omega$,  	$\Primal{2}$ is equivalent to \eqref{eq:recourse2}-\eqref{eq:sp2_yconstr} since
each $\rootnode$-$\term$ path in the BDD corresponds to exactly one feasible solution of  the recourse problem with equal objective function, and vice versa.

\subsection{\bdd{3} Reformulation}

The second BDD reformulation builds BDDs where the first-stage solutions impact the arc weights and allows for a convex representation of subproblem  \eqref{eq:altprob_obj}-\eqref{eq:sp1_yconstr}.

\subsubsection{\bdd{3} Representation}

The procedure to create a BDD of type \bdd{3} is given in Algorithm \ref{alg:bdd2}. The only difference from the creation of a standard BDD for the given transition function is that when we create an arc corresponding to a decision assigning $\yvar_{\map(a)} =1$, we give the arc cost $\costone{\map(a)} (1-\mpsol^\omega_{\map(a)})+ \costtwo{\map(a)}$. This ensures that when a first-stage decision is $\mpsol^\omega_{\map(a)} = 1$, the cost to make decision $y_{\map(a)} $ is $\costtwo{\map(a)}$ otherwise the cost to make decision $y_{\map(a)} $ is $\costone{\map(a)}+\costtwo{\map(a)}$. We remark that all the arcs in this BDD are uncapacitated. 

\begin{algorithm}[htbp]
	\KwIn{Transition functions $\transitionfunc_j$, scenario $\omega$}
	\KwOut{A reduced, exact BDD  \eqref{eq:altprob_obj}-\eqref{eq:sp1_yconstr}}
	Create root node $\rootnode$ with initial state\;
	\ForAll{$j \in \{1,2,\hdots, \yvarindex\}$, $\node \in \nodelayer{j}$}{
		\For{$b \in \{0,1\}$}{
			
			\If{$\transitionfunc_j(\state{\node}, b)$ is feasible}{
				Create node $v \in \nodelayer{{j+1}}$ with state $\transitionfunc_j(\state{\node}, b)$\;
				\eIf{$b = 1$}{
					Add arc $a = (u,v) \in \onearcs$ with cost $\costone{\map(a)} (1-\mpsol^\omega_{\map(a)})+\costtwo{\map(a)}$ to $\arclayer{j}$\;
				}{
					Add arc $(u,v) \in \zeroarcs$ with cost $0$ to $\arclayer{j}$\;
				}
				
			}
		}
	}
	Reduce the resulting BDD 
	\caption{\bdd{3} construction} \label{alg:bdd2}
\end{algorithm}

\begin{example} \label{ex:bdd2}
	Recall the scenario in Figure \ref{fig:smwds_graph}, to construct \bdd{3} we proceed as if we were building a dominating set BDD. Again, the state at each node represents the vertices that still need to be dominated, each dashed arc has a weight of zero, and because each vertex has cost $\costone{i} + \costtwo{i}= 1+ 0$ each solid arc will have a weight of $(1-\hat{\xvar}_i)$, that is, if $\hat{\xvar}_i = 1$, then $\costone{i} = 0$. 
	The reduced \bdd{3} can be seen in Figure \ref{fig:smwds_bdd2}, where nodes that were merged during the reduction have no state shown. We remark that this BDD has fewer nodes and has a smaller width than the BDD created using the  \bdd{2} procedure for the same instance.

\end{example}

\subsubsection{\bdd{3} Recourse Problem Reformulation}

Given a fixed scenario $\omega$ and
a master problem solution $\mpsol^\omega$, the BDD reformulation of \eqref{eq:altprob_obj}-\eqref{eq:sp1_yconstr} yields a minimum cost network flow problem, $\Primal{3}$, with a modified objective function over the BDD.  Again, let continuous decision variables $f_a$ be the flow over arc $a \in \arcs^\omega$.
\bsubeq\label{form:BDD2_primal1}
\begin{alignat}{2}
	\Primal{3}: \min \ & \sum_{a \in \onearcs} f_a\left(\costone{\map(a)} (1-\mpsol^\omega_{\map(a)})+ \costtwo{\map(a)}\right) \\
	\text{s.t.} \ & \text{valid $\rootnode$-$\term$ flow}
\end{alignat}%
\esubeq%
Deriving cuts from this reformulated subproblem is not as straightforward as in the case of $\bdd{2}$.
We next describe in detail how to derive Benders 
cuts from $\Primal{3}$.

\subsubsection{An Alternative Recourse Problem Reformulation for \bdd{3}}
Traditionally, to derive  Benders cuts using the BDD reformulation, we take the dual of \eqref{form:BDD2_primal1}.
%
However, Benders cuts cannot immediately be taken from this dual problem since the first-stage variables appear in the constraints, not the objective. We will reformulate the primal problem into $\AltPrimal{3}$ to ensure the first-stage variables appear only in the objective of the dual subproblem and show at an optimal solution these problems are equivalent in the sense that they achieve the same objective value for the same master problem solution.
Introducing a new set of arc variables, $\gamma$, we propose the following extended minimum cost network flow formulation:
	\bsubeq\label{form:BDD2_primal2}
\begin{alignat}{2}
	\AltPrimal{3}: 	\min \ & \sum_{a \in \onearcs}  \left(\costone{\map(a)} + \costtwo{\map(a)}\right)f_a + \sum_{a \in \onearcs} \costtwo{\map(a)}\gamma_a\\
	\text{s.t.} \ & \sum_{a=(\rootnode,j) \in \arcs^\omega} f_a + \sum_{a =(\rootnode,j) \in \onearcs}  \gamma_a  = 1 \label{eq:P2_flow1} \\ 
	&\sum_{a =(i,\term) \in \arcs^\omega} f_a + \sum_{a =(i,\term) \in \onearcs} \gamma_a  = 1 \\ 
	& \sum_{\substack{a =(i,j)\\ \in \arcs^\omega}} f_a + \sum_{\substack{a =(i,j)\\ \in \onearcs}}\gamma_a =   \sum_{\substack{a =(j,i)\\ \in \arcs^\omega}} f_a + \sum_{\substack{a =(j,i)\\ \in \onearcs}} \gamma_a && \forall i \in \nodes \label{eq:P2_flow2}\\
	& 0 \leq \gamma_a  \leq \mpsol^\omega_{\map(a)} && \hspace*{-0.04cm} \forall a \in \onearcs \label{eq:P2_gammaUB} \\
	& f_a \geq 0 \quad && \forall a\in \arcs^\omega \label{eq:P2_fLB} 
\end{alignat}
\esubeq%

\begin{proposition} \label{lemma:form_equiv}
	Given  a scenario $\omega$ and binary first-stage decisions $\mpsol^\omega$,  formulations $\Primal{3}$ and $\AltPrimal{3}$ have the same optimal objective value. 
\end{proposition}

We remark that Proposition \ref{lemma:form_equiv} relies on binary $\mpsol^\omega$ and does not necessarily hold for fractional values of  $\mpsol^\omega$. The proof of Proposition \ref{lemma:form_equiv}, not only shows the equality of the objective functions of $\Primal{3}$ and $\AltPrimal{3}$ at optimality but gives a transformation between the solutions of these formulations. Thus given an optimal solution to one formulation, we can construct an optimal solution to the other. 

By Proposition \ref{lemma:form_equiv}, we can derive Benders cuts from the dual of \eqref{form:BDD2_primal2}:
\bsubeq\label{form:BDD2_dual2}
\begin{alignat}{2}
	\AltDual{3}: 	\max \ & \pi_\rootnode  - \pi_\term  - \sum_{a  \in \onearcs} z_a\mpsol^\omega_{\map(a)}  \\
	\text{s.t.} \ & \pi_i - \pi_j \leq 0 && \forall (i,j) \in \zeroarcs \\ 
	&  \pi_i - \pi_j \leq \costone{\map(a)} + \costtwo{\map(a)} \quad && \forall a =(i,j) \in \onearcs \label{eq:D2_piUB} \\
	&  \pi_i - \pi_j \leq z_a + \costtwo{\map(a)}\quad && \forall a =(i,j) \in \onearcs \label{eq:D2_pizUB} \\
	& z_a \geq 0 \quad && \forall a =(i,j) \in \onearcs \label{eq:D2_zbd} 
\end{alignat}%
\esubeq%
Again, without loss of generality, we assume $\pi_\term =0$ as the constraint of the primal associated with this variable is linearly dependent on the other constraints. 

\begin{proposition}
	Given a scenario $\omega$ and a binary first-stage solution $\mpsol^\omega$, $\AltPrimal{3}$ and \eqref{eq:altprob_obj}-\eqref{eq:sp1_yconstr} have the same optimal objective value.
\end{proposition}
%

\section{Benders Cuts}\label{sec:benders}

In this section, we will explain how to derive cuts and then strengthen the cuts derived from both BDD-based methods.
We then provide details of the integer L-shaped method, which is a well-known method in the stochastic programming literature and will serve as a point of comparison to the BDD-based decompositions. Next, we describe cuts based on the LP relaxation of the subproblem that can be added to any formulation. Finally, we compare the strength of the presented cuts.

\subsection{BDD-based Benders Cuts} \label{subsec:bddbenderscuts}

For a fixed scenario $\omega$ the Benders cuts that come from \bdd{2} are:
\begin{equation}  \label{eq:bdd1_cuts}
	\eta_\omega \geq \hat{\pi}_\rootnode - \sum_{a  \in \arcs^\omega} \sum_{i \in \viol_a} \hat{\beta}_{ai}(1 - \indvar_i).
\end{equation} 
For a scenario BDD $\BDD^\omega$ and master problem solution $\mpsol^\omega$ we calculate the dual solutions using a bottom-up shortest-path algorithm, as described in \citep{lozano2018bdd}. The terminal has dual value $\pi_\term = 0$ and each subsequent BDD node $\node$ has dual value $\pi_\node$, the length of the shortest path from $\node$ to $\term$. The $\beta_{ai}$ variables are only on those arcs with capacities, for such an arc $a = (u,v)$ if $\pi_u - \pi_v$ is shorter than the arc length, that is, the shortest path does not lie on $a$, then $\beta_{ai} = 0$ for all $i \in \viol_a$. 
Otherwise,  if $a \in \zeroarcs$ we have $\pi_u - \pi_v > 0$, so for  some arbitrary $\bar{i} \in \viol_a$ where $\mpsol^\omega_{\bar{i} }= 1$ we set $\beta_{a\bar{i}} = \pi_u - \pi_v $ and $\beta_{ai} = 0$ for all $i \in \viol_a$ where $i \neq \bar{i} $. Similarly, if $a \in \onearcs$ we have $\pi_u - \pi_v - \secondstagecost{\map(a)}> 0$, then we set $\beta_{a\bar{i}} = \pi_u- \pi_v -  \secondstagecost{\map(a)}$ for an arbitrary $\bar{i} \in \viol_a$ where $\mpsol^\omega_{\bar{i}}= 1$, and set  $\beta_{ai} = 0$ for all $i \in \viol_a$ where $i \neq \bar{i} $. These dual values can be computed in linear time with respect to the number of arcs in the BDD.

\begin{example} \label{ex:bdd1_cuts}
	
	Recall, in Example \ref{ex:bdd1}, we built the BDD for the scenario graph in Figure \ref{fig:smwds_graph}, we will now use this BDD to compute a cut for a master problem solution $\hat{\xvar} = [0,0,1,0,0]$. This solution has $\xvar_2 = 1$, thus any arc where $\xvar_2$ appears in all the capacity expressions will not be interdicted. 
	In Figure \ref{fig:bdd1_cuts} those arcs which still have capacity $0$ are highlighted. We are then able to compute the dual values $\pi$ using the bottom up shortest path algorithm and obtain $\pi_\rootnode =1$. The dual values of each node are given inside the node in Figure \ref{fig:bdd1_cuts}. 
	Finally, to derive the cut we must compute the 
	values of $\beta_{ai}$. We observe that to have $\beta_{ai} \neq 0$ we must have $\mpsol^\omega_i = 1$, in our SMWDS example this means that the arc capacity is $0$, since the indicator is $0$ if neither a vertex nor its neighbours are selected.  
	Thus we consider only the highlighted arcs in Figure \ref{fig:bdd1_cuts}, which we will denote from left to right $a_1$ and $a_2$. For both arcs we have $\pi_u - \pi_v = 1 > 0$, and we only have one indicator expression $x_0 + x_1+x_3$ on each arc, so we will ignore the indices $i \in \viol_a$ for ease of presentation. We can set $\hat{\beta}_{a_1i} = \hat{\beta}_{a_2i} = \pi_u - \pi_v =1$ and $\hat{\beta}_{ai} = 0$ for all other capacitated arcs $a$ and sets $\viol_a$. We have now obtained the following cut	which will be added to $\MP$:
	\begin{equation} \label{eq:bdd1_cut_ex}
		\eta \geq 1 - 2(\xvar_{0} + \xvar_{1} + \xvar_{3} )
	\end{equation}
\end{example}

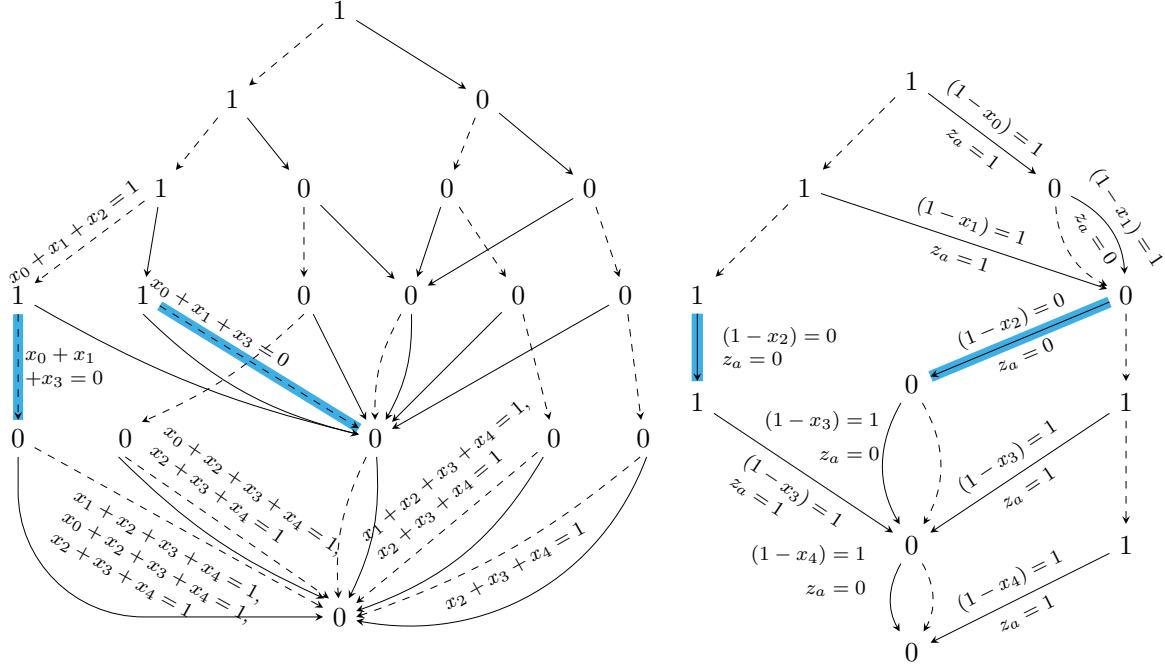
\begin{figure}[h]
\scalebox{0.95}{
	\begin{subfigure}[t]{0.5\textwidth}
		\centering
		\hspace*{-0.4cm}		
		\begin{tikzpicture}[	main_node/.style={}, node distance = 1cm and 1cm,  > = stealth, 
			]
			
			\node[main_node] (s) at  (-0.5,0) {$1$};
			
			\node[main_node] (l1n1) at (-2,-1.25) {$1$};
			\node[main_node] (l1n2) at  (1.5,-1.25) {$0$};
			
			\node[main_node] (l2n1) at (-3,-2.5) {$1$};
			\node[main_node] (l2n2) at (-1,-2.5) {$0$};
			\node[main_node] (l2n3) at (1,-2.5) {$0$};
			\node[main_node] (l2n4) at (3,-2.5) {$0$};
			
			\node[main_node] (l3n0) at (-5,-4) {$1$};
			\node[main_node] (l3n1) at (-3.25,-4) {$1$};
			\node[main_node] (l3n2) at (-1,-4) {$0$};
			\node[main_node](l3n3) at (0.5,-4) {$0$};
			\node[main_node] (l3n4) at (2,-4) {$0$};
			\node[main_node] (l3n5) at (3.5,-4) {$0$};

			\node[main_node] (l4n0) at (-5,-6) {$0$};
			\node[main_node] (l4n1) at (-3.5,-6) {$0$};
			\node[main_node](l4n2) at (0,-6) {$0$};
			\node[main_node] (l4n3) at (2.5,-6) {$0$};
			\node[main_node] (l4n4) at (3.75,-6) {$0$};
			
			\node[main_node](t) at (-.5,-8.5) {$0$};

			\draw[dashed, ->] (s) to (l1n1);
			\draw[->] (s) to (l1n2);
			
			\draw[dashed, ->] (l1n1) to  (l2n1);
			\draw[->] (l1n1) to (l2n2);
			\draw[dashed, ->] (l1n2) to []  (l2n3);
			\draw[->] (l1n2) to []  (l2n4);

			\draw[dashed, ->] (l2n1) to node[sloped, above, xshift=-0.1cm]{\scriptsize $x_0+x_1+x_2=1$} (l3n0);
			\draw[->] (l2n1) to (l3n1);
			\draw[dashed, ->] (l2n2) to (l3n2);
			\draw[->] (l2n2) to (l3n3);
			\draw[dashed, ->] (l2n3) to (l3n4);
			\draw[->] (l2n3) to (l3n3);
			\draw[dashed, ->] (l2n4) to (l3n5);
			\draw[->] (l2n4) to (l3n3);
			
			\draw[-, line width=4pt, color=CornflowerBlue] (l3n0) to (l4n0);
			\draw[dashed, ->] (l3n0) to [bend right=0] node[right, xshift =-0.25cm,rotate=0]{\scriptsize \begin{tabular}{l}
					$x_0+x_1$\\
					$+x_3=0$
			\end{tabular} }  (l4n0);
			\draw[->] (l3n0) to [bend right=8]  (l4n2);

			\draw[-, line width=4pt, color=CornflowerBlue] (l3n1) to [bend right=0]  (l4n2);
			\draw[dashed, ->] (l3n1) to [bend right=0] node[sloped, above, xshift=-.8cm, rotate=0]{\scriptsize  $x_0+x_1+x_3=0$}  (l4n2);
			\draw[->] (l3n1.south) to [bend right= 15]  (l4n2);
			
			\draw[dashed, ->] (l3n2.south) to  [bend left=10]  (l4n1.45);
			\draw[->] (l3n2) to (l4n2);
			
			\draw[dashed, ->] (l3n3) to [bend right=15]  (l4n2);
			\draw[->] (l3n3) to [bend left=15]  (l4n2);
			
			\draw[dashed, ->] (l3n4) to (l4n3);
			\draw[->] (l3n4) to (l4n2);
			\draw[dashed, ->] (l3n5) to (l4n4);
			\draw[->] (l3n5) to (l4n2);

			\draw[dashed, ->] (l4n0) to [bend right=0] node[sloped,below,rotate= 0 ]{\scriptsize
				\begin{tabular}{l}
					$x_1+x_2+x_3+x_4=1$,\\$x_0+x_2+x_3+x_4=1$,\\ $x_2+x_3+x_4=1$
			\end{tabular}}   (t);
			\draw[rounded corners=50pt,->] (l4n0.south) |- (t.west);

			\draw[dashed, ->] (l4n1.south) to [bend left=0] node[sloped,above, rotate=0 ]{\scriptsize
				\begin{tabular}{l}
					$x_0+x_2+x_3+x_4=1$,\\ $x_2+x_3+x_4=1$
			\end{tabular}}   (t.120);
			\draw[->] (l4n1.250) to [bend right=10]  (t.130);
			
			\draw[dashed, ->] (l4n2) to [bend right=15]  (t);
			\draw[->] (l4n2) to [bend left=15]  (t);
			
			\draw[dashed, ->] (l4n3) to [bend right=0] node[sloped, above, rotate= 0, xshift =0.5cm]{\scriptsize \begin{tabular}{l}
					$x_1+x_2+x_3+x_4=1$,\\ $x_2+x_3+x_4=1$
			\end{tabular} }   (t);
			\draw[->] (l4n3) to [bend left=20]  (t);
			
			\draw[dashed, ->] (l4n4.260) to [bend left=10] node[sloped, below, rotate=0]{\scriptsize $x_2+x_3+x_4=1$}   (t.east);
			
			\draw[->] (l4n4.280) to [bend left=400]  (t);
			
		\end{tikzpicture}
		\caption{ \bdd{2} subproblem solution for the SMWDS instance in Example \ref{ex:bdd1_cuts} at  \MP \ solution $\hat{\xvar} = [0,0,1,0,0]$. The nodes contain their dual values, and the highlighted arcs have capacity $0$. \label{fig:bdd1_cuts}}
	\end{subfigure}
	~
	\begin{subfigure}[t]{0.5\textwidth}
		\centering
		\hspace*{1.3cm}		
		\begin{tikzpicture}[main_node/.style={},
			node distance = 1cm and 1cm,  > = stealth, 
			shorten > = 1pt, 
			]
			
			\node[main_node] (s) at  (0,0) {$1$};
			
			\node[main_node] (l1n1) at (-1.5,-1.5) {$1$};
			\node[main_node] (l1n2) at  (2,-1.5) {$0$};
			
			\node[main_node] (l2n1) at (-3,-3) {$1$};
			\node[main_node] (l2n2) at (3,-3) {$0$};

			\node[main_node] (l3n1) at (-3,-4.5) {$1$};
			\node[main_node](l3n2) at (0,-4.25) {$0$};
			\node[main_node] (l3n3) at (3,-4.5) {$1$};

			\node[main_node](l4n1) at (0,-6.5) {$0$};
			\node[main_node] (l4n2) at (3,-6.5) {$1$};

			\node[main_node](t) at (0,-8) {$0$};

			\draw[dashed, ->] (s) to (l1n1);
			\draw[->] (s) to node[sloped, above, rotate=0]{ \scriptsize$(1-\xvar_{0})=1$} node[sloped, below, rotate=0]{ \scriptsize$z_a =1$}  (l1n2);
			
			\draw[->] (l1n1) to node[sloped, above, rotate= 0]{ \scriptsize$(1-\xvar_{1})=1$}  node[sloped, below, rotate=0]{ \scriptsize$z_a =1$}  (l2n2);
			\draw[dashed, ->] (l1n1) to (l2n1);
			
			\draw[dashed, ->] (l1n2) to [bend right=30]  (l2n2);
			\draw[->] (l1n2) to [bend left=30]  node[sloped, above, rotate= 0]{ \scriptsize$(1-\xvar_{1})=1$} node[sloped, below, rotate=0]{ \scriptsize$z_a =0$}(l2n2);

			\draw[-, line width=4.5pt, color=CornflowerBlue ] (l2n1) to (l3n1);
			\draw[->] (l2n1) to node[right, rotate= 0]{ \scriptsize\begin{tabular}{l}
			$(1-\xvar_{2})=0$\\
				$z_a =0$
				\end{tabular}
				}  (l3n1);
			\draw[-, line width=4.5pt, color=CornflowerBlue ] (l2n2) to (l3n2);
			\draw[->] (l2n2) to node[above, sloped,rotate= 0]{ \scriptsize$(1-\xvar_{2})=0$}  node[sloped, below, rotate=0]{ \scriptsize$z_a =0$}(l3n2);
			\draw[dashed, ->] (l2n2) to (l3n3);

			\draw[->] (l3n1) to node[sloped, below, rotate= 0]{ \scriptsize \begin{tabular}{l}
					$(1-\xvar_{3})=1$\\ 
					$z_a =1$
				\end{tabular}	} (l4n1);			
			\draw[dashed, ->] (l3n2) to [bend left=30]  (l4n1);
			\draw[->] (l3n2) to [bend right=30]  node[left, rotate= 0, yshift =0.4cm, xshift=0.3cm]{ \begin{tabular}{r}
					\scriptsize$(1-\xvar_{3})=1$ \\
					\scriptsize$z_a =0$
			\end{tabular}}   (l4n1);
			\draw[dashed, ->] (l3n3) to  (l4n2);
			\draw[->] (l3n3) to node[sloped, above, rotate= 0]{ \scriptsize$(1-\xvar_{3})=1$}node[sloped, below, rotate=0]{ \scriptsize$z_a =1$}  (l4n1);

			\draw[dashed, ->] (l4n1) to [bend left=30]  (t);
			\draw[->] (l4n1) to [bend right=30] node[left, rotate= 0, yshift =0.4cm]{
				\begin{tabular}{r}
				\scriptsize$(1-\xvar_{4})=1$\\
					\scriptsize$z_a =0$
				\end{tabular} }    (t);		
			\draw[->] (l4n2) to [] node[above, sloped, rotate= 0]{ \scriptsize$(1-\xvar_{4})=1$} node[sloped, below, rotate=0]{ \scriptsize$z_a =1$}   (t);

		\end{tikzpicture}
		\caption{\bdd{3} subproblem solution for the SMWDS instance in Example \ref{ex:bdd2_cuts} at \MP \ solution $\hat{\xvar} = [0,0,1,0,0]$. The nodes contain their dual values, and the highlighted 1-arcs have cost $0$.  \label{fig:bdd2_cuts}}
	\end{subfigure}
	}
	\caption{BDD subproblem solutions for the SMWDS instance in Example \ref{ex:bdd1_cuts}  and in Example \ref{ex:bdd2_cuts}.}
\end{figure}

For a fixed scenario $\omega$ the Benders cuts that come from \bdd{3} are:
\begin{equation} \label{eq:bdd2_cuts}
	\eta_\omega \geq \hat{\pi}_\rootnode   - \sum_{a  \in \onearcs} \hat{z}_a\indvar_{\map(a)} .
\end{equation} 
The procedure to obtain the duals for the \bdd{3} subproblems is very similar to that for  the \bdd{2} subproblems. For a scenario BDD $\BDD^\omega$ and master problem solution $\mpsol^\omega$, we calculate the dual solutions using a bottom-up shortest-path algorithm, starting at the terminal node which has $\pi_\term = 0$. For each subsequent node $\node$, $\pi_\node$ is the length of the shortest $\node$-$\term$ path. 
As seen in \eqref{form:BDD2_dual2}, the $z_a$ variables correspond only to arcs in $\onearcs$. By constraints \eqref{eq:D2_pizUB} and $\eqref{eq:D2_zbd}$, if $\pi_u - \pi_v - \costtwo{\map(a)} \leq 0 $, 
we can set $z_a = 0$. Otherwise, we set $ z_a = \pi_u - \pi_v -\costtwo{\map(a)} $. In this way, $z_a + \costtwo{\map(a)} $ can be seen as the change in path length between $u$ and $v$.

\begin{example} \label{ex:bdd2_cuts}
	We will compute a cut at master problem solution $\hat{\xvar} = [0,0,1,0,0]$ for the BDD built in Example \ref{ex:bdd2}. As seen by the highlighted arcs in Figure \ref{fig:bdd2_cuts}, at this solution we have two arcs in $\onearcs$ that now have cost $0$. Using the parametrized arc costs we compute the shortest path dual values $\pi$ which can be seen inside the nodes in Figure  \ref{fig:bdd2_cuts}. To find the values of the $z_a$ variables, we will compute $\pi_u -\pi_v - \costtwo{\map(a)}$ for all $a = (u,v) \in \onearcs$ and set $z_a$ accordingly. We remark that for this instance we have $\costtwo{\map(a)} = 0$ for all arcs $a$. These values can be seen next to the arcs in Figure \ref{fig:bdd2_cuts}. 
	Taking $\xvar_{\map(a)}$ as the indicator we now have the cut:
	\begin{equation} \label{eq:bdd2_cut_ex}
		\eta \geq 1 - (\xvar_{0} + \xvar_{1} + 2\xvar_{3} + \xvar_{4}).
	\end{equation}
	We remark that this cut differs from the cut found in Example \ref{ex:bdd1_cuts} despite being derived for the same subproblem instance at the same master problem solution.
\end{example}

We remark that we are able to apply the cut strengthening strategies proposed by  \cite{lozano2018bdd} to both cuts  \eqref{eq:bdd1_cuts} and \eqref{eq:bdd2_cuts}. The first strategy relies on partitioning $\arcs^\omega$ into disjoint sets of arcs and selecting the maximal value of $\hat{\beta}_{ai}$ or $\hat{z}_a$ from each set to appear in the cut.
First remark that the BDD arc layers,  $\arclayer{j}$ for all $j \in \{1, \hdots, \yvarindex\}$ are disjoint and partition $\arcs^\omega$, and  that for each $a \in \arcs$ we must have $\viol_a \subseteq  \{1,\hdots, \yvarindex\}$. 
In \bdd{2}, for all $j \in \{1, \hdots, \yvarindex\}$ and  $i \in \{1, \hdots, m_1\}$ let $\betamax_{ji} =  \max_{a \in \arclayer{j}}\{\hat{\beta}_{ai} : i \in \viol_a\}$. 
Similarly, in \bdd{3}, remark that because we have assumed the natural ordering of the variables, for each arc $a$ in a BDD arc layer  $\arclayer{j}$, we have $\map(a) = j $ for all $j \in \{1, \hdots, \yvarindex\}$. In \bdd{3}, for all $j \in \{1, \hdots, \yvarindex\}$ define $\zmax_{j} = \max_{a \in \arclayer{j}}\{\hat{z}_{a}\}$, {\cred we remark that $m_1 = \yvarindex$ for \bdd{3}, since every logical expression modifies one second-stage cost coefficient.}
Then the cut
\begin{equation}  \label{eq:bdd1_cut_stren}
	\eta_\omega \geq \hat{\pi}_\rootnode - \sum_{j \in \{1, \hdots, \yvarindex\}} \sum_{i \in \{1, \hdots, m_1\}} \betamax_{ji} (1 - \indvar_i)
\end{equation} 
is at least as strong as \eqref{eq:bdd1_cuts} and the cut
\begin{equation} \label{eq:bdd2_cut_stren}
	\eta_\omega \geq \hat{\pi}_\rootnode   - \sum_{j \in \{1, \hdots, m_1\}} \zmax_j \indvar_{j} .
\end{equation} 
is at least as strong as \eqref{eq:bdd2_cuts}. Henceforth we assume the use of these strengthened cuts for \bdd{2} and \bdd{3} unless otherwise stated. 

The second strengthening strategy solves a secondary optimization problem to help compute lower bounds on the second-stage objective function which are then used to strengthen the cut. We refer the reader to  \citep{lozano2018bdd} for further details on the cut strengthening strategies.

\begin{example}\label{ex:str_cuts}
	For the scenario in Figure \ref{fig:smwds_graph}, the strengthened version of   \eqref{eq:bdd1_cut_ex}  is 
	\begin{equation} \label{eq:bdd1_cut_ex_str}
		\eta \geq 1 - (\xvar_{0} + \xvar_{1} + \xvar_{3} )
	\end{equation}
	and  the strengthened version of the cut \eqref{eq:bdd2_cut_ex} is
	\begin{equation} \label{eq:bdd2_cut_ex_str}
		\eta \geq 1 - (\xvar_{0} + \xvar_{1} + \xvar_{3} +  \xvar_{4}).
	\end{equation}
\end{example}

\subsection{Integer L-shaped Method}

The integer L-shaped method is the generic alternative decomposition method to the BDD-based decompositions. It solves an integer program both in the master problem and in the subproblems. The master problem, $\MP$, is the same as for the BDD-based decompositions. The subproblem for scenario $\omega$ is:
\bsubeq\label{form:lshape_sp}
\begin{alignat}{2}
	\min \ & \secondstagecost{}^{\top}{\yvar} \\
	\text{s.t.} \ & {\yvar} \in \mathcal{Z}(\mpsol^\omega, \omega) \\
	&  {\yvar} \in \mathcal{Y}(\omega) \subseteq  \{0,1\}^{\yvarindex}. 		
\end{alignat}%
\esubeq %
where $\mathcal{Z}(\mpsol^\omega, \omega)$ are those constraints that depend on the first-stage decisions and correspond to the indicator constraints for the BDD-based methods. Since we have binary variables in the subproblems, we use logic-based Benders cuts. Given a scenario $\omega$ and a master problem solution $\mpsol^\omega$, let $\hat{\spobj}_\omega$ be the optimal objective value of \eqref{form:lshape_sp}. The standard cuts are called ``no-good'' cuts and take the following form:
\begin{equation}\label{eq:lshape_cuts}
	\eta_\omega \geq  \hat{\spobj}_\omega - \hat{\spobj}_\omega\left(\sum_{\substack{	i \in \{1,\hdots, m_1\}\\ \mpsol^\omega_i = 1}} (1-\indvar_i) + \sum_{\substack{	i \in \{1,\hdots, m_1\}\\ \mpsol^\omega_i = 0}}\indvar_i \right).
\end{equation}
We remark that based on problem-specific structure, these cuts can be strengthened. We demonstrate this in Section \ref{sec:smwds} on the SMWDS problem.

\subsection{Pure Benders Cuts}

To strengthen any of the decomposition methods, we have incorporated pure Benders (PB) cuts, i.e., cuts based on the solution of the LP relaxation of  \eqref{form:lshape_sp}. We pass a solution from $\MP$ to \eqref{form:lshape_sp} and solve without integrality constraints. We then use the duals of this solution to add to $\MP$ a traditional Benders cut of the form
\begin{equation}\label{eq:pb_cuts}
	\eta_\omega \geq \mathcal{T}_{\mpsol^\omega}(\indvar),
\end{equation}
 where the function $\mathcal{T}_{\mpsol^\omega}(\indvar)$ returns an affine expression of $\indvar$.

\subsection{Cut Strength Comparison}

We have the following findings when we compare the strength of the proposed cuts.
We remark that the following Propositions \ref{prop:bdd1_lshape} and \ref{prop:bdd2_lshape} rely on non-negative objective coefficients in the second-stage problems,  which is quite common in applications as we are often minimizing costs.

\begin{proposition}  \label{prop:bdd1_lshape}
	Given a scenario $\omega$ and an $\MP$ solution $\mpsol^\omega$, if the subproblem objective coefficients are non-negative, i.e.,  $\secondstagecost{i} \geq 0$ for all $i \in \{1,\hdots, \yvarindex\}$,  	
the strengthened \bdd{2} cuts \eqref{eq:bdd1_cut_stren} are at least as strong as the integer L-shaped cuts \eqref{eq:lshape_cuts}.
\end{proposition}

\begin{proposition} \label{prop:bdd2_lshape}
	Given a scenario $\omega$ and an $\MP$ solution $\mpsol^\omega$, 
	if the subproblem objective coefficients are non-negative, i.e.,  $\costtwo{i}\geq 0$ for all $i \in \{1,\hdots, \yvarindex\}$,  	
	the strengthened \bdd{3} cuts \eqref{eq:bdd2_cut_stren} are at least as strong as the integer L-shaped cuts \eqref{eq:lshape_cuts}.
\end{proposition}

We have also shown the two kinds of BDD cuts and the PB cuts are incomparable, whose proof is in \ref{app:cuts}. These cuts being incomparable  indicate that it may be beneficial to add cuts from multiple classes at each iteration. Before we compare the strength of the cuts in practice on the SMWDS application in Section \ref{sec:smwds}, we extend the two-stage problem to a risk-averse setting.

\section{Incorporation of Conditional Value at Risk} \label{sec:cvar}

Incorporating risk measures into SPs is a well-established practice in the literature, however there is limited work when the second-stage problem has integer variables \citep{schultz2006,vanbeesten2020}. To our knowledge, we propose the first decomposition method for 2SP with the CVaR risk measure and integer second-stage variables. We will use BDDs or integer L-shaped cuts to convexify the second-stage value function. We remark that this decomposition is rooted in the work of \cite{Noyan2012} who considers a decomposition approach for risk-averse 2SP with \textit{continuous} second-stage variables. 

{\credrev
Formally, the value-at-risk (VaR) of a random variable $Z$ at level $\alpha$ is the $\alpha$-quantile
\begin{equation}
	\var(Z) =  \inf_{\zeta \in \R} \{ \zeta : \Prob[ Z \leq \zeta] \geq \alpha\}.
\end{equation}
$\var$ is the smallest value $\zeta$ such that with probability $\alpha$ the random variable $Z$ will not exceed $\zeta$ \citep{rockafellar2000optimization}. Thus $\var$ gives an upper bound on the random variable $Z$ that will only be exceeded with a probability of $(1-\alpha)$. However, VaR does not quantify the severity of exceeding this upper bound. This is why CVaR, also called expected shortfall, is a useful risk measure. Moreover, CVaR is a coherent risk measure, which helps keep optimization problems computationally tractable.

The intuitive definition of CVaR, for random variable $Z$ with a continuous distribution or a general distribution where $\var(Z)$ is not an atom of the distribution of $Z$, is the conditional expectation that the random variable $Z$ exceeds $\var$ , i.e.,
\begin{equation}
	\cvar(Z) = \expect[Z : Z \geq \var(Z)].
\end{equation} 
Thus it measures the severity of exceeding $\var$. More formally and generally, letting $[a]_+ = \max\{0,a\}$ for $a \in \R$, CVaR of a random variable $Z$ at confidence level $\alpha \in (0,1]$ is defined as
\begin{equation}
	\cvar(Z) = \inf_{\zeta \in \R} \left\{\zeta + \frac{1}{1-\alpha} \expect([Z-\zeta]_+)\right\}.
\end{equation}
For more details, we refer the readers to \citep{rockafellar2002conditional}.}

We will consider the two-stage mean-risk stochastic program
\begin{equation}\label{eq:cvar_gen}
	\min_{{x} \in \mathcal{X}} \ (1+\lambda){\xweights^\top \xvar} + \expect[\mathcal{Q}({\xvar}, \omega)] + \lambda\cvar(\mathcal{Q}({\xvar}, \omega))
\end{equation}
where $\lambda \geq 0$ is a risk coefficient and $\cvar$ is the CVaR at level $\alpha \in (0,1]$. Thus we take the $\cvar$ of the recourse function $\mathcal{Q}({\xvar}, \omega)$, substituting its objective function for the random variable.
In the decomposition approaches presented by  \cite{Noyan2012}, the second-stage problem is an LP, thus convexity of $\mathcal{Q}({\xvar}, \omega)$ in $\xvar$ for all $\omega \in \Omega$ follows easily, and they use LP duality to approximate the value of the mean-risk function of the recourse cost. The key to our decomposition to solve \eqref{eq:cvar_gen} is the observation that we have already convexified the approximation of the recourse problem using the BDD-based cuts and the integer L-shaped cuts. 

%

\subsection{Decomposition Algorithm}

In this algorithm, we solve the second-stage problems and generate cuts from the scenario BDDs as in Section \ref{sec:benders}. We will take the optimal recourse function value from the BDD subproblem and use it in the $\cvar$ expression, therefore the CVaR value of the recourse cost for $\bdd{3}$ becomes
 {\cred
\begin{equation}  \label{eq:cvar_recourse}
\sum_{\omega \in \Omega} p_\omega \inf_{\zeta\in \R} \left\{\zeta + \frac{1}{1-\alpha} \left [  \left( \hat{\pi}^\omega_\rootnode   - \sum_{j \in \{1, \hdots, m_1\}} \hat{z}^{\max, \omega}_j \indvar_{j}  \right) - \zeta\right]_+\right\},
\end{equation}
 where $\hat{\pi}^\omega_\rootnode   - \sum_{j \in \{1, \hdots, m_1\}} \hat{z}^{\max, \omega}_j \indvar_{j}$ is the optimal recourse cost in scenario $\omega$.}
 
In the master problem of this decomposition, as in $\MP$, we introduce the variables $\eta_\omega$ to model the objective value of the recourse problem and we introduce new variables $\theta_\omega$ to model the value of CVaR for each scenario $\omega \in \Omega$. We remark that we can now replace the optimal recourse cost with $\eta_\omega$ in \eqref{eq:cvar_recourse}, since  because we are minimizing at any optimal solution we will have $ \eta_\omega =  \hat{\pi}^k_\rootnode   -\sum_{j \in \{1, \hdots, m_1\}} \hat{z}^{\max k}_j \indvar_{j}  $ for the $k^{\text{th}}$ cut generated. This is in contrast to \citep{Noyan2012}, which uses the original optimal recourse cost in their model.
We also define variables $\zeta^k$ to model  $\var$ of the recourse cost and $\nu^{\omega k}$ to model the function $[\eta_\omega - \zeta^k]_+, \forall \omega \in \Omega$ in cut $k$. 

We can now formulate the master problem after $K$ cuts have been generated as
\bsubeq\label{form:cvar_mp}
\begin{alignat}{2}
	\min \ & (1+\lambda){\xweights^\top \xvar}  + \sum_{\omega \in \Omega} p_\omega\eta_\omega + \lambda \sum_{\omega \in \Omega} p_\omega\theta_\omega   \\
	\text{s.t.} \ &  {\xvar} \in \mathcal{X} \\ 
	&{\cred  \indvar_i =	\indicator(\cond) } &&  \forall i =1,\hdots,m_1 \\
	& \eta_\omega \geq \hat{\pi}^k_\rootnode   - \sum_{j \in \{1, \hdots, m_1\}} \hat{z}^{\max k}_j \indvar_{j}  
	\quad && \forall \omega \in \Omega, k=1,\hdots,K \label{eq:opt_recourse}\\
	& \theta_\omega\geq \zeta^k + \frac{1}{1-\alpha}  \nu^{\omega k} \quad && \forall \omega \in \Omega, k=1,\hdots,K \label{eq:cvar_cut}\\
	& 		\nu^{\omega k} \geq \eta_\omega - \zeta^k \quad && \forall \omega \in \Omega, k=1,\hdots,K \label{eq:cvar_cut2}\\
	& 		\nu^{\omega k} \geq 0\quad && \forall \omega \in \Omega, k=1,\hdots,K \label{eq:cvar_nubd} \\
	& \zeta^k \in \R &&  k=1,\hdots,K \label{eq:cvar_zetabd}
\end{alignat}%
\esubeq%
where the optimality cuts \eqref{eq:opt_recourse} are generated using the \bdd{3} procedure. These cuts could be replaced with those from the \bdd{2} procedure or with integer L-shaped cuts derived from the original second-stage problem. The complete CVaR decomposition algorithm can be found  in \ref{app:cvar}

\section{Computational Results on a Novel Application } \label{sec:smwds}

\subsection{The Stochastic Minimum Weight Dominating Set Problem}
Recall the SMWDS problem described in Section \ref{subsubsec:smwds}, where we have binary decision variables $\xvar_v = 1 $ if the vertex $v \in \vertexset$ is assigned to the dominating set in the first stage, and $\yvar_v = 1 $ if the vertex $v$ is assigned to the dominating set in the second stage under equally probable scenario $\omega \in \scenset$. 
Given a scenario set  $\scenset$, the deterministic equivalent for SMWDS is formulated as:
\bsubeq\label{form:EF}
\begin{alignat}{2}
	\min \ &  \sum_{v \in \vertexset}\xweights_v \xvar_v  +\frac{1}{ |\scenset|} \sum_{\omega \in \scenset} \sum_{v \in \vertexset} \secondstagecost{v} \yvar_v\\
\text{s.t.} \ & \yvar_v + \sum_{u \in \neighb_\omega'(v)} \yvar_u  \geq 1 - {\xvar}_v -  \sum_{u \in \neighb_\omega'(v)}  {\xvar}_u   \quad && \forall \ v \in \vertexset'_\omega \label{eq:cover}\\ 
		&\yvar_v  \in  \{0,1\} &&\forall \ v \in \vertexset'_\omega \label{eq:xdomain}\\
	&  \xvar_v \in  \{0,1\} &&\forall \ v \in \vertexset \label{EF:ydomain} 
\end{alignat}%
\esubeq %
Given this formulation, the constraints \eqref{eq:sp1_logical} become
\begin{equation}
	\xvar_v = 1 \implies  \secondstagecost{v} = 0, \quad \forall \ v \in \vertexset'_\omega,    \omega \in \scenset 
\end{equation}
which clearly do not need indicator variables, and the constraints \eqref{eq:sp2_logical2} become
\begin{equation}\label{eq:smwds_cap}
	\xvar_v +   \sum_{u \in \neighb_\omega'(v)}  \xvar_u  =0 \implies \yvar_v + \sum_{u \in \neighb_\omega'(v)} \yvar_u  \geq 1, \quad \forall \ v \in \vertexset'_\omega,    \omega \in \scenset.
\end{equation}
Constraints \eqref{eq:smwds_cap} also do not need indicator variables  setting $\mpsol^\omega_{v} = \xvar_v +  \sum_{u \in \neighb_\omega'(v)}  \xvar_u$, this ensures that if a vertex or its neighbour is selected in the first stage and that vertex exists in the scenario, then the cover constraints \eqref{eq:cover} become redundant.

We build the subproblem BDDs as in Examples \ref{ex:bdd1} and \ref{ex:bdd2}, and the cuts for the BDD-based algorithms are exactly as seen in \eqref{eq:bdd1_cut_stren} and \eqref{eq:bdd2_cut_stren} replacing the variables $\indvar$ with the respective indicator expressions. We will also compare the BDD-based algorithms with the integer L-shaped approach, where we solve the subproblems as IPs. The strengthened integer L-shaped cuts \eqref{eq:lshape_cuts} for this problem are:
\begin{equation} \label{eq:smwds_lshape}
	\eta_\omega \geq \mathcal{Q}(\hat{\xvar}, \omega) - \mathcal{Q}(\hat{\xvar}, \omega)\sum_{{v \in \vertexset,  \hat{\xvar}_v =0}} \xvar_v.
\end{equation}

Finally, we also compare the solution methods with and without the PB cuts which are given by relaxing the binary constraints, removing the upper bounds on the $\yvar$ variables, and solving the subproblem as an LP. We then take traditional Benders cut from the dual solution. Letting $\hat{\dualvar}$ be the optimal dual variables associated with the constraints \eqref{eq:cover}, then the PB  optimality cuts are:
\begin{equation}
\eta_\omega \geq  \sum_{v \in \vertexset_\omega'} \hat{\dualvar}_v \left( 1 - {\xvar}_v -  \sum_{u \in \neighb_\omega'(v)}  {\xvar}_u\right).
\end{equation}
Given these methods, the extension to the risk averse setting is natural, using \eqref{form:cvar_mp} as the master problem and adding the appropriate CVaR and optimality cuts given the solution method. We next compare these methods on SMWDS instances.

\subsection{Computational Results}

We compare both BDD-based decomposition algorithms against the integer L-shaped method, with and without PB cuts. We implement the algorithms in C++ and use the CPLEX 12.10 solver, the experiments were conducted on a single thread on a Linux workstation with 3.6GHz Intel Core i9-9900K CPUs and 40GB RAM with a solution time limit of $3600$ seconds after the BDDs are generated. We implement the decomposition algorithms as a branch-and-cut procedure, using CPLEX lazy callbacks to add optimality and PB cuts at integer feasible solutions. We note that PB cuts could also be added at all feasible solutions. Finally, we use CPLEX heuristic callbacks, which are called at every LP feasible solution that also has a better relaxation value than the best available integer solution. The callback then constructs an integer feasible solution to use as a new incumbent. 

The test set is comprised of $200$ instances, with $|\vertexset| \in \{30,35,...,50\}$ and edge density $\in \{0.2,0.4,0.6,0.8\}$, randomly generated based on benchmarks from the deterministic MWDS literature \citep{jovanovic2010ant}. 
{\credrev Details of the instance generation can be found in \ref{sec:instancegen}.}
We conduct sample average approximation (SAA) analysis and conclude that to achieve a worst-case optimality gap below $1\%$, we must use at least $800$ scenarios and proceed to use $850$ scenarios in our experiments. The SAA analysis table can be found in \ref{app:saa}.

\begin{figure}[b] 
	\centering
	\begin{subfigure}[t]{0.48\textwidth}
		\includegraphics[width=\textwidth]{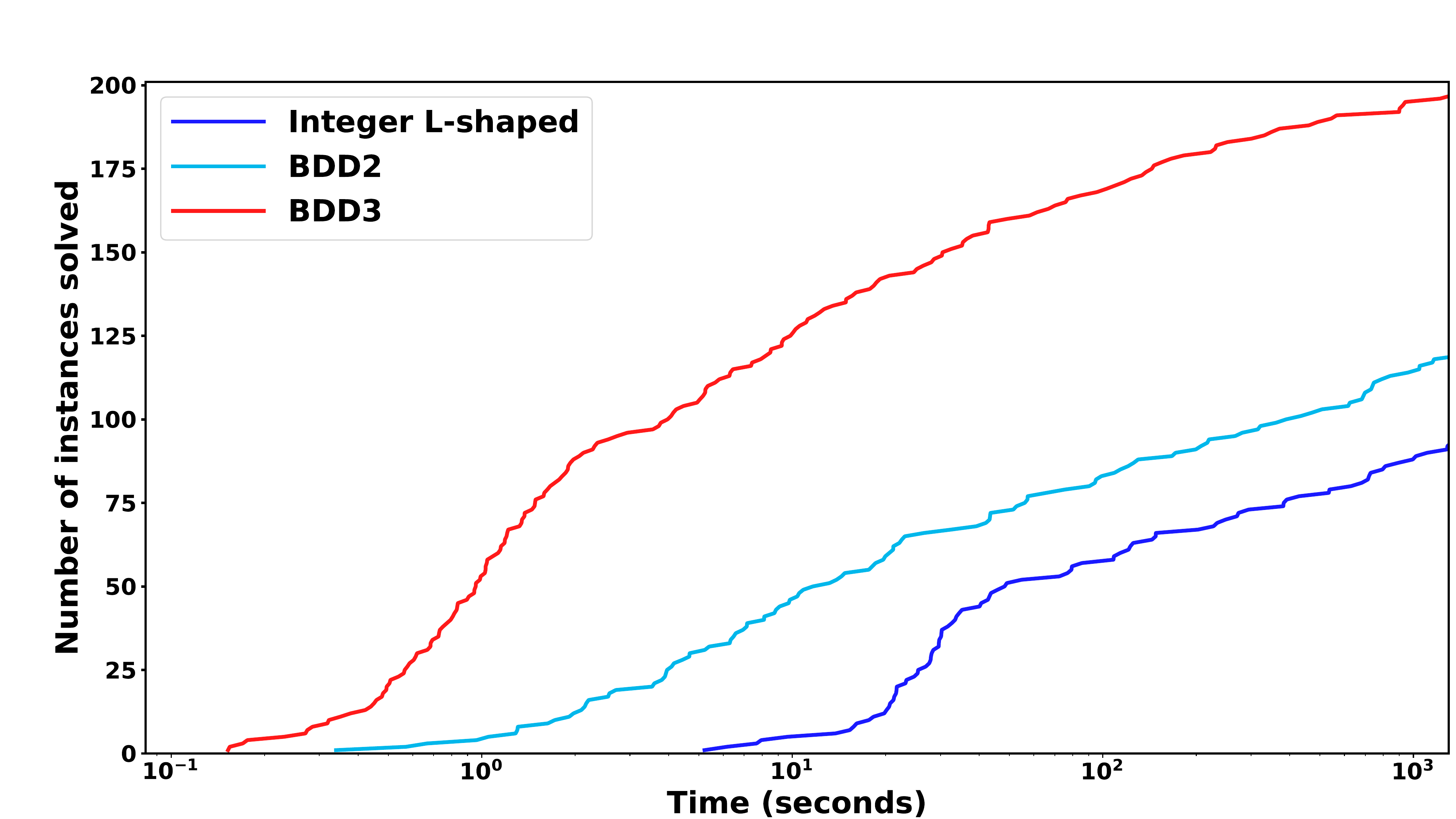}
		\caption{Methods without PB cuts.}
		\label{fig:smwds_nopb}
	\end{subfigure}
~
	\begin{subfigure}[t]{0.48\textwidth}
		\includegraphics[width=\textwidth]{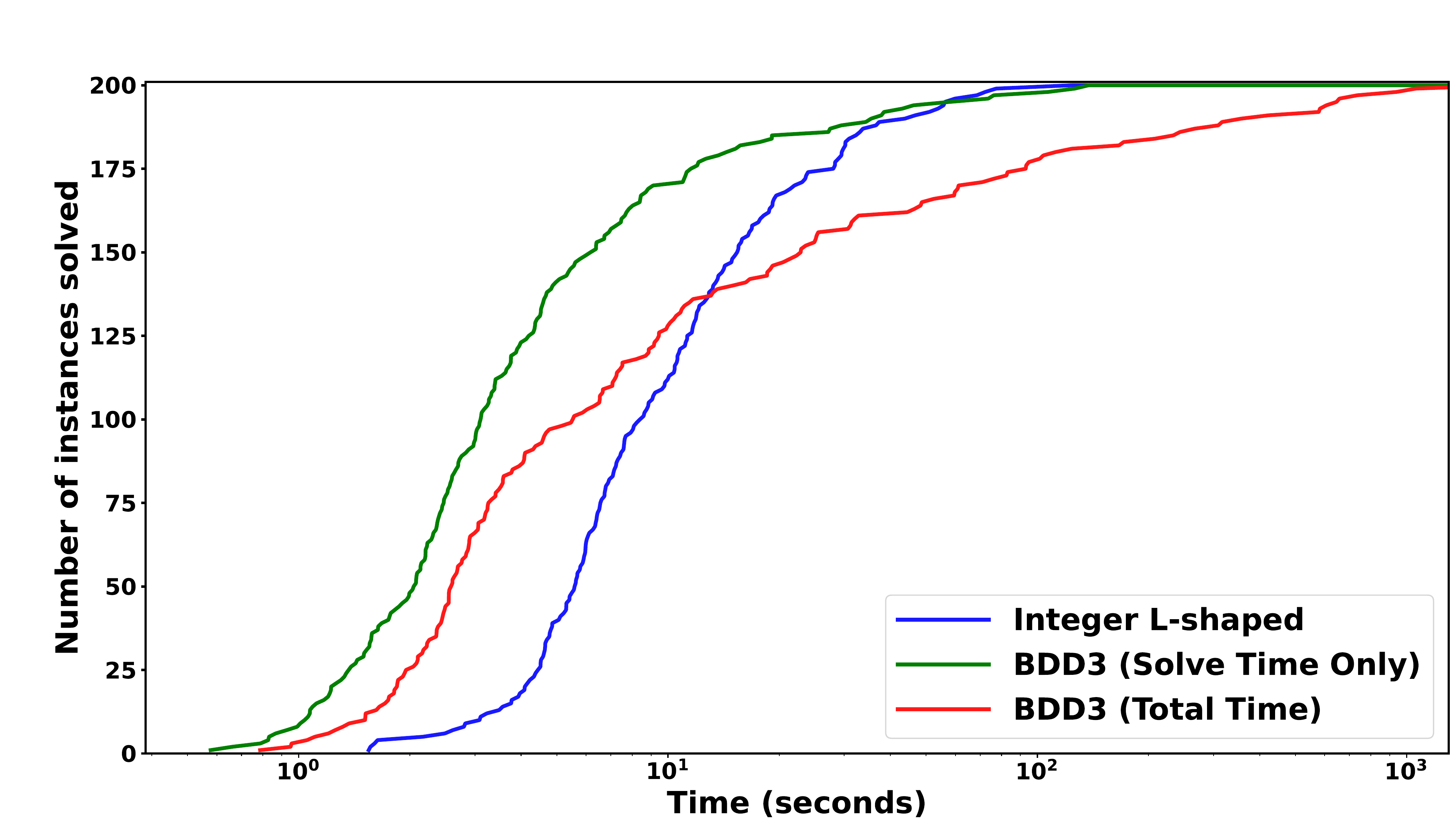}
		\caption{Methods with PB cuts.}
		\label{fig:smwds_pb}
	\end{subfigure}
	\caption{Comparison of \bdd{2}, \bdd{3}, and integer L-shaped methods.
	 }
	\label{fig:smwds_results}
\end{figure}%

We first compare both BDD-based decomposition methods against the integer L-shaped method without PB cuts. As seen in Figure \ref{fig:smwds_nopb}, which shows the total time (taking into account the building of the BDDs) the \bdd{3} method performs best and is able to solve all instances within the time limit, while the integer L-shaped algorithm is only able to solve $103$ of the $200$ instances within the time limit. The \bdd{2} method is able to solve $135$ instances, but it hits the memory limit on $28$ instances while building the BDDs, for this reason we exclude it from the remainder of our study. 
{\cred For this set of test instances, as seen in Figure \ref{fig:smwds_bddsize},  the mean size of the BDDs  of type $\bdd{3}$ are always smaller than those of $\bdd{2}$, {\credrev further details of this comparison can be found in \ref{app:bddsize}}.}
Results with PB cuts are given in Figure \ref{fig:smwds_pb}.
 The PB cuts for this problem are very strong and allow the integer L-shaped algorithm to perform best on the instances with low density. The  \bdd{3} method outperforms integer L-shaped on $135$ instances, and we see that the main obstacle for the remaining $35$ instances is the time it takes to generate the BDDs for the scenario subproblems. {\credrev A brief sensitivity analysis of the methods can be found in \ref{app:sens}.}

\begin{figure}[t]\centering
	\includegraphics[width=0.6\textwidth]{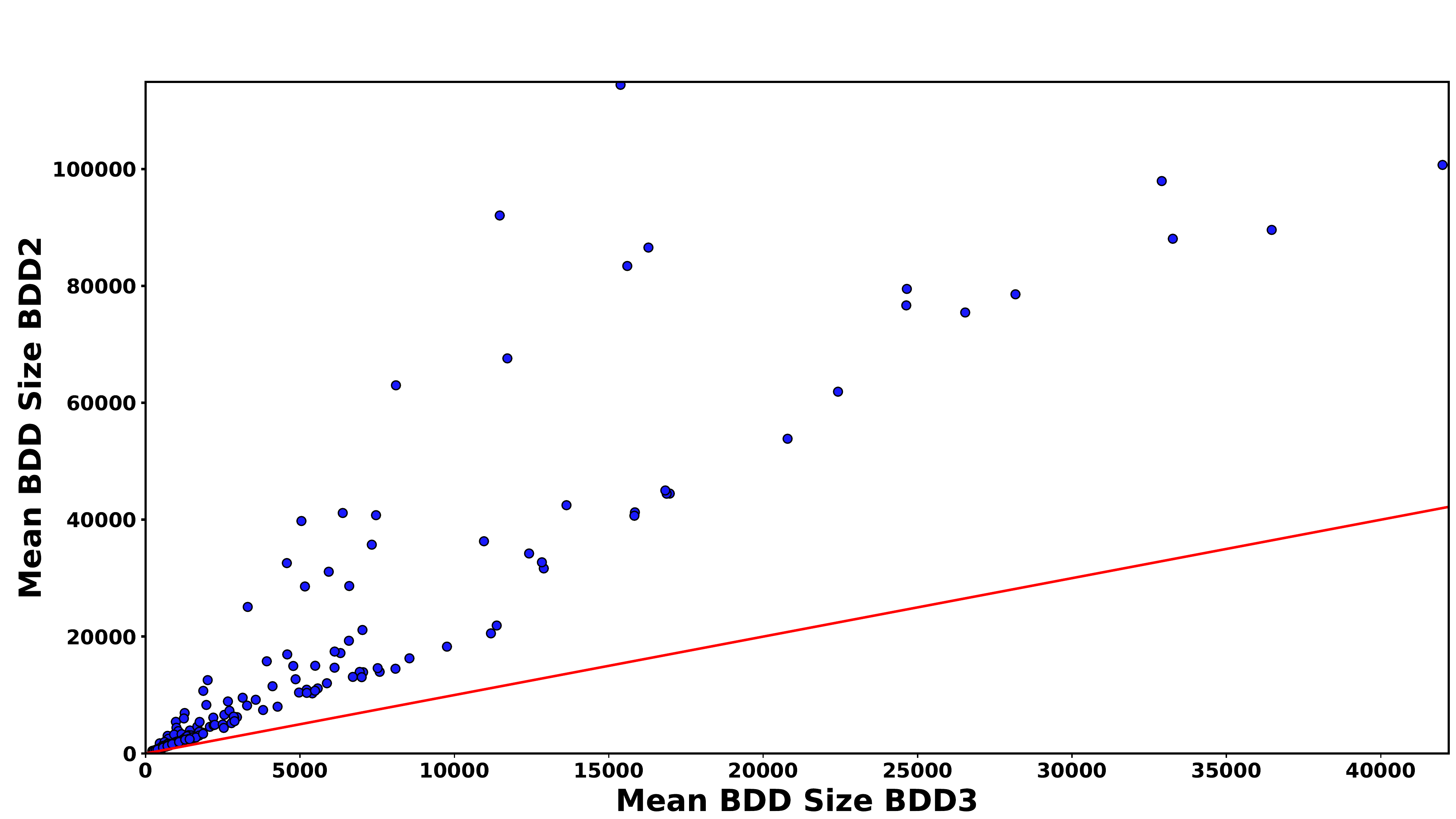}
	\caption{Comparison of mean BDD size for each instance, the red line shows where the mean BDD size is equal. Instances where \bdd{2} hit the memory limit have been omitted.}
	\label{fig:smwds_bddsize}
\end{figure}%

Next, we compare the algorithms in the CVaR setting with $\lambda =0.1$ and $\alpha =0.9$, and separating the CVaR cuts only after there are no remaining recourse optimality cuts, and using PB cuts. These results can be seen in Figure \ref{fig:smwds_cvar}, where the \bdd{3} method performs best on all but $5$ instances.

\begin{figure}[h]\centering
	\includegraphics[width=0.6\textwidth]{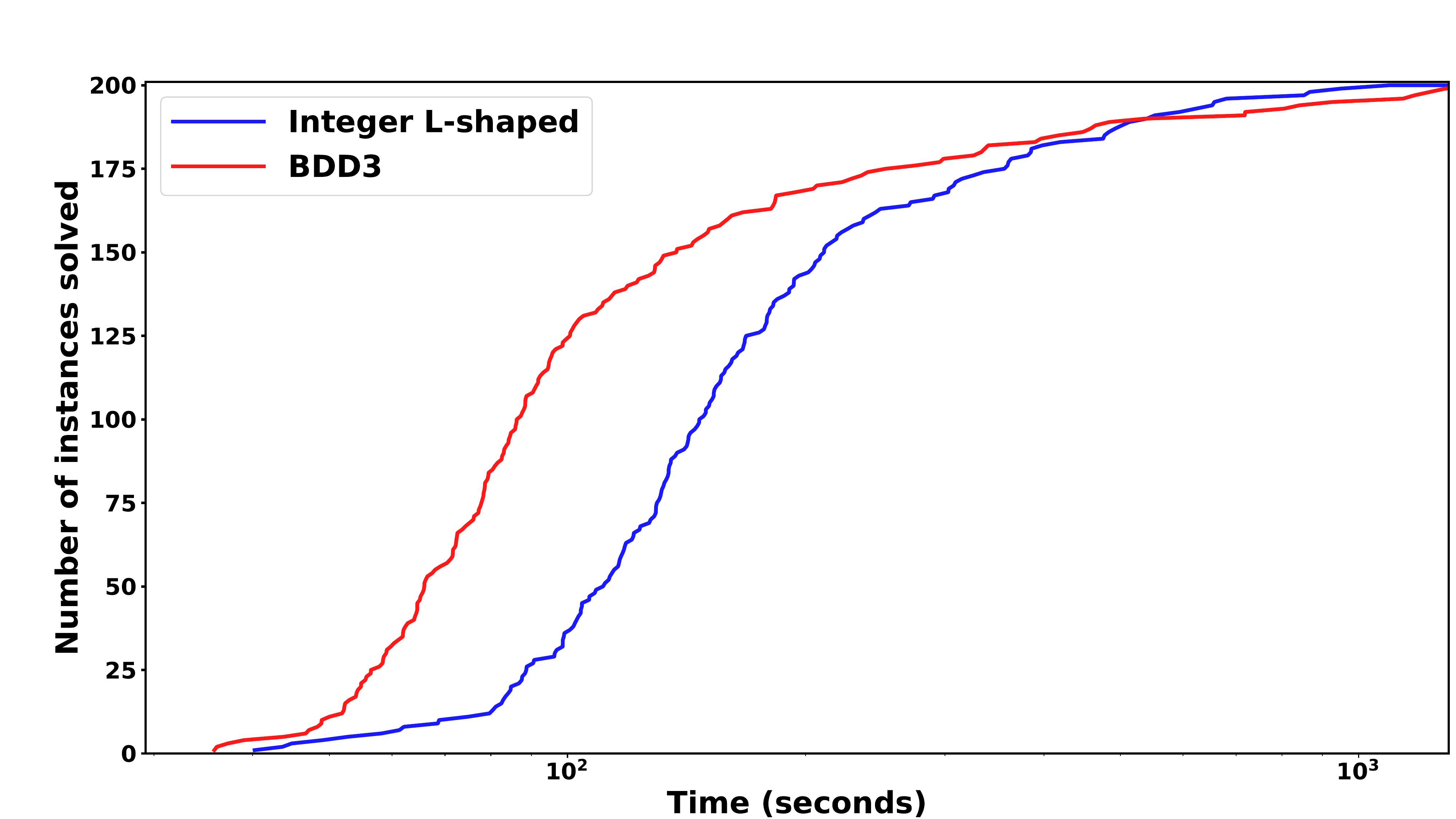}
	\caption{Comparison of methods using PB cuts on CVaR with $\lambda =0.1$ and $\alpha =0.9$.}
	\label{fig:smwds_cvar}
\end{figure}%

Overall, we conclude that the \bdd{3} method performs best on this application problem, as it requires less memory to build the BDDs than the \bdd{2} method and is better than or comparable to the integer L-shaped method with PB cuts on the standard and risk-averse formulations. 

\section{Conclusion} \label{sec:concl}

We considered a class of two-stage stochastic programming problems with binary second-stage variables and characterized by implication constraints, which cover a wide array of applications.
 We reformulated the second-stage problems using BDDs in order to apply Benders decomposition. We proposed a novel reformulation that parameterized the BDD arc costs using first-stage solutions and generalized an existing setting for the BDD reformulation  from \citep{lozano2018bdd}. We analyzed the strength of these BDD-based Benders cuts and found they were incomparable but that both dominated the more traditional integer L-shaped cuts. We then extended the results to a risk-averse setting, presenting the first decomposition method for 2SP with binary recourse and CVaR. Finally, we proposed the SMWDS problem as an application for these methods and showed via computational experiments that our novel BDD-based algorithm performs best in practice.

One of the main limitations of this BDD-based approach is the time spent to build the BDDs, thus it may be better suited to stochastic problems where a smaller number of scenarios are required. 
As a future research direction, we propose applying this methodology in a non-linear setting, such as submodular optimization. There is also promise in applying these methods to problems where PB cuts are weak, and to applications in the risk-averse and robust optimization settings.
{\cred The BDD representations presented are exact, so another promising future direction is to explore ways to circumvent computational cost using relaxations and restrictions of the BDD structure. These could then be easily incorporated into the branch-and-cut algorithm to provide bounds and could be used to give approximate solutions.}


{\small
\bibliographystyle{apalike}
\bibliography{MWDSrefs} 
}

\newpage

\appendix

\section{Stochastic Knapsack Problem} \label{app:stochknap}
Recall the stochastic knapsack problem with uncertain profit. We must first commit to procuring, i.e., producing ourselves or outsourcing, a subset of a set of items, $I$, with stochastic costs. These items have expected profit $p^0_i$ and are subject to a first-stage knapsack constraint with capacity $C^0$.
We then observe a profit degradation  $d_i$, and in the second stage we have three recourse options for each item.
Firstly, we can choose to produce the item at the reduced profit $p_i = p^0_i - d_i$ and consume second-stage knapsack capacity $c_i$. 
Secondly, we can repair the item consuming some extra second-stage knapsack capacity, $c_i + t_i$, but recovering the original expected profit of $p_i^0$, or finally, we can outsource the item so that the actual profit of the outsourced item is $p^0_i - f_i$. The items produced or repaired in the second stage are subject to knapsack constraint with capacity $C$.

Let $x_i = 1$ if we select item $i$ to procure in the first stage, and in the second stage let $y_i = 1$ if we select item $i$ to produce ourselves and let $r_i = 1$ if we choose to repair item $i$. The stochastic knapsack problem can be formulated as follows
\bsubeq	
\begin{alignat}{2}
	\max \ & \sum_{i \in I} (p^0_i - f_i) x_i + \expect[	\mathcal{Q}(x, \omega)]\\
	\text{s.t.} \ & \sum_{i \in I}  c_i^0 x_i \leq C^0\\ 
	&  {x} \in  \{0,1\}^{|I|}
\end{alignat}%
\esubeq
where
\bsubeq	
\begin{alignat}{2}
	\mathcal{Q}(x, \omega) = 	\max \ &  \sum_{i \in I} (f_i - d_i)y_i + d_ir_i  \\
	\text{s.t.} \ &  \sum_{i \in I} c_iy_i + t_ir_i \leq C\\
	& y_i \leq {x}_i && { \forall i \in I}\\
	& r_i \leq y_i &&   { \forall i \in I}\\
	&   y \in  \{0,1\}^{|I|} , \  r \in  \{0,1\}^{|I|}. \qquad
\end{alignat}%
\esubeq

\section{Benders Algorithm Overview} \label{app:bendersoverview}

We remark that the algorithm outlined here is a pure cutting plane version of the Benders algorithm, however it is generally implemented as a branch-and-cut algorithm as seen in Section 7.

We denote the lower bound on  $\eta_\omega$ as $\LB^\omega$ and the lower and upper bounds on the subproblem objective at iteration $k$ as $\LB_k$ and $\UB_k$, respectively. Let the set of cuts in the master problem be $\cutset$ and the cuts found at a given iteration $k$ be $\cutset_k$ so that at iteration $k$, $\cutset = \cup_{i=0}^{k-1} \cutset_i$. We denote the BDD subproblems as $\SP(\mpsol, \omega)$.
\BE
\setlength\itemsep{0.1em}
\I[0.] Initialize: Set iteration counter $k=0$ and set the upper bound $\UB_0 = \infty$. Generate a BDD  $\BDD^\omega$ for each scenario $\omega \in \Omega$. Let $\cutset = \cutset_0 = \{\eta_\omega \geq \LB^\omega \ \forall \omega \in \Omega\}$. 
\I Set $k=k+1$ and solve $\MP(\cutset)$. If this problem is infeasible, whole problem is infeasible so terminate. Otherwise we have solution $(\hat{\xvar}, \mpsol, \hat{\eta})$, set $\LB_k$ to the objective value at this solution.
\I Solve each BDD subproblem $\SP(\mpsol, \omega)$ using the shortest path algorithm for the dual solution and yielding objective value $\hat{\spobj}_\omega$. Calculate $\UB_{est} = {\xweights^\top \hat{\xvar}} + \sum_{\omega \in \Omega} p_\omega \hat{\spobj}_\omega$. If $\UB_{est} < \UB_{k-1}$, set $\UB_k = \UB_{est}$ and update the incumbent solution, otherwise $\UB_k = \UB_{k-1}$.
\I If $\LB_k = \UB_k$ accept the incumbent as an optimal solution and terminate. Otherwise, use the dual solutions from Step 2 to get Benders cuts $\cutset_k$. Update $\cutset = \cutset \cup \cutset_k $. Go to Step 1. 
\EE

\section{BDD Recourse Problem Formulations} \label{app:bddforms}

\subsection{\bdd{2} Recourse Problems}
	\bsubeq\label{form:BDD1_primal_a}
	\begin{alignat}{2}
		\Primal{2}: \min \ & \sum_{a  \in \onearcs} f_a\secondstagecost{v(a)} \\
		\text{s.t.} \ &  \sum_{a = (n,k) \in \arcs^\omega}f_a - \sum_{a = (k,n) \in \arcs^\omega} f_a = 0 \quad &&  \forall n \in \nodes^\omega \setminus \{\rootnode, \term\} \label{eq:flowbal1_a} \\ 
		&  \sum_{a = (\rootnode,j)\in \arcs^\omega}f_a = 1  \\ 
		& \sum_{a = (j,\term) \in \arcs^\omega} f_a =  1  \label{eq:flowbal3_a}\\ 
		& f_a \leq 1 - \mpsol^\omega_i && \forall \ a \in \arcs^\omega , i \in \viol_a \label{eq:capconstr_}\\ 
		&  f_a \geq 0 && \forall \ a \in \arcs^\omega 
	\end{alignat}
	\esubeq

	\bsubeq\label{form:BDD1_dual_a}
	\begin{alignat}{2}
		\Dual{2}: \max \ & \pi_\rootnode - \pi_\term  - \sum_{a  \in \arcs^\omega} \sum_{i \in \viol_a} \beta_{ai}(1 - \mpsol^\omega_i) \\
		\text{s.t.} \ & \pi_u - \pi_v -   \sum_{i \in \viol_a} \beta_{ai} \leq 0 \qquad &&  \forall a=(u,v) \in \zeroarcs \\ 
		&\pi_u - \pi_v -   \sum_{i \in \viol_a} \beta_{ai} \leq\secondstagecost{\map(a)} \quad  &&  \forall a=(u,v) \in\onearcs\\ 
		&  \beta_{ai} \geq 0  && \forall \ a \in \arcs^\omega, \forall i \in \viol_a.
	\end{alignat}
	\esubeq

\subsection{\bdd{3} Recourse Problems}
\bsubeq\label{form:BDD2_primal1_a}
\begin{alignat}{2}
	\Primal{3}: \min \ & \sum_{a \in \onearcs} f_a\left(\costone{\map(a)} (1-\mpsol^\omega_{\map(a)})+ \costtwo{\map(a)}\right) \\
	\text{s.t.} \ &  \sum_{a = (n,k) \in \arcs^\omega}f_a - \sum_{a = (k,n) \in \arcs^\omega} f_a = 0 \quad &&  \forall n \in \nodes^\omega \setminus \{\rootnode, \term\} \label{eq:P1_flowbal1_a} \\ 
	&  \sum_{a = (\rootnode,j)\in \arcs^\omega}f_a = 1  \\ 
	& \sum_{a = (j,\term) \in \arcs^\omega} f_a =  1  \label{eq:P1_flowbal3_a}\\ 
	&  f_a \geq 0 && \forall \ a \in \arcs^\omega \label{eq:P1_bds_a}
\end{alignat}%
\esubeq%

\bsubeq\label{form:BDD2_dual1_a}
\begin{alignat}{2} 
	\Dual{3}: 	\max \ & \pi_\rootnode -\pi_\term  \\
	\text{s.t.} \ & \pi_i - \pi_j \leq 0 && \forall (i,j) \in \zeroarcs \\ 
	&  \pi_i - \pi_j \leq\left(\costone{\map(a)} (1-\mpsol^\omega_{\map(a)})+ \costtwo{\map(a)}\right)  \quad && \forall a =(i,j) \in \onearcs
\end{alignat}
\esubeq

	\bsubeq\label{form:BDD2_primal2_a}
	\begin{alignat}{2}
		\AltPrimal{3}: 	\min \ & \sum_{a \in \onearcs}  \left(\costone{\map(a)} + \costtwo{\map(a)}\right)f_a + \sum_{a \in \onearcs} \costtwo{\map(a)}\gamma_a\\
		\text{s.t.} \ & \sum_{a=(\rootnode,j) \in \arcs^\omega} f_a + \sum_{a =(\rootnode,j) \in \onearcs}  \gamma_a  = 1 \label{eq:P2_flow1_a} \\ 
		&\sum_{a =(i,\term) \in \arcs^\omega} f_a + \sum_{a =(i,\term) \in \onearcs} \gamma_a  = 1 \\ 
		& \sum_{\substack{a =(i,j)\\ \in \arcs^\omega}} f_a + \sum_{\substack{a =(i,j)\\ \in \onearcs}}\gamma_a =   \sum_{\substack{a =(j,i)\\ \in \arcs^\omega}} f_a + \sum_{\substack{a =(j,i)\\ \in \onearcs}} \gamma_a && \forall i \in \nodes \label{eq:P2_flow2_a}\\
		& \gamma_a  \leq \mpsol^\omega_{\map(a)} && \hspace*{-0.04cm} \forall a \in \onearcs \label{eq:P2_gammaUB_a} \\
		& \gamma_a  \geq 0 \quad && \hspace*{-0.04cm}   \forall a \in \onearcs \label{eq:P2_gammaLB_a}  \\
		& f_a \geq 0 \quad && \forall a\in \arcs^\omega \label{eq:P2_fLB_a} 
	\end{alignat}
	\esubeq

	\bsubeq\label{form:BDD2_dual2_a}
\begin{alignat}{2}
	\AltDual{3}: 	\max \ & \pi_\rootnode  - \pi_\term  - \sum_{a  \in \onearcs} z_a\mpsol^\omega_{\map(a)}  \\
	\text{s.t.} \ & \pi_i - \pi_j \leq 0 && \forall (i,j) \in \zeroarcs \\ 
	&  \pi_i - \pi_j \leq \costone{\map(a)} + \costtwo{\map(a)} \quad && \forall a =(i,j) \in \onearcs \label{eq:D2_piUB_a} \\
	&  \pi_i - \pi_j \leq z_a + \costtwo{\map(a)}\quad && \forall a =(i,j) \in \onearcs \label{eq:D2_pizUB_a} \\
	& z_a \geq 0 \quad && \forall a =(i,j) \in \onearcs \label{eq:D2_zbd_a} 
\end{alignat}%
\esubeq

\section{Proofs of Propositions} \label{app:proofs}

\begin{proposition} \label{lemma:form_equiv_a}
	Given  a scenario $\omega$ and binary first-stage decisions $\mpsol^\omega$,  formulations $\Primal{3}$ and $\AltPrimal{3}$ have the same optimal objective value. 
\end{proposition}
\begin{proof}{Proof.} 
	Denote an optimal solution to $\Primal{3}$ as ${\hat{f}}^1$  and let $\spobj_1^*$ denote the optimal objective value. Similarly, denote an optimal solution to $\AltPrimal{3}$ as $({\hat{f}}^2, {\hat{\gamma}}^2)$  and let $\spobj_2^*$ denote the optimal objective value.
	\BE[leftmargin=\parindent]
	\I 	Given ${\hat{f}}^1$, let $\gamma_a^2 = \hat{f}_a^1\mpsol^\omega_{\map(a)} $ and $f_a^2 = \hat{f}_a^1 - \gamma_a^2 $ for all $a \in \onearcs$ and let $f_a^2 = \hat{f}_a^1 $ for all $ a \in \zeroarcs$. 
	We will first show that the constructed solution $({{f}}^2, {{\gamma}}^2)$  is feasible to \eqref{form:BDD2_primal2_a}. 
	For each arc $a \in \zeroarcs$, we have flow $f_a^2 = \hat{f}_a^1$.
	For each arc $a \in \onearcs$, we have if $\mpsol^\omega_{\map(a)} =0$ then $\gamma_a^2 = 0$ and $f_a^2 = \hat{f}_a^1 $, otherwise, if $\mpsol^\omega_{\map(a)} =1$ then $\gamma_a^2 = \hat{f}_a^1 $ and $f_a^2 = 0$. In either case the flow on each $a \in \onearcs$ is $f_a^2 + \gamma_a^2 = \hat{f}_a^1$. Therefore any solution feasible to \eqref{form:BDD2_primal1_a}  is also feasible to \eqref{eq:P2_flow1_a}-\eqref{eq:P2_flow2_a}.
	Constraints \eqref{eq:P2_gammaUB_a} hold since by construction if $\mpsol^\omega_{\map(a)} = 0$ then $\gamma_a^2 = 0$, and there is an implicit upper bound of $1$ on ${f}_a^1$ at any optimal solution, therefore if $\mpsol^\omega_{\map(a)} = 1$,  $\gamma_a^2 \leq 1$ by construction. 
	Similarly, by constraint \eqref{eq:P1_bds_a}, $\hat{f}^1 \geq 0$ and $\mpsol^\omega_{\map(a)} \geq 0$, so  $\gamma_a^2  \geq 0$ and $f_a^2 \geq 0$ by construction, since $\mpsol^\omega_{\map(a)} \leq 1$,  and constraints \eqref{eq:P2_gammaLB_a} and \eqref{eq:P2_fLB_a} hold.
	Thus given the solution ${\hat{f}}^1$ we can construct a feasible solution to $\AltPrimal{3}$.

	We will now show that at the constructed solution $({{f}}^2, {{\gamma}}^2)$, the objective values of \eqref{form:BDD2_primal1_a} and \eqref{form:BDD2_primal2_a} are equal.
	The objectives of both \eqref{form:BDD2_primal1_a} and \eqref{form:BDD2_primal2_a} consider only arcs in $\onearcs$, so we compare terms for a fixed arc $a \in \onearcs$ at binary solutions of $\mpsol^\omega_{\map(a)}$.
	As before, if $\mpsol^\omega_{\map(a)} =0$, $\gamma_a^2 = 0$ by construction and $f_a^2 = \hat{f}_a^1$. 
	Thus the objective term in $\AltPrimal{3}$ at arc $a$ is $ (\costone{\map(a)}+ \costtwo{\map(a)} \hat{f}_a^1$, the same as in $\Primal{3}$.
	Again, by construction, if $\mpsol^\omega_{\map(a)} =1$ then $\gamma_a^2 = \hat{f}_a^1$  and  $f_a^2 = 0$.
	Thus, the objective term for arc $a$ in $\AltPrimal{3}$ is $ \costtwo{\map(a)}\hat{f}_a^1$, which is equal to the objective term of $\Primal{3}$ since we have $ f^1_a(1-\mpsol^\omega_{\map(a)} )\costone{\map(a)} =0$. The constructed solution to $\AltPrimal{3}$ yields the objective value of $\spobj_1^*$ and we conclude $\spobj_1^* \geq \spobj_2^*$.

	\I Given $({\hat{f}}^2, {\hat{\gamma}}^2)$, let $f_a^1 = \hat{f}_a^2 + \hat{\gamma}^2_a$ for all $a \in \onearcs$ and let $f_a^1 = \hat{f}_a^2$ for all $ a \in \zeroarcs$. 
	
	We will show the solution, $f^1$, is feasible to \eqref{form:BDD2_primal1_a}. For each arc $a \in \onearcs$, if $\mpsol^\omega_{\map(a)} =0$, by constraint \eqref{eq:P2_gammaUB_a} $\hat{\gamma}_a^2 = 0$, thus $f_a^1= \hat{f}_a^2$.
	If $\mpsol^\omega_{\map(a)} =1$ then $f_a^1 = \hat{f}_a^2 + \hat{\gamma}^2_a$ but at an optimal solution of $\AltPrimal{3}$ since $\costone{\map(a)} > 0$ by definition and we are minimizing we will have $\hat{\gamma}^2_a = 1$ and $\hat{f}_a^2 = 0$, thus $f_a^1 =  \hat{\gamma}^2_a$.
	For each arc $a \in \zeroarcs$, we have $f_a^1 = \hat{f}_a^2$. 
	This solution is feasible to constraints \eqref{eq:P1_flowbal1_a}-\eqref{eq:P1_flowbal3_a} since the flow on an arc is either given by $\hat{\gamma}^2_a$ or $\hat{f}^2_a$ and thus satisfy flow balance.
	Constraints \eqref{eq:P1_bds_a} hold since by \eqref{eq:P2_gammaLB_a} and \eqref{eq:P2_fLB_a}
	$\hat{f}_a^2 \geq 0$ and $\hat{\gamma}^2_a \geq 0$ so we have $f_a^1 \geq 0$ by construction.
	
	We will again compare objective terms for a fixed arc $a \in \onearcs$ at binary solutions of $\mpsol^\omega_{\map(a)}$. 
	As before, if $\mpsol^\omega_{\map(a)} =0$, $f_a^1= \hat{f}_a^2$ and the objective term for $\Primal{3}$ at $a$ is $  (\costone{\map(a)} + \costtwo{\map(a)})\hat{f}_a^2$, the same as in $\AltPrimal{3}$.
	If $\mpsol^\omega_{\map(a)} =1$, again, we have $f_a^1 =  \hat{\gamma}^2_a$.
	The objective term of $\AltPrimal{3}$ is $\costtwo{\map(a)}\hat{\gamma}^2_a$, which is equal to the objective term of $\Primal{3}$ since we have $ f^1_a(\costone{\map(a)} (1-\mpsol^\omega_{\map(a)})+ \costtwo{\map(a)}) = \hat{\gamma}^2_a\costtwo{\map(a)}$.
	
	The constructed solution to $\Primal{3}$ yields the objective value of $\spobj_2^*$ and we conclude $\spobj_1^* \leq \spobj_2^*$.
	
	\EE
	
	Therefore we have shown at a binary first-stage decision $\mpsol^\omega$, the formulations $\Primal{3}$ and $\AltPrimal{3}$ are have the same optimal objective value, as required.
	\end{proof}

\begin{proposition}
	Given a scenario $\omega$ and a binary first-stage solution $\mpsol^\omega$, $\AltPrimal{3}$ and the recourse problem of Problem 3 have the same optimal objective value.
\end{proposition}
\begin{proof}{Proof.}
	First observe that every $\rootnode$-$\term$ path in the BDD resulting from the \bdd{3} procedure corresponds to exactly one feasible solution with the same objective value as in second-stage of Problem 3. Every feasible solution in Problem 3's second-stage corresponds to exactly one $\rootnode$-$\term$ path in the BDD with the same length as the objective value. That is, $\Primal{3}$ is equivalent to the recourse problem of Problem 3.
	
	We also have by the proof of Proposition \ref{lemma:form_equiv}, given an optimal solution to 
	$\AltPrimal{3}$ we can construct an optimal solution to $\Primal{3}$ and vice versa. Therefore, by transitivity, we have $\AltPrimal{3}$ and the recourse problem of Problem 3  have the same optimal objective value as required. 
\end{proof}

\begin{proposition}  \label{prop:bdd1_lshape_a}
	Given a scenario $\omega$ and a master problem solution $\mpsol^\omega$, if the subproblem objective coefficients are non-negative, i.e.,  $\secondstagecost{i} \geq 0$ for all $i \in \{1,\hdots, \yvarindex\}$, then 	
	the strengthened \bdd{2} cuts 
	\begin{equation}  \label{eq:bdd1_cut_stren_a}
		\eta_\omega \geq \hat{\pi}_\rootnode - \sum_{j \in \{1, \hdots, m_1\}} \sum_{i \in \{1, \hdots, m_1\}} \betamax_{ji} (1 - \indvar_i)
	\end{equation} 
are at least as strong as the integer L-shaped cuts 
\begin{equation}\label{eq:lshape_cuts_a}
	\eta_\omega \geq  \hat{\spobj}_\omega -  \hat{\spobj}_\omega\left(\sum_{\substack{	i \in \{1,\hdots, m_1\}\\ \mpsol^\omega_i = 1}} (1-\indvar_i) + \sum_{\substack{	i \in \{1,\hdots, m_1\}\\ \mpsol^\omega_i = 0}}\indvar_i \right).
\end{equation}
\end{proposition}
\begin{proof}{Proof.} 
	Assume $\secondstagecost{i} \geq 0$ for all $i \in \{1,\hdots, \yvarindex\}$. We want to show the strengthened \bdd{2} cuts \eqref{eq:bdd1_cut_stren_a} are as strong as 
	integer L-shaped cuts \eqref{eq:lshape_cuts_a} for all $i \in \{1,\hdots, \yvarindex\}$, that is,
	\[ \hat{\pi}_\rootnode - \sum_{j \in \{1, \hdots, m_1\}} \sum_{i \in \{1, \hdots, m_1\}} \betamax_{ji} (1 - \indvar_i)
	\geq 
	 \hat{\spobj}_\omega  - 
	 \hat{\spobj}_\omega\left(\sum_{\substack{	i \in \{1,\hdots, m_1\}\\ \mpsol^\omega_i  = 1}} (1-\indvar_i) + \sum_{\substack{	i \in \{1,\hdots, m_1\}\\ \mpsol^\omega_i  = 0}}\indvar_i \right).
	\]
	As described in Section 5.1,  for each capacitated arc $a$, we have $\hat{\beta}_{ai} > 0  $ for exactly one $i \in \viol_a$ where $\mpsol^\omega_i=1$ and otherwise  $\hat{\beta}_{ai} = 0$. Therefore if $\betamax_{ji} > 0$, we must have $(1 - \mpsol^\omega_i)= 0$, and otherwise  $\betamax_{ji} = 0$, since $\betamax_{ji}  =  \max_{a \in \arclayer{j}}\{\hat{\beta}_{ai} : i \in \viol_a\}$. 
	As a result for all $j \in \{1, \hdots, m_1\}$ and $i \in \{1, \hdots, m_1\}$, we have 
	$ \betamax_{ji} (1 - \mpsol^\omega_i)= 0$ at an optimal solution. 
$ \hat{\spobj}_\omega$ is the optimal objective value to an equivalent subproblem to that of \bdd{2}, so we have $$ \hat{\pi}_\rootnode -\sum_{j \in \{1, \hdots, m_1\}} \sum_{i \in \{1, \hdots, m_1\}} \betamax_{ji} (1 - \mpsol^\omega_i)	= \hat{\pi}_\rootnode  =  \hat{\spobj}_\omega.$$ 
	Substituting $\hat{\pi}_\rootnode$ for $  \hat{\spobj}_\omega$, and swapping the summation indices, we would like to show
	\begin{equation}\label{eq:proof_bdd1_lshape_a}
		\sum_{i \in \{1, \hdots, m_1\}} \sum_{j \in \{1, \hdots, m_1\}} \betamax_{ji} (1 - \indvar_i)
		\leq 
		\hat{\pi}_\rootnode
		\left(\sum_{\substack{	i \in \{1,\hdots, m_1\}\\ \mpsol^\omega_i  = 1}} (1-\indvar_i) + \sum_{\substack{	i \in \{1,\hdots, m_1\}\\ \mpsol^\omega_i  = 0}}\indvar_i \right). 
	\end{equation}
	By definition of the strengthened cuts \eqref{eq:bdd1_cut_stren_a}, we select at most one arc $a = (u,v)$ per BDD layer $j$ and index $i$. 
	By construction of the solution, $\hat{\beta}_{ai} > 0$ for the index  $i \in \viol_a$ only when $\mpsol^\omega_i =1$. Therefore we also have $\betamax_{ji} > 0$ only when $\mpsol^\omega_i =1$.  Thus the outer left-hand term of \eqref{eq:proof_bdd1_lshape_a} and $\sum_{\substack{	i \in \{1,\hdots, m_1\}\\ \mpsol^\omega_i = 1}} (1-\indvar_i) $ sum over the same  $\indvar$ indices.

	Without loss of generality, fix $i = 1 $, for each BDD arc layer, $\arclayer{j}$, there exits $\bar{a} \in  \arclayer{j}$ such that $\hat{\beta}_{\bar{a}1} = \betamax_{j1}$. 
	The value of  $\hat{\beta}_{\bar{a}1} + \secondstagecost{{\map{(\bar{a})} } } = \pi_{\bar{u}} - \pi_{\bar{v} }$ is the path length increase over the arc $\bar{a}$
	since $ \secondstagecost{\map(a)} \geq 0$ for all $a \in \arcs$. 
	By construction of the cut, we will select at most one arc $\bar{a}$ per layer $ \arclayer{j}$, and  $\hat{\beta}_{\bar{a}1}  > 0$ only if $\bar{a}$ is on the shortest path and $\mpsol^\omega_1 = 1$.
	Thus, $\betamax_{j1}  + \secondstagecost{\bar{a}}$ is the path length increase over layer $j$.
	Since at most every violated index will exist on every layer, and $\hat{\pi}_{\rootnode}$ is the total path length increase over the entire BDD, we have  \[
	\sum_{j \in \{1, \hdots, m_1\}} \betamax_{j1} + \sum_{a \in \bar{A}} \secondstagecost{a} 
	\leq
	\hat{\pi}_\rootnode,\]
	where $\bar{A}$ is a set of one arc $a$ per layer $\arclayer{j}$ such $\hat{\beta}_{a1} = \betamax_{j1}$.
	We have $\secondstagecost{\map{(a)}} \geq 0$ by assumption, thus  
	$
	\sum_{j \in \{1, \hdots, m_1\}} \betamax_{ji}
	\leq
	\hat{\pi}_\rootnode$ for any $i$. Therefore we have 
	\[
	\sum_{i \in \{1, \hdots, m_1\}} \sum_{j \in \{1, \hdots, m_1\}} \betamax_{ji} (1 - \indvar_i)
	\leq
	\sum_{\substack{	i \in \{1,\hdots, m_1\}\\ \mpsol^\omega_i  = 1}}  \hat{\pi}_\rootnode(1-\indvar_i).
	\]
	By definition $1 \geq \indvar_i \geq 0$, and by construction of dual solutions and since $\secondstagecost{i} \geq 0$ by assumption,  $\hat{\pi}_{\rootnode} \geq0$ thus
	\[ 
	\sum_{\substack{	i \in \{1,\hdots, m_1\}\\ \mpsol^\omega_i  = 1}}  \hat{\pi}_\rootnode(1-\indvar_i) 
	\leq  
	\hat{\pi}_\rootnode
	\left(\sum_{\substack{	i \in \{1,\hdots, m_1\}\\ \mpsol^\omega_i  = 1}} (1-\indvar_i) + \sum_{\substack{	i \in \{1,\hdots, m_1\}\\ \mpsol^\omega_i  = 0}}\indvar_i \right). 
	\]
	
	Thus by transitivity the equation \eqref{eq:proof_bdd1_lshape_a} holds and we conclude  the cuts \eqref{eq:bdd1_cut_stren_a} are at least as strong as the cuts \eqref{eq:lshape_cuts_a} when $\secondstagecost{i} \geq 0$ for all $i \in \{1,\hdots, \yvarindex\}$.
\end{proof}


\begin{proposition} \label{prop:bdd2_lshape_a}
	Given a scenario $\omega$ and a master problem solution $\mpsol^\omega$, 
	if the subproblem objective coefficients are non-negative, i.e.,  $\costtwo{i}\geq 0$ for all $i \in \{1,\hdots, \yvarindex\}$, then 	
	the strengthened \bdd{3} cuts \begin{equation} \label{eq:bdd2_cut_stren_a}
		\eta_\omega \geq \hat{\pi}_\rootnode   - \sum_{j \in \{1, \hdots, m_1\}} \zmax_j \indvar_{j} .
	\end{equation}  are at least as strong as the integer L-shaped cuts \eqref{eq:lshape_cuts_a}.
\end{proposition}
\begin{proof}{Proof.}
	Assume $\costtwo{i}\geq 0$ for all $i \in \{1,\hdots, \yvarindex\}$ and
	observe that by Proposition \ref{lemma:form_equiv}, we have $ \hat{\spobj}_\omega = \hat{\pi}_{\rootnode}$ at an optimal subproblem solution. Thus we would like to show:
	\begin{equation}\label{eq:bdd2_proof_a}
		\sum_{j \in \{1, \hdots, m_1\}} \zmax_j \indvar_{j} 
		\leq \hat{\pi}_{\rootnode}\left(\sum_{\substack{	i \in \{1,\hdots, m_1\}\\ \mpsol^\omega_i = 1}} (1-\indvar_i) + \sum_{\substack{i \in \{1,\hdots, m_1\}\\ \mpsol^\omega_i = 0}}\indvar_i \right). 
	\end{equation}  
	We remark that at an optimal subproblem solution, {\cred the change in path length over arc $a = (i,j)$ is $\hat{\pi}_i - \hat{\pi}_j  = \hat{z}_a + \costtwo{\map(a)}$}. 
	We have  $\costtwo{i} \geq 0$ for all $i \in \{1,\hdots, \yvarindex\}$ by assumption and $\hat{z}_a \geq 0$ by definition \eqref{eq:D2_zbd_a}, therefore $\hat{z}_a + \costtwo{\map(a)} \geq 0$ for all $a \in \arcs^\omega$ and $\hat{\pi}_i - \hat{\pi}_j \geq 0$ for all $(i,j) \in \arcs^\omega$.
	Thus we must have $ 0 \leq \hat{z}_a + \costtwo{\map(a)}   \leq \hat{\pi}_{\rootnode}$ since 
	{\cred by definition of the dual problem $\Dual{3}$,  we have $ \hat{\pi}_{\rootnode}  - \hat{\pi}_\term$ is the length of the shortest $\rootnode$-$\term$ path,  $\hat{\pi}_{\term} =0$ by assumption, and $\hat{\pi}_i - \hat{\pi}_j $ is the length of a single arc in that path.
	 }
	Therefore, since for all $j \in \{1,\hdots, m_1\}$ there exists an arc $a$ where $\zmax_j  = \hat{z}_a$, and $\costtwo{j}\geq 0$ by assumption,
	we must also have $\zmax_j   \leq \hat{\pi}_{\rootnode}$.

	By definition of the strengthened cuts \eqref{eq:bdd2_cut_stren_a} we know there is exactly one  $\hat{z}_a$ term per layer in the cut, i.e., $\zmax_j$. We also know that if $\mpsol^\omega_{\map(a)}  =1$ the cost of all arcs in a layer $j$, are $\costtwo{\map{(a)}}$, {\cred then $\hat{z}_a =0$ for all $a \in \arclayer{j}$ since by assumption all any arc with $\pi_i-\pi_j - \costtwo{\map{(a)}} \leq 0 $ sets ${z}_a =0$ and we have $\costtwo{\map{(a)}} \geq 0$ and $\pi_i \leq \pi_j  + \costtwo{\map{(a)}}$  by $\AltDual{3}$.} 
	
	Thus the only terms which appear in the \bdd{3} cut are those where $\mpsol^\omega_{\map(a)} = 0$, which allow $\hat{z}_a \geq 0$, and therefore $\zmax_j \geq 0$. Thus the summation on the left-hand side of \eqref{eq:bdd2_proof_a} is over the same terms as the summation on the right-hand side with  $\mpsol^\omega_{\map(a)} = 0$. Thus with $\zmax_j   \leq \hat{\pi}_{\rootnode}$  we have 
	\[ \sum_{j \in \{1, \hdots, m_1\}} \zmax_j \indvar_{j} 
	\leq \sum_{\substack{	i \in \{1,\hdots, m_1\}\\ \mpsol^\omega_i = 0}} \hat{\pi}_{\rootnode} \indvar_i.
	\]
	By definition $1 \geq \indvar_i \geq 0$, for all $i \in \{1,\hdots,m_1\}$ and we have shown $\hat{\pi}_{\rootnode} \geq0$ thus
	
	\[ \sum_{\substack{	i \in \{1,\hdots, m_1\}\\ \mpsol^\omega_i = 0}} \hat{\pi}_{\rootnode} \indvar_i \leq \sum_{\substack{	i \in \{1,\hdots, m_1\}\\ \mpsol^\omega_i = 1}} \hat{\pi}_{\rootnode}(1-\indvar_i) + \sum_{\substack{	i \in \{1,\hdots, m_1\}\\ \mpsol^\omega_i = 0}} \hat{\pi}_{\rootnode}\indvar_i 
	\]
	and by transitivity we have \eqref{eq:bdd2_proof_a}. Thus we conclude  if $\costtwo{i} \geq 0$ for all $i \in \{1,\hdots, \yvarindex\}$ then 	the cuts \eqref{eq:bdd2_cut_stren_a} are at least as strong as the cuts \eqref{eq:lshape_cuts_a}.
\end{proof}

\section{Cut Strength Comparison} \label{app:cuts}

Next we show the two kinds of BDD cuts and the PB cuts are incomparable. To do so we derive cuts at two master problem solutions using the SMWDS scenario in Figure 1.
First we fix $\MP$ solution $\hat{\xvar} = [0,0,0,0,0]$ and derive the following cuts:
\begin{equation} \label{eq:bdd1_ex1}
	\bdd{2}: \eta \geq 2 - (2\xvar_{0} + 3\xvar_{1} + 3\xvar_{2} + 3\xvar_{3} + 2\xvar_{4})
\end{equation}
\begin{equation} \label{eq:bdd2_ex1}
	\bdd{3}: \eta \geq 2 - (\xvar_{0} + \xvar_{1} + \xvar_{2} + \xvar_{3} + \xvar_{4})
\end{equation}
\begin{equation} \label{eq:pb_ex1}
	\text{PB}: \eta \geq 1.5 - (\xvar_{0} + \xvar_{1} + \xvar_{2} + \xvar_{3} + 0.5\xvar_{4}).
\end{equation}
Next we fix $\MP$ solution $\hat{\xvar} = [0,0,1,0,0]$ and derive the cuts from Example 5,
	\begin{equation} \label{eq:bdd1_cut_ex_str}
	\eta \geq 1 - (\xvar_{0} + \xvar_{1} + \xvar_{3} )
\end{equation}
\begin{equation} \label{eq:bdd2_cut_ex_str}
	\eta \geq 1 - (\xvar_{0} + \xvar_{1} + \xvar_{3} +  \xvar_{4})
\end{equation}

 and the following PB cut,
\begin{equation}  \label{eq:pb_ex2}
	\text{PB}: \eta \geq 1 - (\xvar_{0} + \xvar_{1} + \xvar_{3} ).
\end{equation}
In the following propositions we compare the strength of the derived cuts.

\begin{proposition} \label{lemma:bdd1_pb}
	The strengthened \bdd{2} cuts \eqref{eq:bdd1_cut_stren} are incomparable with the pure Benders cuts \begin{equation}\label{eq:pb_cuts_a}
		\eta_\omega \geq \mathcal{T}_{\mpsol^\omega}(\indvar).
	\end{equation}
\end{proposition}
\begin{proof}{Proof.}
	First we show a point where \eqref{eq:bdd1_cut_stren}  dominates  \eqref{eq:pb_cuts}. Let $\hat{\xvar} = [0.25,0,0,0,0]$ and let $\hat{\eta} = 1.3$. The inequality \eqref{eq:bdd1_ex1} becomes $ \eta \geq 1.5$ which will cut off the current solution of $\hat{\eta} = 1.3$, however the inequality \eqref{eq:pb_ex1} becomes $\eta \geq 1.25$ which does not cut off the current solution.
	Next we show a point where \eqref{eq:pb_cuts} dominates  \eqref{eq:bdd1_cut_stren},  let $\hat{\xvar} = [0,0,0,0,1]$ and let $\hat{\eta} = 0.5$. The inequality \eqref{eq:bdd1_ex1} becomes $ \eta \geq 0$ which does not cut  off the current solution. The inequality  \eqref{eq:pb_ex1} becomes $\eta \geq 1$ which cuts off the solution as $\hat{\eta} = 0.5$.
	We conclude these cuts are incomparable.
\end{proof}

\begin{proposition} \label{lemma:bdd2_pb}
	The strengthened \bdd{3} cuts  \eqref{eq:bdd2_cut_stren} are incomparable with the pure Benders cuts \eqref{eq:pb_cuts}.
\end{proposition}
\begin{proof}{Proof.}
	Let $\hat{\xvar} = [1,0,0,0,0]$ and let $\hat{\eta} = 0.75$, then the inequality \eqref{eq:bdd2_ex1} becomes $\eta \geq 1$ and \eqref{eq:pb_ex1} becomes $\eta \geq 0.5$ Thus \eqref{eq:bdd2_ex1} cuts off the current solution $\hat{\eta} =0.75$ and \eqref{eq:pb_ex1}  does not.
	Let $\hat{\xvar} = [0,0,0,0,1]$ and let $\hat{\eta} = 0.5$, the inequality \eqref{eq:bdd2_cut_ex_str} becomes $\eta \geq 0$ which does not cut off the current solution and \eqref{eq:pb_ex2} becomes $\eta \geq 1$ which does.
	We conclude these cuts are incomparable.
\end{proof}

\begin{proposition} \label{lemma:bdd1_bdd2_a}
	The strengthened \bdd{2} cuts  \eqref{eq:bdd1_cut_stren} and the strengthened \bdd{3} cuts   \eqref{eq:bdd2_cut_stren} are incomparable.
\end{proposition}
\begin{proof}{Proof.}
	Let $\hat{\xvar}= [0,0,0,0,1]$ and let $\hat{\eta} = 0.5$, the inequality \eqref{eq:bdd1_ex1} becomes $ \eta \geq 0$ which does not cut off the solution $\hat{\eta} = 0.5$. The inequality \eqref{eq:bdd2_ex1} becomes $\eta \geq 1$ which cuts off the current solution.
	Keeping the same $\hat{\xvar}$ and $\hat{\eta}$ solution, the inequality \eqref{eq:bdd1_cut_ex_str} becomes $\eta \geq 1$ which cuts off $\hat{\eta} = 0.5$. However, the inequality \eqref{eq:bdd2_cut_ex_str} becomes $\eta \geq 0$ which does not cut off the current solution.
	Thus we conclude these cuts are incomparable.
\end{proof}

Propositions \ref{lemma:bdd1_pb}, \ref{lemma:bdd2_pb}, and \ref{lemma:bdd1_bdd2_a} indicate that it may be beneficial to add cuts from multiple classes at each iteration. 

\section{CVaR Algorithm Overview} \label{app:cvar}

\BE
\I[0.] Initialize: Set iteration counter $k=0$ and set the upper bound $\UB_0 = \infty$. If using the BDD-based methods, generate a BDD  $\BDD^\omega$ for each scenario $\omega \in \Omega$ using either \bdd{2} or \bdd{3}. 

\I Solve the master problem and get optimal solution $(\hat{\xvar}, \mpsol, \hat{\eta}, \hat{\theta})$. Set $\LB_k$ to the objective value at this solution.

\I Solve each subproblem for the dual solution and yielding objective value ${\spobj}_\omega$. Set $k=k+1$.

\I \label{step:optcuts} For each $\omega \in \Omega$, if ${\spobj}_\omega > \hat{\eta}_\omega$ add the optimality cut 
\[	\eta_\omega \geq \hat{\pi}^k_\rootnode   - \sum_{j \in \{1, \hdots, m_1\}} \hat{z}^{\max k}_j \indvar_{j}  
\quad \forall \omega \in \Omega, k=1,\hdots,K \]
 using the subproblem solution.

\I Calculate the $\alpha$-quantile, ${\zeta}$, of the set of recourse problem realizations, and then calculate 
\[ \cvar({\spobj}_\omega, {\zeta})  =  {\zeta} + \frac{1}{1-\alpha}   [ \spobj_\omega-{\zeta}]_+. \]

\I \label{step:cvarcuts} For each $\omega \in \Omega$, if $ \cvar({\spobj}_\omega, {\zeta})   > \hat{\theta}_\omega$ add the CVaR optimality cuts
\bsubeq\label{form:cvar_mp_a}
\begin{alignat}{2}
	& \theta_\omega\geq \zeta^k + \frac{1}{1-\alpha}  \nu^{\omega k} \quad && \forall \omega \in \Omega, k=1,\hdots,K \label{eq:cvar_cut_a}\\
	& 		\nu^{\omega k} \geq \eta_\omega - \zeta^k \quad && \forall \omega \in \Omega, k=1,\hdots,K \label{eq:cvar_cut2_a}\\
	& 		\nu^{\omega k} \geq 0\quad && \forall \omega \in \Omega, k=1,\hdots,K \label{eq:cvar_nubd_a} \\
	& \zeta^k \in \R &&  k=1,\hdots,K \label{eq:cvar_zetabd_a}
\end{alignat}%
\esubeq
using the subproblem solution and introducing axillary variables.

\I If neither Step \ref{step:optcuts} nor Step \ref{step:cvarcuts} yield cuts, accept the incumbent as an optimal solution and terminate. Otherwise, go to Step 1. 
\EE

{\credrev \section{SMWDS Instance Generation} \label{sec:instancegen}
	
	We generate $200$ stochastic instances with $|\vertexset| \in \{30,35,40,45,50\}$ and edge density $\in \{0.2,0.4,0.6,0.8\}$. For each vertex set size and edge density pair we create $10$ instances with random seeds having the values from $0$ to $9$. Graphs are generated using the  \verb|gnm_random_graph| function of the Python package Networkx  \citep{SciPyProceedings_11}, and are then checked to ensure connectedness.  The first-stage vertex weights are randomly selected from the interval  $[20,70]$.
	We remark that this closely follows the structure of benchmarks from the deterministic MWDS literature \citep{jovanovic2010ant}. 
	For the second-stage graph and vertex weights we create a discrete distribution, similar to those seen in applications of graph problems uncertain edge weights \citep{Alexopoulos_statespace, Hutson_mst}.
	This distribution is composed of two weights randomly drawn from the interval $[20,70]$ and the weight $0$, each of which is randomly assigned a probability such that all probabilities sum to $1$. The inclusion of the weight $0$ is to model vertex failures in the second-stage problem, such that a vertex with weight $0$ is considered removed from the graph.

%

%
%

}
\section{Sample Average Approximation Analysis} \label{app:saa}

We conduct sample average approximation (SAA) analysis and conclude that to achieve a worst-case optimality gap below $1\%$, we must use at least $800$ scenarios and proceed to use $850$ scenarios in our experiments. Table \ref{table:saa} gives the $95\%$ confidence interval (CI) on the upper and lower bounds, respectively.

\begin{table}[H]\small \caption{\label{table:saa} SAA anaylsis for an instance with $30$ vertices and edge density $0.5$.  }
	\centering
	\begin{tabular}{l l l r}
		\toprule
		$|\Omega|$ & \multicolumn{1}{l}{95\% CI on} &  \multicolumn{1}{l}{95\% CI on}  & \multicolumn{1}{l}{Worst Case} \\
		&\multicolumn{1}{l}{Lower Bound}& \multicolumn{1}{l}{Upper Bound }& Optimality Gap (\%) \\
		\midrule
		10 & [34.147, 38.103] & [37.574, 37.634] & 10.21 \\
		25 & [35.509, 38.879] & [37.569, 37.629] & 5.97 \\
		50 & [35.761, 38.249] & [37.590, 37.650] & 5.28 \\
		100 & [36.666, 38.375] & [37.538, 37.591] & 2.52 \\
		250 & [36.789, 37.671] & [36.988, 37.040] & 0.68 \\
		500 & [36.731, 37.477] & [37.181, 37.236] & 1.37 \\
		600 & [36.664, 37.330] & [37.259, 37.301] & 1.74 \\
		750 & [36.810, 37.340] & [37.143, 37.211] & 1.09 \\
		800 & [36.846, 37.367] & [37.131, 37.205] & 0.98 \\
		850 & [36.865, 37.346] & [37.123, 37.183] & 0.86 \\
		\bottomrule
	\end{tabular}
\end{table}

{\credrev
\section{Sensitivity Analysis} \label{app:sens}
Here we present a brief sensitivity analysis of our algorithms. Using the procedure outlined in \ref{sec:instancegen}, we generated $80$ random graphs with $|\vertexset| \in \{30,35\}$,  edge density $\in \{0.2,0.4,0.6,0.8\}$ and random seeds having the values from $0$ to $9$. For each graph, we generated two weight distributions, one drawing weights from the interval $[0,10]$ and one from $[100,1000]$. We then generate $850$ scenarios per instance and compare the solution time of each instance on these two distributions. Figure \ref{fig:compare} shows the solution time comparison for each distribution on the \bdd{2} and \bdd{3} methods with and without PB cuts. The performance using the \bdd{3} methods is comparable for both distributions, with the $[0,10]$ weights being faster than the  $[100,1000]$ weights exactly half the time without PB cuts and in $39$ of the instances with PB cuts. When using the \bdd{2} methods  the results show a very slight preference for the distribution with larger vertex weights, when using PB cuts $38$ of the instances with $[0,10]$ weights are faster whereas without PB cuts only $34$ are faster.

\begin{figure}[h] 
	\centering

	\begin{subfigure}{0.5\textwidth}
	\includegraphics[width=\textwidth]{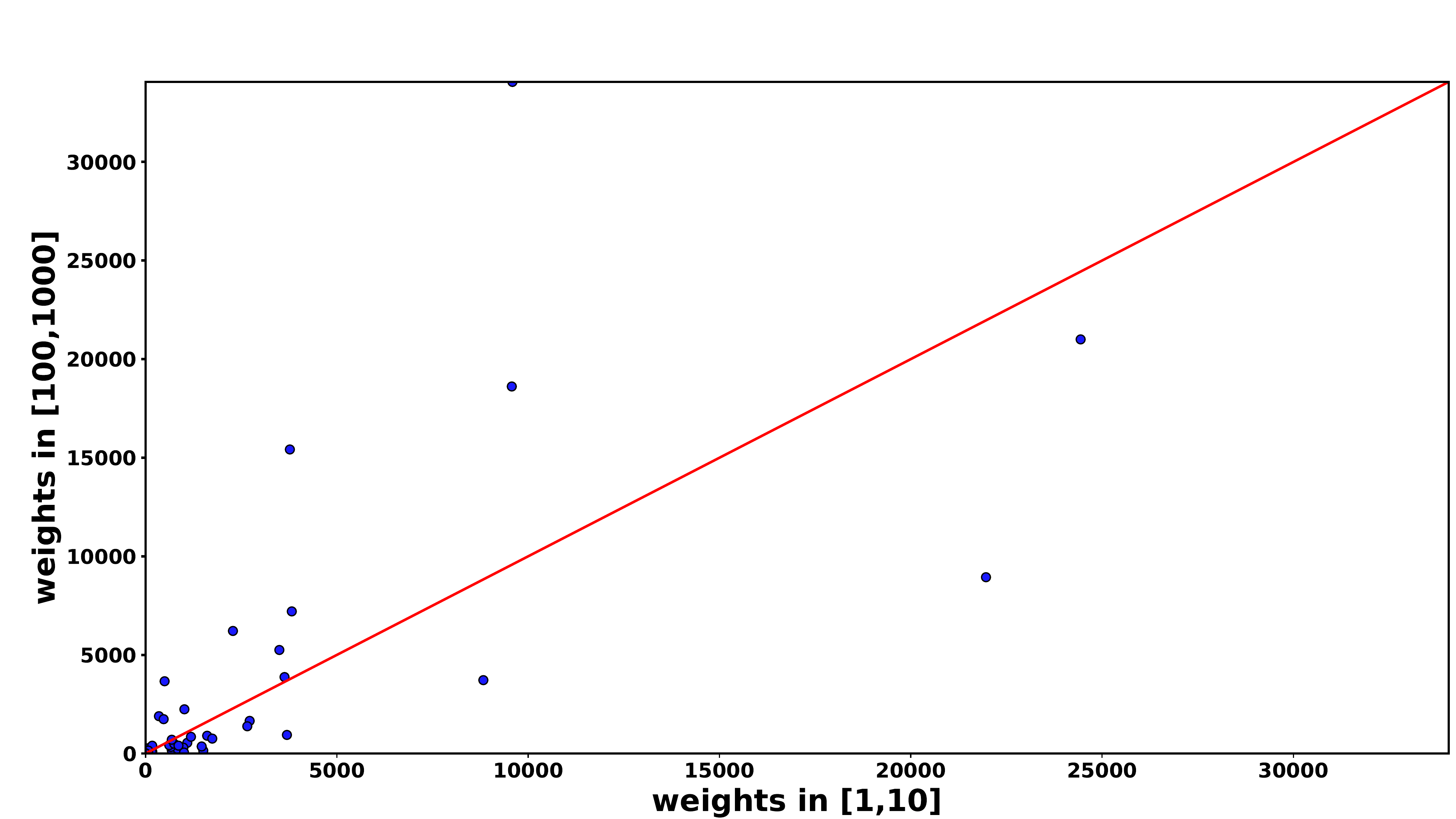}
	\caption{\bdd{2} without PB cuts.}
	\label{fig:compare_capbdd}
	\end{subfigure}%
\hfill
	\begin{subfigure}{0.5\textwidth}
	\includegraphics[width=\textwidth]{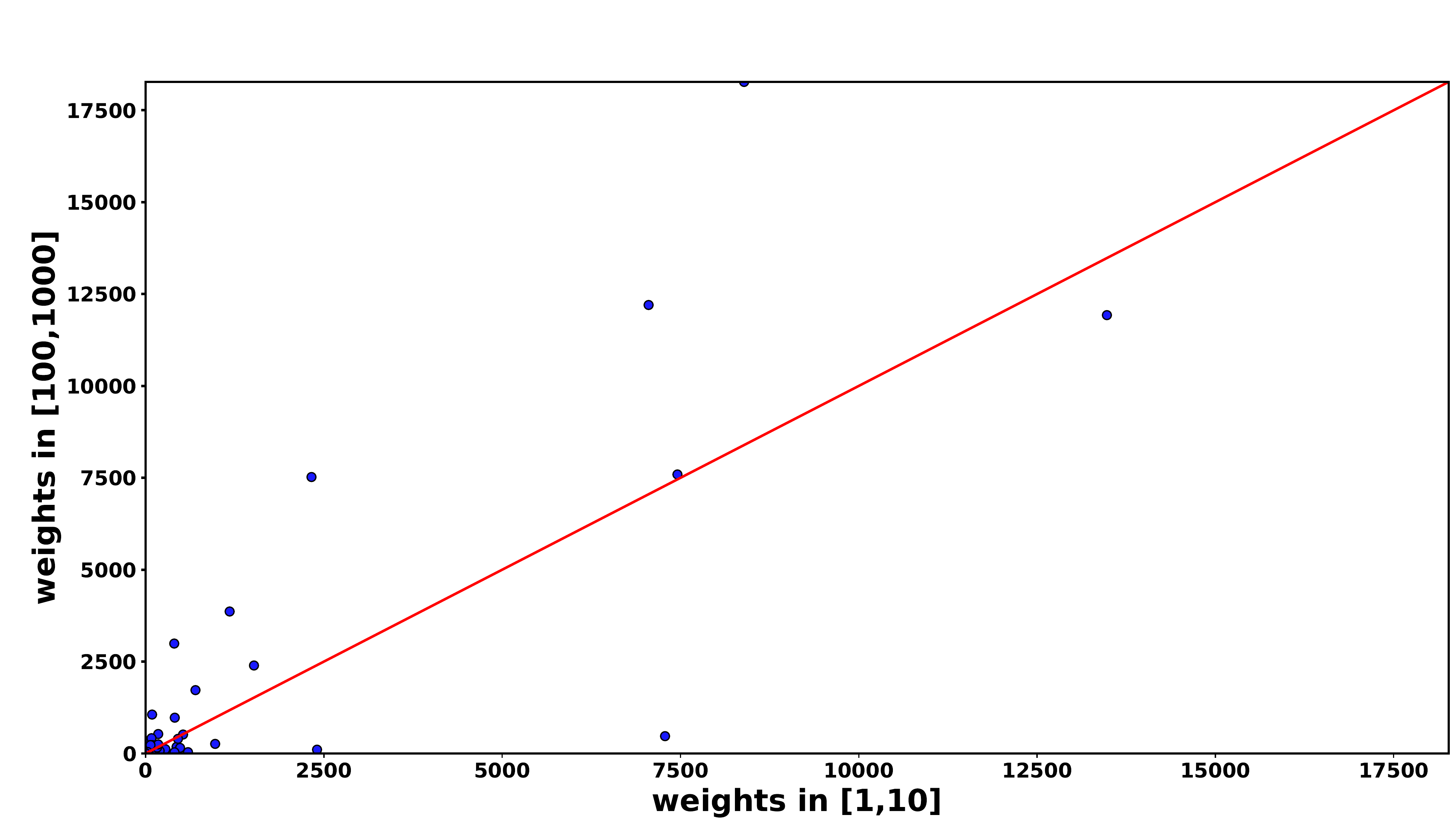}
	\caption{\bdd{2} with PB cuts.}
	\label{fig:compare_capbdd_pb}
		\end{subfigure}%
		\hfill
			\begin{subfigure}{0.5\textwidth}
			\includegraphics[width=\textwidth]{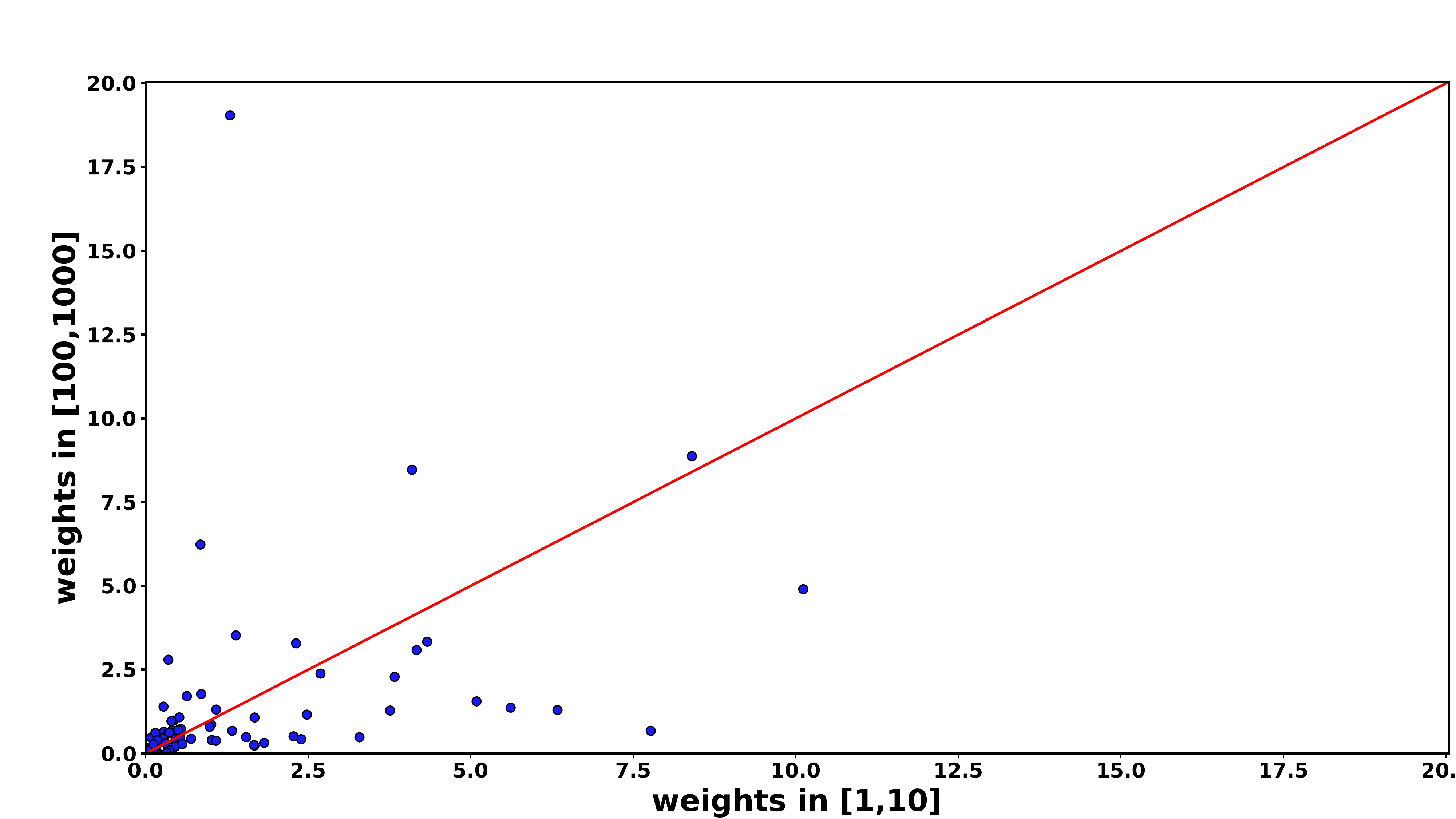}
			\caption{\bdd{3} without PB cuts.}
			\label{fig:compare_bdd}
		\end{subfigure}%
		\hfill
		\begin{subfigure}{0.5\textwidth}
			\includegraphics[width=\textwidth]{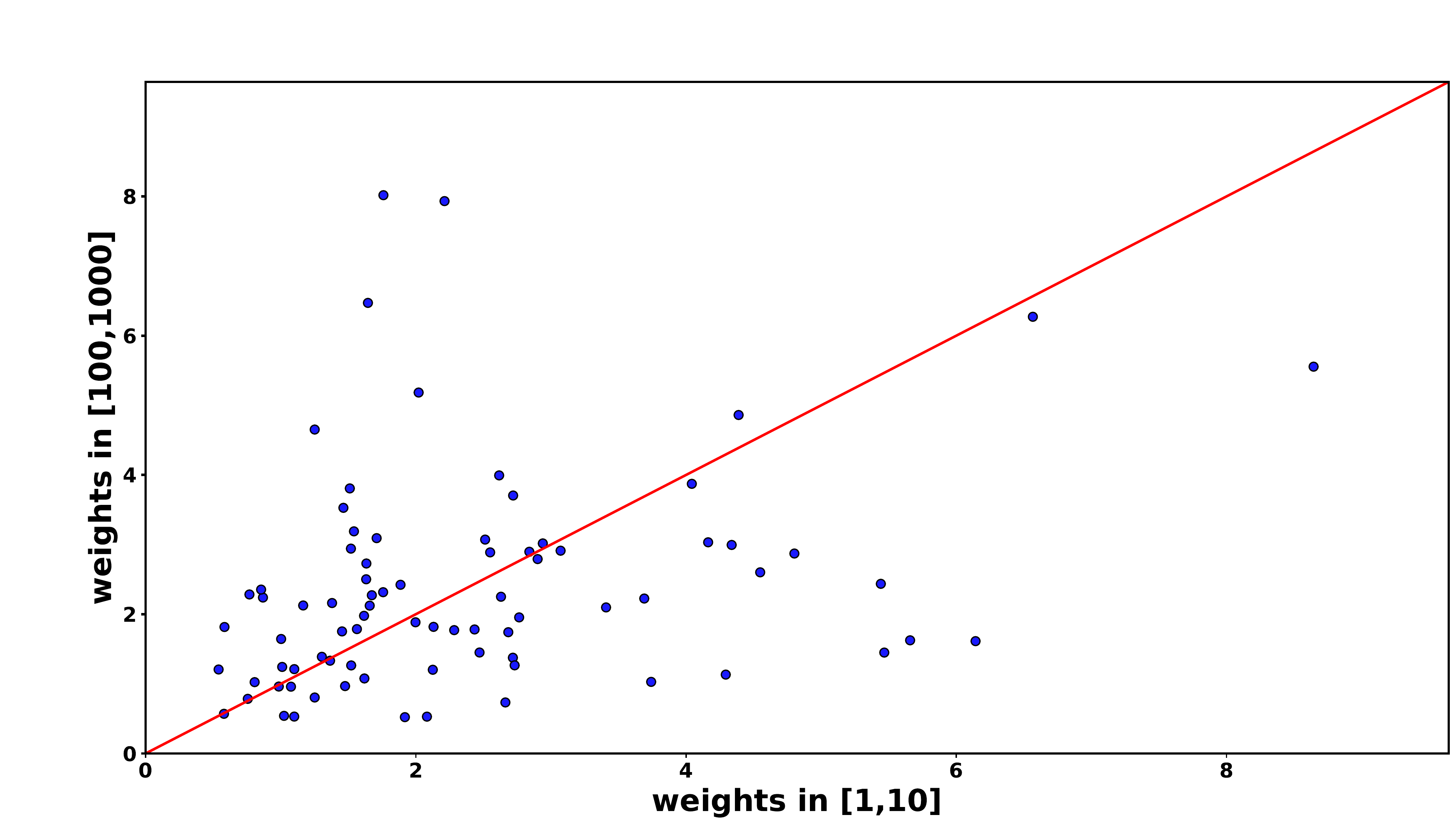}
			\caption{\bdd{3} with PB cuts.}
			\label{fig:compare_bdd_pb}
		\end{subfigure}%
			\caption{Comparison of solution times for instances with weights drawn from $[1,10]$  and from $[100,1000]$.  The red line shows where the computation time is equal.}
		\label{fig:compare}
\end{figure}%

\section{BDD Size Analysis} \label{app:bddsize}

During the construction of the $850$-scenario BDDs for each instance, we record statistics about BDD structure. Figures \ref{fig:smwds_bdd0arcsize} and \ref{fig:smwds_bdd1arcsize} show that the mean number of arcs in \bdd{2} is always more than the mean number of arcs in \bdd{3}, and as presented in Section 7.2, \bdd{2} has more mean nodes than \bdd{3}. As such, we conclude that the \bdd{3} method always creates smaller BDDs than the \bdd{2} method. For \bdd{2}, we always have the same mean number of  one-arcs and zero-arcs, this is because we have no constraints $\mathcal{Y}(\omega)$ for this application and so any infeasibility can be handled via the capacitated arcs. This means each node has one of each arc type, (none are removed due to transitioning to an infeasible state). Finally, we remark that approximately 18\% of zero-arcs are capacitated in \bdd{2}. We report detailed statistics in Tables \ref{table:30n},  \ref{table:35n},  \ref{table:40n},  \ref{table:45n},  and \ref{table:50n}.

\begin{figure}[h]\centering
	\includegraphics[width=0.6\textwidth]{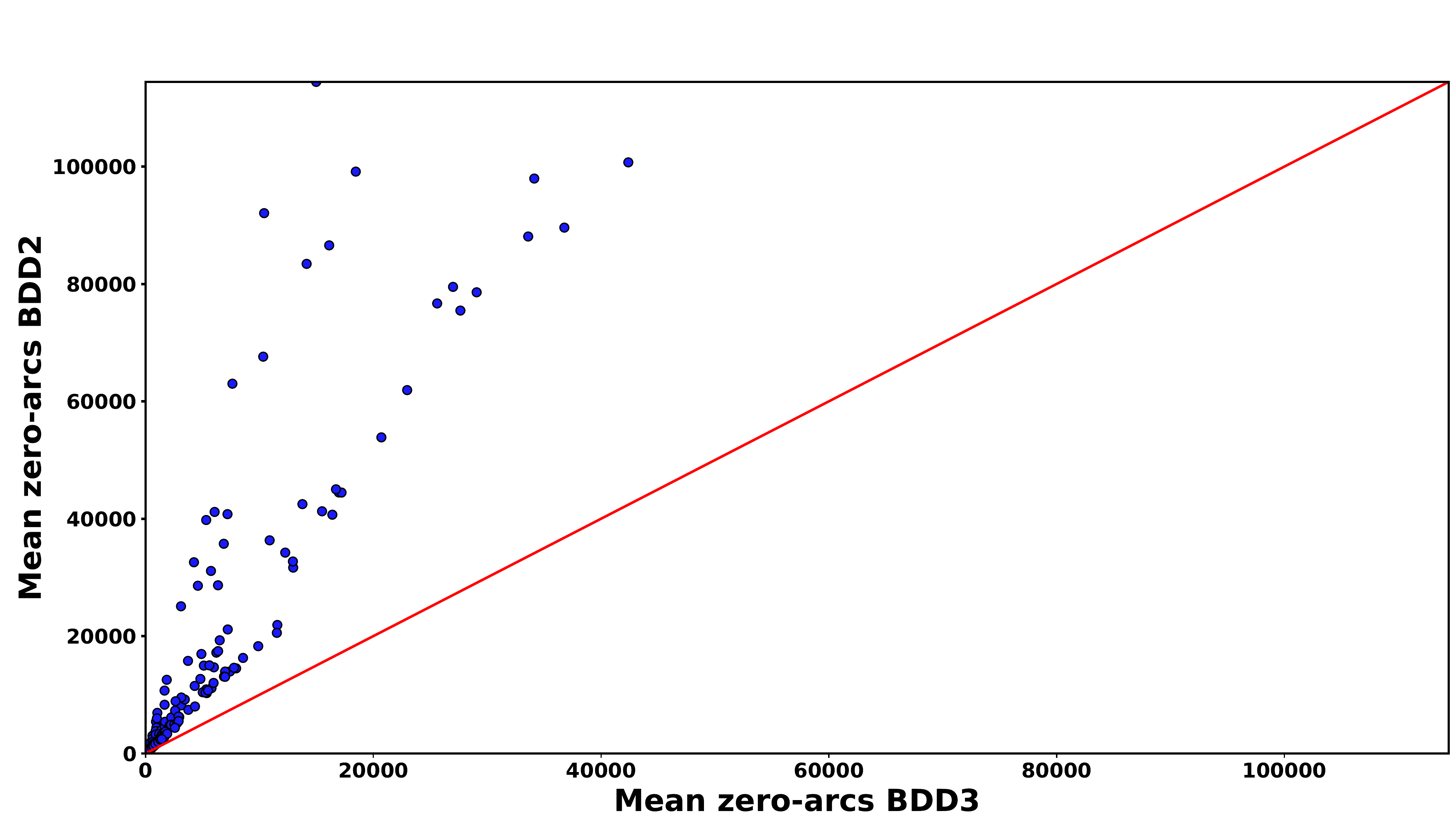}
	\caption{Comparison of mean number of zero-arcs in the BDDs for each instance, the red line shows where the mean number of arcs is equal. Instances where \bdd{2} hit the memory limit have been omitted.}
	\label{fig:smwds_bdd0arcsize}
\end{figure}%

\begin{figure}[h]\centering
	\includegraphics[width=0.6\textwidth]{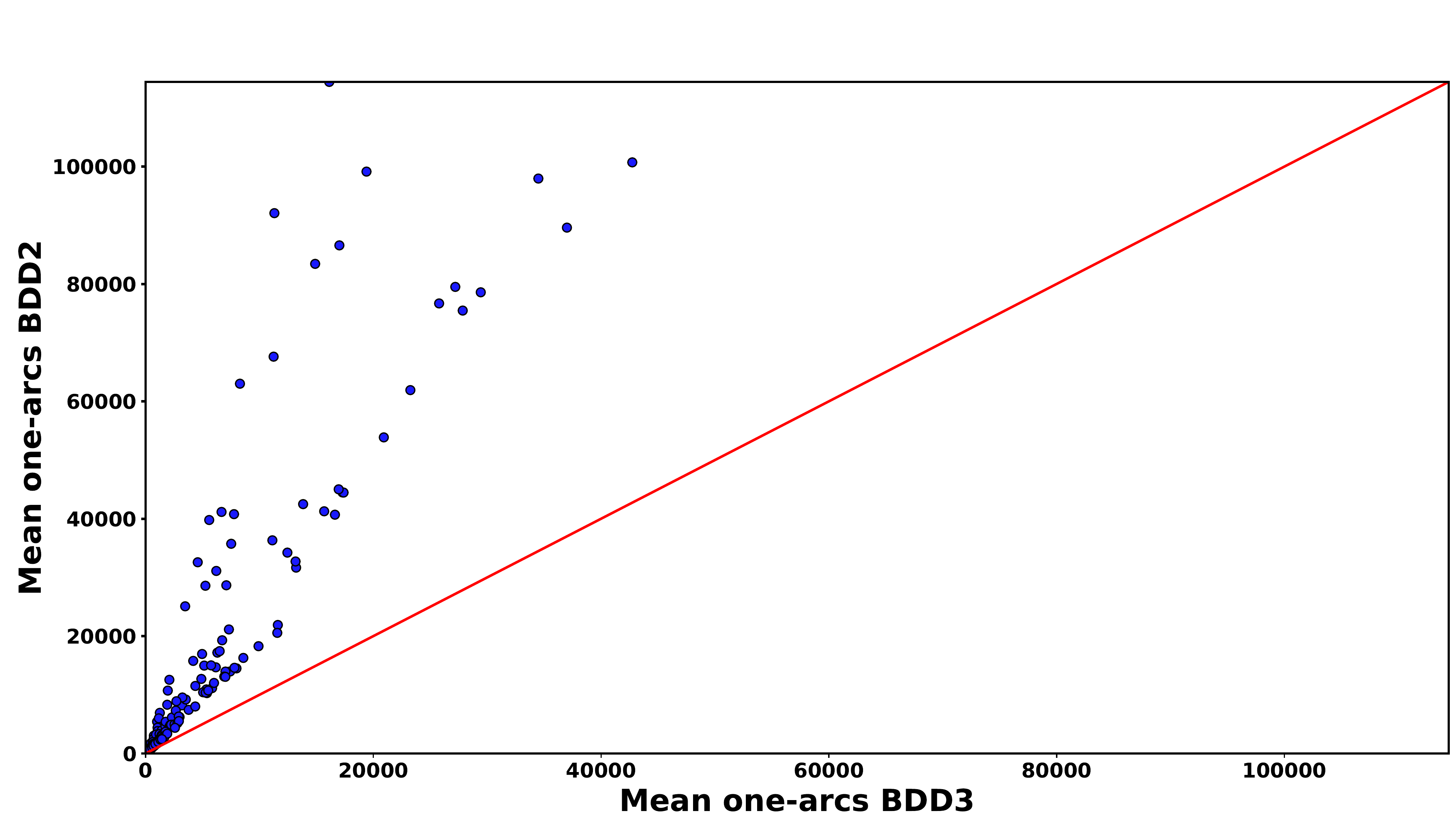}
	\caption{Comparison of mean number of one-arcs in the BDDs for each instance, the red line shows where the mean number of arcs is equal. Instances where \bdd{2} hit the memory limit have been omitted.}
	\label{fig:smwds_bdd1arcsize}
\end{figure}%

\newpage
 \newgeometry{left=2.5cm,bottom=0.1cm}
\begin{longtable}{|r|r|r|r|r|r|r|r|r|r|}
		\caption{BDD size results for $|\vertexset| = 30$. \label{table:30n}}\\
			\hline
					\multicolumn{2}{|l|}{~}  & \multicolumn{4}{c|}{\bdd{2}} & \multicolumn{3}{ c|}{\bdd{3}}\\
					\hline
edge 		& seed &  mean &   mean	&mean &  mean \# &    mean &   mean &mean  \\ 
	 density &  			& $|\nodes^\omega|$	&   $| \onearcs|$ 	&  $| \zeroarcs|$ &  cap.\  arcs&  $|\nodes^\omega|$ 			&   $| \onearcs|$ 	&  $| \zeroarcs|$  \\ 
  			&  				&  		per BDD  				& per BDD  				&  per BDD &  		per BDD  				& per BDD  				&  per BDD &  per BDD\\ \hline
0.2 & 0 & 3018.86 & 3017.86 & 3017.86 & 1114.21 & 724.30 & 723.30 & 602.16 \\ 
~ & 1 & 5420.35 & 5419.35 & 5419.35 & 1962.76 & 1009.27 & 1008.27 & 915.36 \\ 
~ & 2 & 6924.34 & 6923.34 & 6923.34 & 2689.62 & 1250.99 & 1249.99 & 1031.36 \\ 
~ & 3 & 8306.26 & 8305.26 & 8305.26 & 3233.21 & 1899.32 & 1898.32 & 1666.14 \\ 
~ & 4 & 4402.26 & 4401.26 & 4401.26 & 1394.33 & 1051.16 & 1050.16 & 944.08 \\ 
~ & 5 & 10715.90 & 10714.90 & 10714.90 & 4265.93 & 1950.00 & 1949.00 & 1663.82 \\ 
~ & 6 & 5992.64 & 5991.64 & 5991.64 & 2113.53 & 1163.27 & 1162.27 & 992.93 \\ 
~ & 7 & 1737.89 & 1736.89 & 1736.89 & 735.48 & 467.33 & 466.33 & 376.66 \\ 
~ & 8 & 3826.97 & 3825.97 & 3825.97 & 1427.94 & 1062.61 & 1061.61 & 881.27 \\ 
~ & 9 & 2520.30 & 2519.30 & 2519.30 & 940.76 & 710.76 & 709.76 & 602.68 \\ \hline
0.4 & 0 & 2769.06 & 2768.06 & 2768.06 & 668.72 & 1006.10 & 1005.10 & 971.61 \\ 
~ & 1 & 3816.87 & 3815.87 & 3815.87 & 936.23 & 1498.83 & 1497.83 & 1459.81 \\ 
~ & 2 & 3931.80 & 3930.80 & 3930.80 & 990.46 & 1415.15 & 1414.15 & 1369.65 \\ 
~ & 3 & 3233.86 & 3232.86 & 3232.86 & 792.84 & 921.04 & 920.04 & 887.57 \\ 
~ & 4 & 4609.23 & 4608.23 & 4608.23 & 1147.37 & 1739.53 & 1738.53 & 1676.26 \\ 
~ & 5 & 1961.03 & 1960.03 & 1960.03 & 544.44 & 644.06 & 643.06 & 613.76 \\ 
~ & 6 & 5400.92 & 5399.92 & 5399.92 & 1341.70 & 1761.20 & 1760.20 & 1710.92 \\ 
~ & 7 & 8200.66 & 8199.66 & 8199.66 & 1842.12 & 3207.28 & 3206.28 & 3131.03 \\ 
~ & 8 & 3312.92 & 3311.92 & 3311.92 & 858.91 & 1220.99 & 1219.99 & 1185.24 \\ 
~ & 9 & 6609.91 & 6608.91 & 6608.91 & 1590.08 & 2621.95 & 2620.95 & 2564.96 \\ \hline
0.6 & 0 & 1294.16 & 1293.16 & 1293.16 & 245.84 & 567.34 & 566.34 & 550.86 \\ 
~ & 1 & 2068.50 & 2067.50 & 2067.50 & 327.67 & 1022.37 & 1021.37 & 1009.02 \\ 
~ & 2 & 1690.09 & 1689.09 & 1689.09 & 285.69 & 720.96 & 719.96 & 710.30 \\ 
~ & 3 & 1674.31 & 1673.31 & 1673.31 & 272.27 & 723.34 & 722.34 & 714.43 \\ 
~ & 4 & 1277.53 & 1276.53 & 1276.53 & 227.00 & 564.40 & 563.40 & 555.51 \\ 
~ & 5 & 1256.47 & 1255.47 & 1255.47 & 242.12 & 684.39 & 683.39 & 671.27 \\ 
~ & 6 & 1373.36 & 1372.36 & 1372.36 & 265.19 & 643.71 & 642.71 & 630.44 \\ 
~ & 7 & 1259.99 & 1258.99 & 1258.99 & 232.17 & 615.56 & 614.56 & 604.10 \\ 
~ & 8 & 3136.36 & 3135.36 & 3135.36 & 477.88 & 1409.16 & 1408.16 & 1395.90 \\ 
~ & 9 & 2392.07 & 2391.07 & 2391.07 & 386.25 & 1208.65 & 1207.65 & 1191.20 \\ \hline
0.8 & 0 & 443.08 & 442.08 & 442.08 & 53.67 & 224.96 & 223.96 & 220.23 \\ 
~ & 1 & 597.07 & 596.07 & 596.07 & 68.78 & 330.32 & 329.32 & 324.46 \\ 
~ & 2 & 466.08 & 465.08 & 465.08 & 55.13 & 234.86 & 233.86 & 232.28 \\ 
~ & 3 & 418.49 & 417.49 & 417.49 & 53.89 & 217.67 & 216.67 & 212.44 \\ 
~ & 4 & 464.73 & 463.73 & 463.73 & 61.59 & 251.38 & 250.38 & 246.08 \\ 
~ & 5 & 677.15 & 676.15 & 676.15 & 81.13 & 377.80 & 376.80 & 371.61 \\ 
~ & 6 & 469.29 & 468.29 & 468.29 & 60.19 & 260.03 & 259.03 & 254.11 \\ 
~ & 7 & 462.66 & 461.66 & 461.66 & 65.05 & 260.99 & 259.99 & 253.24 \\ 
~ & 8 & 540.88 & 539.88 & 539.88 & 68.64 & 295.15 & 294.15 & 289.23 \\ 
~ & 9 & 507.49 & 506.49 & 506.49 & 60.49 & 272.28 & 271.28 & 268.15 \\ \hline
\end{longtable}

\newpage
\begin{longtable}{|r|r|r|r|r|r|r|r|r|r|}
	\caption{BDD size results for $|\vertexset| = 35$. \label{table:35n}}\\
	\hline
					\multicolumn{2}{|l|}{~}  & \multicolumn{4}{c|}{\bdd{2}} & \multicolumn{3}{ c|}{\bdd{3}}\\
	\hline
	edge 		& seed &  mean &   mean	&mean &  mean \# &    mean &   mean &mean  \\ 
	density &  			& $|\nodes^\omega|$	&   $| \onearcs|$ 	&  $| \zeroarcs|$ &  cap.\  arcs&  $|\nodes^\omega|$ 			&   $| \onearcs|$ 	&  $| \zeroarcs|$   \\ 
	&  				&  		per BDD  				& per BDD  				&  per BDD &  		per BDD  				& per BDD  				&  per BDD &  per BDD\\ \hline

 0.2 & 0 & 32580.90 & 32579.90 & 32579.90 & 11081.80 & 4578.36 & 4577.36 & 4246.18 \\ \hline
~ & 1 & 41154.10 & 41153.10 & 41153.10 & 13437.50 & 6668.76 & 6667.76 & 6060.34 \\ 
~ & 2 & 12553.00 & 12552.00 & 12552.00 & 4485.07 & 2087.75 & 2086.75 & 1856.42 \\ 
~ & 3 & 63004.10 & 63003.10 & 63003.10 & 19502.20 & 8286.36 & 8285.36 & 7623.79 \\ 
~ & 4 & 40779.00 & 40778.00 & 40778.00 & 13589.90 & 7764.78 & 7763.78 & 7191.78 \\ 
~ & 5 & 25077.70 & 25076.70 & 25076.70 & 9576.54 & 3475.27 & 3474.27 & 3108.48 \\ 
~ & 6 & 31101.90 & 31100.90 & 31100.90 & 9417.06 & 6207.69 & 6206.69 & 5739.17 \\ 
~ & 7 & 15775.40 & 15774.40 & 15774.40 & 5433.17 & 4182.02 & 4181.02 & 3723.79 \\ 
~ & 8 & 39781.10 & 39780.10 & 39780.10 & 12794.90 & 5586.29 & 5585.29 & 5322.22 \\ 
~ & 9 & 28580.80 & 28579.80 & 28579.80 & 9080.37 & 5252.05 & 5251.05 & 4591.24 \\ \hline
0.4 & 0 & 19288.90 & 19287.90 & 19287.90 & 4260.05 & 6723.02 & 6722.02 & 6508.76 \\ 
~ & 1 & 6137.60 & 6136.60 & 6136.60 & 1278.35 & 2310.96 & 2309.96 & 2250.50 \\ 
~ & 2 & 9192.87 & 9191.87 & 9191.87 & 2006.57 & 3530.45 & 3529.45 & 3443.09 \\ 
~ & 3 & 7309.37 & 7308.37 & 7308.37 & 1697.80 & 2649.26 & 2648.26 & 2591.55 \\ 
~ & 4 & 17184.40 & 17183.40 & 17183.40 & 3645.04 & 6301.42 & 6300.42 & 6212.91 \\ 
~ & 5 & 14681.00 & 14680.00 & 14680.00 & 3224.40 & 6165.96 & 6164.96 & 5995.87 \\ 
~ & 6 & 12702.60 & 12701.60 & 12701.60 & 2819.79 & 4895.53 & 4894.53 & 4810.14 \\ 
~ & 7 & 9539.16 & 9538.16 & 9538.16 & 2357.20 & 3233.98 & 3232.98 & 3137.95 \\ 
~ & 8 & 8903.68 & 8902.68 & 8902.68 & 2236.32 & 2705.45 & 2704.45 & 2631.44 \\ 
~ & 9 & 11509.30 & 11508.30 & 11508.30 & 2399.82 & 4368.44 & 4367.44 & 4311.02 \\ \hline
0.6 & 0 & 4560.06 & 4559.06 & 4559.06 & 741.14 & 2092.15 & 2091.15 & 2068.77 \\ 
~ & 1 & 1956.92 & 1955.92 & 1955.92 & 332.31 & 998.92 & 997.92 & 979.97 \\ 
~ & 2 & 2840.13 & 2839.13 & 2839.13 & 462.63 & 1465.41 & 1464.41 & 1443.62 \\ 
~ & 3 & 2811.50 & 2810.50 & 2810.50 & 456.18 & 1380.67 & 1379.67 & 1359.94 \\ 
~ & 4 & 2773.21 & 2772.21 & 2772.21 & 437.50 & 1336.89 & 1335.89 & 1316.86 \\ 
~ & 5 & 3091.57 & 3090.57 & 3090.57 & 483.19 & 1469.51 & 1468.51 & 1457.98 \\ 
~ & 6 & 1593.04 & 1592.04 & 1592.04 & 262.48 & 782.79 & 781.79 & 771.85 \\ 
~ & 7 & 2518.44 & 2517.44 & 2517.44 & 410.63 & 1245.07 & 1244.07 & 1228.23 \\ 
~ & 8 & 2955.18 & 2954.18 & 2954.18 & 495.02 & 1429.02 & 1428.02 & 1414.85 \\ 
~ & 9 & 2163.63 & 2162.63 & 2162.63 & 365.13 & 1137.66 & 1136.66 & 1121.04 \\ \hline
0.8 & 0 & 758.41 & 757.41 & 757.41 & 86.49 & 441.41 & 440.41 & 435.52 \\ 
~ & 1 & 882.72 & 881.72 & 881.72 & 95.58 & 466.12 & 465.12 & 459.88 \\ 
~ & 2 & 687.38 & 686.38 & 686.38 & 84.28 & 358.86 & 357.86 & 352.73 \\ 
~ & 3 & 598.93 & 597.93 & 597.93 & 69.30 & 333.90 & 332.90 & 328.87 \\ 
~ & 4 & 907.33 & 906.33 & 906.33 & 97.77 & 461.97 & 460.97 & 457.37 \\ 
~ & 5 & 919.44 & 918.44 & 918.44 & 96.50 & 547.12 & 546.12 & 540.55 \\ 
~ & 6 & 1082.32 & 1081.32 & 1081.32 & 117.59 & 591.52 & 590.52 & 586.18 \\ 
~ & 7 & 799.60 & 798.60 & 798.60 & 102.57 & 465.37 & 464.37 & 456.96 \\ 
~ & 8 & 1113.26 & 1112.26 & 1112.26 & 119.18 & 594.97 & 593.97 & 588.88 \\ 
~ & 9 & 657.24 & 656.24 & 656.24 & 76.21 & 381.44 & 380.44 & 376.91 \\ \hline
\end{longtable}

\newpage
\begin{longtable}{|r|r|r|r|r|r|r|r|r|r|}
	\caption{BDD size results for $|\vertexset| = 40$, `-' indicates the memory limit was hit. \label{table:40n}}\\
	\hline
					\multicolumn{2}{|l|}{~}  & \multicolumn{4}{c|}{\bdd{2}} & \multicolumn{3}{ c|}{\bdd{3}}\\
	\hline
	edge 		& seed &  mean &   mean	&mean &  mean \# &    mean &   mean &mean  \\ 
	density &  			& $|\nodes^\omega|$	&   $| \onearcs|$ 	&  $| \zeroarcs|$ &  cap.\ arcs&  $|\nodes^\omega|$ 			&   $| \onearcs|$ 	&  $| \zeroarcs|$   \\ 
	&  				&  		per BDD  				& per BDD  				&  per BDD &  		per BDD  				& per BDD  				&  per BDD &  per BDD\\ \hline
 0.2 & 0 & 35731.10 & 35730.10 & 35730.10 & 11506.30 & 7516.93 & 7515.93 & 6866.43 \\ 
~ & 1 & - & - & - & - & 23708.40 & 23707.40 & 22633.60 \\ 
~ & 2 & 83423.20 & 83422.20 & 83422.20 & 25173.30 & 14889.00 & 14888.00 & 14139.80 \\ 
~ & 3 & 67611.20 & 67610.20 & 67610.20 & 21236.00 & 11242.70 & 11241.70 & 10326.80 \\ 
~ & 4 & - & - & - & - & 26396.10 & 26395.10 & 24908.60 \\ 
~ & 5 & 86574.40 & 86573.40 & 86573.40 & 25498.60 & 17015.70 & 17014.70 & 16121.20 \\ 
~ & 6 & 92062.80 & 92061.80 & 92061.80 & 28073.70 & 11309.40 & 11308.40 & 10408.20 \\ 
~ & 7 & 114425.00 & 114424.00 & 114424.00 & 37122.40 & 16127.00 & 16126.00 & 14981.80 \\ 
~ & 8 & 28658.40 & 28657.40 & 28657.40 & 9826.14 & 7088.74 & 7087.74 & 6358.44 \\ 
~ & 9 & 99140.90 & 99139.90 & 99139.90 & 30424.50 & 19397.80 & 19396.80 & 18455.60 \\ \hline
0.4 & 0 & 14973.90 & 14972.90 & 14972.90 & 2797.87 & 5147.83 & 5146.83 & 5114.27 \\ 
~ & 1 & 53858.30 & 53857.30 & 53857.30 & 9643.59 & 20918.00 & 20917.00 & 20704.20 \\ 
~ & 2 & 41264.30 & 41263.30 & 41263.30 & 7888.16 & 15681.20 & 15680.20 & 15493.40 \\ 
~ & 3 & 36318.40 & 36317.40 & 36317.40 & 7787.66 & 11132.80 & 11131.80 & 10892.60 \\ 
~ & 4 & 17441.30 & 17440.30 & 17440.30 & 3861.26 & 6505.10 & 6504.10 & 6368.71 \\ 
~ & 5 & 31666.90 & 31665.90 & 31665.90 & 6486.48 & 13217.60 & 13216.60 & 12959.00 \\ 
~ & 6 & 16951.30 & 16950.30 & 16950.30 & 3516.88 & 4962.26 & 4961.26 & 4899.94 \\ 
~ & 7 & 21132.50 & 21131.50 & 21131.50 & 4571.27 & 7316.75 & 7315.75 & 7214.10 \\ 
~ & 8 & 42486.00 & 42485.00 & 42485.00 & 7746.49 & 13832.50 & 13831.50 & 13770.20 \\ 
~ & 9 & 15013.10 & 15012.10 & 15012.10 & 3481.65 & 5752.12 & 5751.12 & 5601.79 \\ \hline
0.6 & 0 & 6219.20 & 6218.20 & 6218.20 & 921.96 & 2984.65 & 2983.65 & 2955.24 \\ 
~ & 1 & 7438.49 & 7437.49 & 7437.49 & 1083.39 & 3781.67 & 3780.67 & 3755.73 \\ 
~ & 2 & 6110.10 & 6109.10 & 6109.10 & 946.17 & 2940.47 & 2939.47 & 2918.22 \\ 
~ & 3 & 3570.38 & 3569.38 & 3569.38 & 547.99 & 1857.76 & 1856.76 & 1838.03 \\ 
~ & 4 & 10913.90 & 10912.90 & 10912.90 & 1452.61 & 5337.63 & 5336.63 & 5319.85 \\ 
~ & 5 & 4847.15 & 4846.15 & 4846.15 & 709.71 & 2166.37 & 2165.37 & 2145.42 \\ 
~ & 6 & 4901.15 & 4900.15 & 4900.15 & 746.00 & 2245.30 & 2244.30 & 2227.91 \\ 
~ & 7 & 4962.42 & 4961.42 & 4961.42 & 749.74 & 2538.77 & 2537.77 & 2518.23 \\ 
~ & 8 & 6307.69 & 6306.69 & 6306.69 & 900.55 & 2910.85 & 2909.85 & 2894.84 \\ 
~ & 9 & 3704.52 & 3703.52 & 3703.52 & 515.74 & 1746.03 & 1745.03 & 1731.43 \\ \hline
0.8 & 0 & 1399.40 & 1398.40 & 1398.40 & 147.34 & 781.18 & 780.18 & 773.72 \\ 
~ & 1 & 1158.79 & 1157.79 & 1157.79 & 130.54 & 672.65 & 671.65 & 665.58 \\ 
~ & 2 & 1390.09 & 1389.09 & 1389.09 & 144.92 & 760.24 & 759.24 & 752.67 \\ 
~ & 3 & 1389.55 & 1388.55 & 1388.55 & 150.83 & 805.48 & 804.48 & 800.21 \\ 
~ & 4 & 1658.74 & 1657.74 & 1657.74 & 151.32 & 914.55 & 913.55 & 910.19 \\ 
~ & 5 & 1287.04 & 1286.04 & 1286.04 & 131.01 & 742.35 & 741.35 & 735.80 \\ 
~ & 6 & 1739.48 & 1738.48 & 1738.48 & 166.21 & 911.13 & 910.13 & 905.46 \\ 
~ & 7 & 996.49 & 995.49 & 995.49 & 116.56 & 577.86 & 576.86 & 570.90 \\ 
~ & 8 & 2343.45 & 2342.45 & 2342.45 & 199.35 & 1256.32 & 1255.32 & 1250.96 \\ 
~ & 9 & 1265.89 & 1264.89 & 1264.89 & 134.14 & 703.43 & 702.43 & 697.16 \\ \hline
\end{longtable}

\newpage
\begin{longtable}{|r|r|r|r|r|r|r|r|r|r|}
	\caption{BDD size results for $|\vertexset| = 45$, `-' indicates the memory limit was hit. \label{table:45n}}\\
	\hline
					\multicolumn{2}{|l|}{~}  & \multicolumn{4}{c|}{\bdd{2}} & \multicolumn{3}{ c|}{\bdd{3}}\\
	\hline
	edge 		& seed &  mean &   mean	&mean &  mean \# &    mean &   mean &mean  \\ 
	density &  			& $|\nodes^\omega|$	&   $| \onearcs|$ 	&  $| \zeroarcs|$ &  cap.\ arcs&  $|\nodes^\omega|$ 			&   $| \onearcs|$ 	&  $| \zeroarcs|$  \\ 
	&  				&  		per BDD  				& per BDD  				&  per BDD &  		per BDD  				& per BDD  				&  per BDD &  per BDD\\ \hline
	 \hline
	0.2 & 0 & - & - & - & - & 62105.50 & 62104.50 & 58704.50 \\ 
	~ & 1 & - & - & - & - & 102714.00 & 102713.00 & 98640.60 \\ 
	~ & 2 & - & - & - & - & 22311.60 & 22310.60 & 20322.30 \\ 
	~ & 3 & - & - & - & - & 51947.30 & 51946.30 & 49120.50 \\ 
	~ & 4 & - & - & - & - & 40946.40 & 40945.40 & 37888.90 \\ 
	~ & 5 & - & - & - & - & 29582.10 & 29581.10 & 27900.30 \\ 
	~ & 6 & - & - & - & - & 52297.50 & 52296.50 & 49495.40 \\ 
	~ & 7 & - & - & - & - & 51002.40 & 51001.40 & 47613.20 \\ 
	~ & 8 & - & - & - & - & 130971.00 & 130970.00 & 125546.00 \\ 
	~ & 9 & - & - & - & - & 40342.50 & 40341.50 & 38585.20 \\ \hline
	0.4 & 0 & 75463.80 & 75462.80 & 75462.80 & 13958.70 & 27846.30 & 27845.30 & 27641.10 \\ 
	~ & 1 & 34224.10 & 34223.10 & 34223.10 & 6719.66 & 12452.20 & 12451.20 & 12260.60 \\ 
	~ & 2 & 40684.30 & 40683.30 & 40683.30 & 7995.72 & 16626.70 & 16625.70 & 16398.00 \\ 
	~ & 3 & 32715.50 & 32714.50 & 32714.50 & 6355.53 & 13169.10 & 13168.10 & 12934.10 \\ 
	~ & 4 & 44474.70 & 44473.70 & 44473.70 & 9005.59 & 17292.90 & 17291.90 & 16970.40 \\ 
	~ & 5 & 44455.50 & 44454.50 & 44454.50 & 8472.07 & 17384.60 & 17383.60 & 17204.90 \\ 
	~ & 6 & 76680.00 & 76679.00 & 76679.00 & 14324.10 & 25773.90 & 25772.90 & 25601.90 \\ 
	~ & 7 & 97957.60 & 97956.60 & 97956.60 & 17410.00 & 34488.70 & 34487.70 & 34127.40 \\ 
	~ & 8 & 45009.90 & 45008.90 & 45008.90 & 8647.36 & 16955.80 & 16954.80 & 16718.10 \\ 
	~ & 9 & 61911.80 & 61910.80 & 61910.80 & 11393.90 & 23249.80 & 23248.80 & 22966.60 \\ \hline
	0.6 & 0 & 5191.71 & 5190.71 & 5190.71 & 801.60 & 2801.46 & 2800.46 & 2770.37 \\ 
	~ & 1 & 10430.70 & 10429.70 & 10429.70 & 1356.74 & 5052.33 & 5051.33 & 5019.59 \\ 
	~ & 2 & 8008.23 & 8007.23 & 8007.23 & 1109.41 & 4357.89 & 4356.89 & 4336.15 \\ 
	~ & 3 & 10285.70 & 10284.70 & 10284.70 & 1349.79 & 5387.60 & 5386.60 & 5364.74 \\ 
	~ & 4 & 10369.90 & 10368.90 & 10368.90 & 1352.94 & 5278.57 & 5277.57 & 5263.40 \\ 
	~ & 5 & 11135.90 & 11134.90 & 11134.90 & 1568.53 & 5839.28 & 5838.28 & 5787.67 \\ 
	~ & 6 & 10733.90 & 10732.90 & 10732.90 & 1306.78 & 5502.45 & 5501.45 & 5483.80 \\ 
	~ & 7 & 5524.69 & 5523.69 & 5523.69 & 787.34 & 2924.02 & 2923.02 & 2899.98 \\ 
	~ & 8 & 12020.10 & 12019.10 & 12019.10 & 1568.62 & 6007.58 & 6006.58 & 5968.16 \\ 
	~ & 9 & 13904.40 & 13903.40 & 13903.40 & 1843.16 & 7164.71 & 7163.71 & 7148.26 \\ \hline
	0.8 & 0 & 1663.57 & 1662.57 & 1662.57 & 158.38 & 892.82 & 891.82 & 886.51 \\ 
	~ & 1 & 2509.13 & 2508.13 & 2508.13 & 218.19 & 1486.72 & 1485.72 & 1478.44 \\ 
	~ & 2 & 2025.95 & 2024.95 & 2024.95 & 184.70 & 1141.07 & 1140.07 & 1134.87 \\ 
	~ & 3 & 1700.56 & 1699.56 & 1699.56 & 173.70 & 1014.69 & 1013.69 & 1008.64 \\ 
	~ & 4 & 3077.11 & 3076.11 & 3076.11 & 273.40 & 1731.44 & 1730.44 & 1723.24 \\ 
	~ & 5 & 1906.07 & 1905.07 & 1905.07 & 172.47 & 1035.33 & 1034.33 & 1029.43 \\ 
	~ & 6 & 2029.38 & 2028.38 & 2028.38 & 174.05 & 1179.28 & 1178.28 & 1173.85 \\ 
	~ & 7 & 2774.18 & 2773.18 & 2773.18 & 257.67 & 1540.59 & 1539.59 & 1535.02 \\ 
	~ & 8 & 1606.21 & 1605.21 & 1605.21 & 160.39 & 861.76 & 860.76 & 855.43 \\ 
	~ & 9 & 2512.36 & 2511.36 & 2511.36 & 224.34 & 1413.61 & 1412.61 & 1408.48 \\ \hline
\end{longtable}

\newpage
\begin{longtable}{|r|r|r|r|r|r|r|r|r|r|}
	\caption{BDD size results for $|\vertexset| = 50$, `-' indicates the memory limit was hit. \label{table:50n}}\\
	\hline
					\multicolumn{2}{|l|}{~}  & \multicolumn{4}{c|}{\bdd{2}} & \multicolumn{3}{ c|}{\bdd{3}}\\
	\hline
	edge 		& seed &  mean &   mean	&mean &  mean \# &    mean &   mean &mean  \\ 
	density &  			& $|\nodes^\omega|$	&   $| \onearcs|$ 	&  $| \zeroarcs|$ &  cap.\ arcs&  $|\nodes^\omega|$ 			&   $| \onearcs|$ 	&  $| \zeroarcs|$  \\ 
	&  				&  		per BDD  				& per BDD  				&  per BDD &  		per BDD  				& per BDD  				&  per BDD &  per BDD\\ \hline
	 0.2 & 0 & - & - & - & - & 327442.00 & 327441.00 & 320174.00 \\ 
	~ & 1 & - & - & - & - & 170885.00 & 170884.00 & 162326.00 \\ 
	~ & 2 & - & - & - & - & 130126.00 & 130125.00 & 124547.00 \\ 
	~ & 3 & - & - & - & - & 136905.00 & 136904.00 & 129049.00 \\ 
	~ & 4 & - & - & - & - & 124376.00 & 124375.00 & 119414.00 \\ 
	~ & 5 & - & - & - & - & 69718.80 & 69717.80 & 66652.30 \\ 
	~ & 6 & - & - & - & - & 93547.50 & 93546.50 & 89920.60 \\ 
	~ & 7 & - & - & - & - & 68696.90 & 68695.90 & 64089.60 \\ 
	~ & 8 & - & - & - & - & 131671.00 & 131670.00 & 126477.00 \\ 
	~ & 9 & - & - & - & - & 115068.00 & 115067.00 & 110098.00 \\ \hline
	0.4 & 0 & - & - & - & - & 99474.90 & 99473.90 & 98615.80 \\ 
	~ & 1 & - & - & - & - & 46381.80 & 46380.80 & 45854.10 \\ 
	~ & 2 & 79499.30 & 79498.30 & 79498.30 & 15044.80 & 27194.60 & 27193.60 & 26995.90 \\ 
	~ & 3 & - & - & - & - & 52667.70 & 52666.70 & 52309.00 \\ 
	~ & 4 & - & - & - & - & 58799.00 & 58798.00 & 58509.80 \\ 
	~ & 5 & 88082.30 & 88081.30 & 88081.30 & 16200.00 & 33811.20 & 33810.20 & 33594.00 \\ 
	~ & 6 & 89597.50 & 89596.50 & 89596.50 & 15643.10 & 36997.80 & 36996.80 & 36770.00 \\ 
	~ & 7 & 100717.00 & 100716.00 & 100716.00 & 17943.30 & 42735.90 & 42734.90 & 42385.70 \\ 
	~ & 8 & 78578.30 & 78577.30 & 78577.30 & 14754.20 & 29430.90 & 29429.90 & 29070.10 \\ 
	~ & 9 & - & - & - & - & 66489.80 & 66488.80 & 66039.40 \\ \hline
	0.6 & 0 & 13972.40 & 13971.40 & 13971.40 & 1690.96 & 7425.82 & 7424.82 & 7410.40 \\ 
	~ & 1 & 14500.70 & 14499.70 & 14499.70 & 1705.05 & 7980.19 & 7979.19 & 7948.53 \\ 
	~ & 2 & 13955.60 & 13954.60 & 13954.60 & 1874.83 & 7011.65 & 7010.65 & 6980.21 \\ 
	~ & 3 & 18281.00 & 18280.00 & 18280.00 & 2288.24 & 9916.87 & 9915.87 & 9890.49 \\ 
	~ & 4 & 21894.70 & 21893.70 & 21893.70 & 2723.56 & 11612.20 & 11611.20 & 11568.40 \\ 
	~ & 5 & 14601.80 & 14600.80 & 14600.80 & 1843.30 & 7804.77 & 7803.77 & 7763.38 \\ 
	~ & 6 & 16282.30 & 16281.30 & 16281.30 & 2058.19 & 8583.95 & 8582.95 & 8556.23 \\ 
	~ & 7 & 13102.90 & 13101.90 & 13101.90 & 1716.00 & 6908.13 & 6907.13 & 6881.85 \\ 
	~ & 8 & 13059.40 & 13058.40 & 13058.40 & 1600.05 & 7002.98 & 7001.98 & 6979.62 \\ 
	~ & 9 & 20553.60 & 20552.60 & 20552.60 & 2364.84 & 11564.30 & 11563.30 & 11525.90 \\ \hline
	0.8 & 0 & 2754.49 & 2753.49 & 2753.49 & 210.63 & 1559.11 & 1558.11 & 1552.24 \\ 
	~ & 1 & 2740.98 & 2739.98 & 2739.98 & 221.27 & 1636.37 & 1635.37 & 1627.45 \\ 
	~ & 2 & 2093.24 & 2092.24 & 2092.24 & 194.85 & 1210.75 & 1209.75 & 1203.95 \\ 
	~ & 3 & 2343.25 & 2342.25 & 2342.25 & 180.12 & 1259.18 & 1258.18 & 1252.74 \\ 
	~ & 4 & 3374.97 & 3373.97 & 3373.97 & 259.02 & 1900.05 & 1899.05 & 1894.78 \\ 
	~ & 5 & 4373.31 & 4372.31 & 4372.31 & 311.40 & 2578.56 & 2577.56 & 2570.64 \\ 
	~ & 6 & 2068.93 & 2067.93 & 2067.93 & 177.21 & 1205.17 & 1204.17 & 1198.19 \\ 
	~ & 7 & 1958.27 & 1957.27 & 1957.27 & 174.57 & 1120.64 & 1119.64 & 1113.73 \\ 
	~ & 8 & 2361.90 & 2360.90 & 2360.90 & 200.51 & 1304.02 & 1303.02 & 1295.72 \\ 
	~ & 9 & 2441.54 & 2440.54 & 2440.54 & 210.59 & 1432.74 & 1431.74 & 1424.29 \\ \hline
\end{longtable}

\restoregeometry

}


\end{document}